\documentclass[amssymb,amsfonts,11pt,verbatim,righttag,refcheck]{amsart}

\usepackage{amssymb}

\usepackage{graphicx,mathrsfs}

\usepackage[pdftex]{hyperref}
\numberwithin{equation}{section}
\theoremstyle{plain}

\DeclareMathOperator{\supp}{supp} 
 
\numberwithin{equation}{section}

\newtheorem{thm}{Theorem}[section]

\newtheorem{pro}{Proposition}[section]
\newtheorem{cor}{Corollary}[section]

\newtheorem{rem}{Remark}[section]
\newtheorem{fact}{Properties}[section]

\def\R {{\mathbb R}}
\def\N {{\Bbb N}}

\begin{document}
\baselineskip 16pt
\title{Inverse problems in  multifractal analysis}

\author{Julien Barral}
\address{LAGA (UMR 7539), D\'epartement de Math\'ematiques, Institut Galil\'ee, Universit\'e Paris 13 et Paris-Sorbonne-Cit\'e, 99 avenue Jean-Baptiste Cl\'ement , 93430  Villetaneuse, France}
\email{barral@math.univ-paris13.fr}

\thanks {
The author is grateful to De-Jun Feng and Jacques Peyri\`ere for valuable comments.}

\date{}
\begin{abstract}
Multifractal formalism is designed to describe the distribution at small scales of the elements of $\mathcal M^+_c(\R^d)$, the set of positive, finite and compactly supported Borel measures on $\R^d$. It is valid for such a measure $\mu$ when its  Hausdorff spectrum is the upper semi-continuous  function given by the concave Legendre-Fenchel transform of the  free energy function $\tau_\mu$ associated with $\mu$; this is the case for fundamental classes of exact dimensional measures. 

For any  function $\tau$  candidate to be the free energy function of some $\mu\in \mathcal M^+_c(\R^d)$, we build  such a measure, exact dimensional, and obeying the multifractal formalism. This result is extended to a refined formalism considering  jointly   Hausdorff and packing spectra. Also,  for any upper semi-continuous function candidate to be the lower Hausdorff spectrum of some exact dimensional $\mu\in\mathcal M^+_c(\R^d)$, we build such a measure.

Our results transfer to the analoguous inverse problems in multifractal analysis   of H\"older continuous functions.

 \end{abstract}
\maketitle

\tableofcontents

\section{Introduction and main statements}\label{introduc}
\subsection{Inverse problems in multifractal analysis of measures} Let $\mathcal M^+_c(\R^d)$ stand for the set of compactly supported  Borel  positive and finite measures on $\mathbb{R}^d$ ($d\ge 1)$, and for $\mu\in \mathcal M^+_c(\R^d)$ denote by $\supp(\mu)$ the topological support of $\mu$ (i.e. the compact set obtained as the complement of those points $x$ for which $\mu(B(x,r))=0$ for some $r>0$, where $B(x,r)$ stands for the closed ball of radius $r$ centered at $x$). 

The upper and lower  box dimensions of a bounded set $E\subset \R^d$ will be denoted $\overline\dim_B E$ and $\underline \dim_B E$ respectively, and  its  Hausdorff and packing dimensions will be  denoted by $\dim_H E$ and $\dim_P E$ respectively (see \cite{Falcbook,Mat,Pesin,Tricot} for introductions to dimension theory).

Multifractal analysis is a natural framework to finely  describe geometrically the heterogeneity in the distribution at small scales of the elements of  $\mathcal M^+_c(\R^d)$. 
Specifically, if $\mu\in \mathcal M^+_c(\R^d)$, this heterogeneity can be described via the lower and upper local dimensions of $\mu$, namely 
$$
\underline d(\mu,x)= \liminf_{r\to 0^+}\frac{\log (\mu(B(x,r)))}{\log (r)}\quad \text{and}\quad   \overline d(\mu,x)=\limsup_{r\to 0^+}\frac{\log (\mu(B(x,r)))}{\log (r)}, 
$$
and the level sets 
\begin{equation*}\label{ealpha}
E(\mu,\alpha,\beta)=\Big \{x\in\supp(\mu): \underline d(\mu,x)=\alpha,\ \overline d(\mu,x)= \beta\Big \} \quad (\alpha\le \beta \in \R\cup\{\infty\}),
\end{equation*}
which form a partition of $\supp(\mu)$ (notice that  $E(\mu,\alpha,\beta)=\emptyset$ whenever $\alpha<0$). The sets 
$$
 \underline E(\mu,\alpha)=\Big \{x\in\supp(\mu): \underline d(\mu,x)=\alpha\Big\},\quad   \overline E(\mu,\alpha)=\Big \{x\in\supp(\mu): \overline d(\mu,x)=\alpha\Big\},
$$
and 
$$
E(\mu,\alpha)= \underline E(\mu,\alpha)\cap \overline E(\mu,\alpha)=E(\mu,\alpha,\alpha)\quad (\alpha \in \R\cup\{\infty\})
$$
are also very natural, and the most studied in the literature (although the sets defined above are empty  if $\alpha<0$ because $\mu$ is a bounded function of Borel sets, it is convenient to include negative values of $\alpha$ in connection with the using along the paper of the Legendre-Fenchel transform of  functions defined on $\R$ or $\R\cup\{\infty\}$).

The {\it lower Hausdorff spectrum} of $\mu$ is the mapping defined as
$$
\underline f_\mu^H: \alpha\in\R\cup\{\infty\} \mapsto \dim_H \underline E(\mu,\alpha),
$$
with the convention that $\dim_H \emptyset=-\infty$, so that $\underline f_\mu^H(\alpha)=-\infty$ if $\alpha<0$. This spectrum provides a geometric hierarchy between the sets $\underline E(\mu,\alpha)
$, which partition the support of~$\mu$. Here, the lower local dimension is emphasized for it provides at any point the best pointwise H\"older control one can have on the measure $\mu$ at  small scales.  However, the upper local dimension is of course of interest, and much attention is paid in general to the sets $E(\mu,\alpha)$ of points at which one has an exact local dimension  $\underline d(\mu,x)= \overline d(\mu,x)$, especially when studying ergodic measures in the context of hyperbolic and more generally non uniformly hyperbolic dynamical systems.

The {\it Hausdorff spectrum}  of $\mu$ is the mapping defined as
$$
f_\mu^H:\alpha\in\R\cup\{\infty\} \mapsto \dim_H E(\mu,\alpha).
$$

Inspired by the observations made by physicists of turbulence and statistical mechanics \cite{hentschel,FrischParisi,HaJeKaPrSh}, mathematicians derived, and in many situations justified the heuristic claiming that for a measure  possessing a self-conformal like property,  its Hausdorff  spectrum should be obtained as the Legendre transform of a kind of free energy function, called $L^q$-spectrum. This gave birth to an abundant literature on the so-called multifractal formalisms \cite{Falcbook,BMP,Be,Olsen,Pesin,LN,bbh,Pey,JLV-T}, which aim at linking the asymptotic statistical properties of  a given measure with its fine geometric properties.

To be more specific we need some definitions. Given $I\in\{\R,\,\R\cup\{\infty\}\}$ and a  fonction $f: I\to \R\cup\{-\infty\}$, the domain of $f$ is defined as $\mathrm{dom}(f)=\{x\in I: f(x)>-\infty\}$.  

Let $\tau: \R\to \R\cup\{-\infty\}$. If $\mbox{dom}(\tau)\neq\emptyset$, the concave Legendre-Fenchel transform, or concave  conjugate function, of $\tau$ is the upper-semi continuous concave function defined as $\tau^*: \alpha\in \R\mapsto \inf\{\alpha q-\tau(q):q\in\mbox{dom}(\tau)\}$ (see \cite{Roc}). We will need a slight extension of this definition.

If $\tau:\R\to \R\cup\{-\infty\}$, $\mathrm{dom}(\tau)\neq\emptyset$, and $0\in \mathrm{dom}(\tau)$,  we define its (extended) concave Legendre-Fenchel transform as
$$
\tau^*:\alpha\in \R\cup\{\infty\}\mapsto 
\begin{cases}
\inf\{\alpha q-\tau(q):q\in\mbox{dom}(\tau)\} &\text{if }\alpha\in \R,\\
\inf\{\alpha q-\tau(q):q\in\mbox{dom}(\tau)\cap \R_-\}&\text{if }\alpha=\infty,
\end{cases} 
$$
with the conventions $\infty \times q=-\infty$ if $q<0$ and $\infty \times 0=0$.  Consequently, $\infty\in \mathrm{dom}(\tau^*)$ if and only if $0=\min (\mathrm{dom}(\tau))$, and in this case $\tau^*(\infty)=-\tau(0)=\max (\tau^*)$. In any case, $\tau^*$ is upper semi-continuous over $\mathrm{dom}(\tau^*)$, and concave over the interval $\mathrm{dom}(\tau^*)\setminus \{\infty\}$ (here the notion of upper semi-continuous function is relative to $\R\cup\{\infty\}$ endowed with the topology generated by the open subsets of $\R$ and the sets $(\alpha,\infty)\cup\{\infty\}$, $\alpha\in\R$).

Now, define  the (lower) {\it $L^q$-spectrum} of $\mu\in\mathcal M^+_c(\R^d)$ as 
$$
\tau_\mu:q\in\mathbb{R}\mapsto \liminf_{r\to 0^+}\frac{\log \sup\Big \{\sum_i \mu(B(x_i,r))^q\Big \}}{\log(r)},
$$
where the supremum is taken over all the centered packings of $\supp(\mu)$ by closed balls of radius  $r$. 

By construction, $\tau_\mu$ is concave and non decreasing,  and $-d\le \tau_\mu (0)=-\overline{\dim}_B\supp(\mu) \le 0=\tau_\mu(1)$, so that one always has $\R_+\subset \mathrm{dom} (\tau_\mu)$;  also  $\tau_\mu^*$ takes values in $[0,d]\cup \{-\infty\}$, and $\mathrm{dom}(\tau_\mu^*)$ is a closed subinterval of $ \R_+\cup\{\infty\}$  (see Propositions~\ref{condtau} and \ref{condtau1}). 

For $\alpha\in\R$ we always have (see \cite[Section 2.7]{Olsen} or \cite[Section 3]{LN}) 
\begin{equation}\label{MF0}
f_\mu^H(\alpha)\le \underline f_\mu^H(\alpha)\le \tau_\mu^*(\alpha)\le\max (\alpha, -\tau_\mu(0))\le \max (\alpha,d);
\end{equation}
we also have
$$
f_\mu^H(\infty) \le \tau_\mu^*(\infty),
$$
a dimension equal to $-\infty$ meaning that the set is empty (the second inequality is not standard, and will be proved in Section~\ref{proof3}; the inequality $ \tau_\mu^*(\alpha)\le\max (\alpha, -\tau_\mu(0))$ is a direct consequence of the definition of $\tau_\mu^*$ and the fact that $\tau_\mu(1)=0$).


We notice that due to \eqref{MF0}, if $f^H_\mu(\alpha)\ge \alpha$ at some $\alpha$, then $0\le \alpha\le d$ and $f^H_\mu(\alpha)= \tau_\mu^*(\alpha)=\alpha$, so that $\alpha$ is a fixed point of $\tau_\mu^*$. Moreover, since $\tau_\mu(1)=0$ and $\tau_\mu$ is concave, the set of fixed points of $\tau_\mu^*$ is the interval $[\tau_\mu'(1^+),\tau_\mu'(1^-)]$. 

We will say that $\mu$  obeys the multifractal formalism at $\alpha\in \R\cup\{\infty\}$ if $ \underline f_\mu^H(\alpha)=\tau_\mu^*(\alpha)$, and that the multifractal formalism holds (globally)  for  $\mu$ if it holds at any $\alpha\in \R\cup\{\infty\}$. 

If  $\underline f_\mu^H(\alpha)$ can be replaced by $ f_\mu^H(\alpha)$ in the previous definition, we will say that the multifractal formalism holds strongly, and it is in this form that this formalism has been introduced and studied the most. It turns out that in this case one has 
\begin{equation*}
\dim_H E(\mu,\alpha)=\dim_P E(\mu,\alpha)=\dim_H \underline E(\mu,\alpha)=\dim_H \overline E(\mu,\alpha)=\tau_\mu^*(\alpha).
\end{equation*} 
However, in general, nice families of discrete measures only obey the multifractal formalism associated with the lower Hausdorff spectrum as defined above.

The multifractal formalism turns out to hold globally, or on some non trivial subinterval of $\mathrm{dom}(\tau_\mu^*)$, for some important classes of continuous measures possessing (or close to have) self-conformal properties, namely  some classes of self-conformal measures (among which some Bernoulli convolutions), Gibbs  and weak Gibbs measures on conformal repellers (e.g. the harmonic measure on such a disconnected set) or attractors of axiom A diffeomorphisms   \cite{collet,EM,Rand,Po,MEH,BMP,LP1,LP2,MR,PW,Pesin,Pat,Ma,He1,Kes,FengO,Fengadv,Iommi,Shm,Testud1,Feng,FengLau,JR,Feng2}, harmonic measure on the Brownian frontier \cite{L},  and scale invariant  limits of certain multiplicative chaos \cite{HoWa,Falc,Mol,AP,B2,barman,barman2, RV13,AB}; in these cases it also holds strongly. It also holds for scale invariant discrete measures obtained as  limits in law of Gibbs measures in the context of random directed polymers \cite{JAFFLevy,BaSe07,barhova12} (see also \cite{AvB,J,Falc2,BS2,OS} for other classes of discrete measures obeying the multifractal formalism). Other examples are special self-affine or Gibbs measures on self-affine Sierpinski carpets \cite{Kin95,Ols98,BaMe07,BaFe}, or on  almost all the attractors of  IFS associated with certain  families of $d\times d$ invertible matrices with small enough singular values \cite{Fal99,BF13}, as well as generic probability measures on a compact subset of $\R^d$   \cite{BucNa,BucS,Bay}. 

The measures mentioned above share the geometric  property to be exact dimensional, i.e. for such a measure $\mu$, there exists $D\in [0,d]$ such that $\displaystyle \lim_{r\to 0^+}\frac{\log (\mu(B(x,r)))}{\log (r)}=D$, $\mu$-almost everywhere. This implies $f_\mu^H(D)\ge D$, hence  $D\in[\tau_\mu'(1^+),\tau_\mu'(1^-)]$ and $\mu$ strongly obeys the multifractal formalism at $D$ by a remark made above. In fact, for any $\mu\in  \mathcal M^+_c(\R^d)$, for $\mu$-almost every $x$ one has $\tau_\mu'(1^+)\le \underline d(\mu,x)\le \overline d(\mu,x)\le \tau_\mu'(1^-)$ (\cite{Ngai}), and for most of the continuous measures in the previous references, $\tau_\mu'(1)$ exists, hence equals $D$; also, $\tau_\mu$ is piecewise $C^1$, and even analytic in certain cases, a typical example being Gibbs measures associated with H\"older potentials on repellers of $C^{1+\alpha}$ conformal mappings.

Another property of the previous measures is, when they obey globally the multifractal formalism, to be homogeneously multifractal (HM), this meaning that the lower Hausdorff spectrum of the restriction of $\mu$ to any closed ball whose interior intersects $\mathrm{\supp}(\mu)$ is equal to the lower Hausdorff spectrum of~$\mu$.

In this paper we solve the inverse problem consisting in constructing, for any concave function $\tau$ satisfying the necessary conditions to be the $L^q$-spectrum of an element of  $\mathcal M^+_c(\R^d)$, an exact dimensional  and (HM) measure whose $L^q$-spectrum equals $\tau$, and which strongly satisfies the  multifractal formalism. More specifically:
\begin{thm}\label{cor}
Let $\tau:\R\to\R\cup \{-\infty\}$ be a concave  function satisfying the necessary properties (see  Propositon~\ref{condtau}) to be the $L^q$-spectrum of some element of $\mathcal M^+_c(\R^d)$. Let $D\in[\tau'(1^+),\tau'(1^-)]$. There exists an (HM) measure $\mu\in \mathcal M^+_c(\R^d)$, exact dimensional with dimension $D$,  and which strongly satisfies the multifractal formalism with $\tau_\mu=\tau$. 
\end{thm} 
Theorem~\ref{cor} will be obtained as a consequence of more general statements which also describe the Hausdorff and packing dimensions of the sets  $E(\mu,\alpha,\beta)$ (Theorem~\ref{thmmain} and Corollary~\ref{cor1} of Section~\ref{State}). We  will also study the inverse problem associated with a finer multifractal formalism designed to describe the more general situation where the  Hausdorff spectrum $f_\mu^H$ and the packing spectrum  $f_\mu^P:\alpha\mapsto \dim_P E(\mu,\alpha)$ differ (Theorem~\ref{thmHP} and Corollary~\ref{corHP} of Section~\ref{State}). As a by product of these results  new multifractal behaviors are exhibited.

In general, $\mathrm{dom}(\underline f_\mu^H)=\{\alpha\in \R\cup\{\infty\}: \underline E(\mu,\alpha)\neq\emptyset\}$ is not necessarily a closed subinterval of $[0,\infty]$, and even when it is the case, the restriction of $\underline f_\mu^H$ to $\mathrm{dom}(\underline f_\mu^H)\cap\R_+$ is not necessarily concave.  Consequently, we  also study the inverse problem consisting in  associating  to  a  function~$f:\R\cup\{\infty\}\to [0,d]\cup\{-\infty\}$ whose domain is a subset of $\R_+\cup\{\infty\}$ and such that $f(\alpha)\le \alpha$ for all $\alpha\ge 0$, an (HM) measure whose lower Hausdorff spectrum is equal to $f$.   We construct such a measure $\mu$ when $\mathrm{dom} (f)$ is a closed subset of $\R_+\cup\{\infty\}$,  $f$ is upper semi-continuous, and $f$ has at least one fixed point,  three properties shared with $\tau_\mu^*$.  Moreover, the measure $\mu$ is exact dimensional. 


Thus, we will prescribe lower Hausdorff spectra in the family: 
$$
\mathcal F(d)=\left \{f: \R\cup\{\infty\}\to [0,d]\cup\{-\infty\}:
\begin{cases}
 \mathrm{Fix}(f)\neq\emptyset\\
  \mathrm{dom}(f)\text{ is a closed subset of $[0,\infty]$}\\
 f \text{ is upper semi-continuous}\\
 \text{$f(\alpha)\le \alpha$ for all $\alpha\in  \mathrm{dom}(f)$}\\
 \end{cases}
 \right \},
$$
where $\mathrm{Fix}(f)$ ($\subset [0,d]$) stands for the set of fixed points of $f$. 

%
\begin{thm}\label{short-thm1}
Let $f\in \mathcal F(d)$. For each  $D\in \mathrm{Fix}(f)$,  there exists an (HM) measure $\mu\in\mathcal M^+_c(\R^d)$, exact dimensional with dimension $D$, such that $\underline f_\mu^H=f$. 
\end{thm} 
This result will be strengthened in Theorem~\ref{thm1} of Section~\ref{State}. It turns out that  the approach used in this paper does not make it possible to replace  $\underline f_\mu^H=f$ by $f^H_\mu=f$ in the previous statement unless one of the following properties hold: $\mathrm{dom}(f)=\mathrm{Fix}(f)$ (see Theorem~\ref{thm1}), or $\mathrm{dom}(f)$ is an interval and $f$ is concave over $\mathrm{dom}(f)\cap\R_+$ (in this case we will get a measure obeying the strong multifractal formalism, see Theorem~\ref{thmmain}). 

Theorems~\ref{cor1} and \ref{short-thm1} have counterparts  in multifractal analysis of H\"older continuous functions; this will be the object of  Section~\ref{MAfonc}, which is independent of Section~\ref{State}. 

\medskip

Before developing further results and comments, let us outline the main ideas leading to the construction of the measure $\mu$ provided by Theorem~\ref{short-thm1}. To establish Theorem~\ref{cor} one must improve this approach in order to control both the finer level sets $E(\mu,\alpha)$ and the upper large deviations spectrum of $\mu$ (to be defined in Section~\ref{addi}) when $f$ is the concave function $\tau^*$, and then use the duality property linking the $L^q$-spectrum and the upper large deviations spectrum to show that the multifractal formalism holds strongly.   

For simplicity, we assume that  $\mathrm{dom}(f)$ is a non trivial compact interval $[\alpha_{\min},\alpha_{\max}]\subset \R_+$, $f$ is continuous over $ [\alpha_{\min},\alpha_{\max}]$,  $0\le f (\alpha)\le \min(\alpha,d)$  over $[\alpha_{\min},\alpha_{\max}]$, and $f(D)=D$ for a unique point $D$ in $[\alpha_{\min},\alpha_{\max}]$.  The homogeneity of the construction of the measure $\mu$ automatically implies that the measure is (HM). 

At first one shows (independently of $f$) that for any $\gamma\in [0,d]$ and $\alpha\ge \gamma$, one can find two Borel probability measures $\mu_{\alpha,\gamma}$ and $\nu_{\alpha,\gamma}$ supported on $[0,1]^d$ such that $\mu_{\gamma,\gamma}=\nu_{\gamma,\gamma}$, $\nu_{\alpha,\gamma}$ is exact dimensional with dimension $\gamma$, and $\nu_{\alpha,\gamma}$ is concentrated on $E(\mu_{\alpha,\gamma},\alpha)$, as well as on the set defined similarly but with $\alpha(\mu,x)$ replaced by $\lim_{n\to\infty}\frac{\log(\mu(I_n(x)))}{-n\log(2)}$, where $I_n(x)$ stands for the closure of dyadic cube semi-open to the right containing $x$. 

Set $A_1=\{\alpha_1=D\}$, and for each integer $m\ge 1$, define $A_{m+1}=A_m\cup\{\alpha_{m+1}\}$, where $\alpha_{m+1}\in [\alpha_{\min},\alpha_{\max}]\setminus A_m$, in such a way  that the set  $\{\alpha_m: m\ge 1\}$ be dense in $[\alpha_{\min},\alpha_{\max}]$.  By using the previous property with $\gamma=f(\alpha)$, for all $m\ge 1$ one gets an integer $n_m$ such that for all $\alpha\in A_m$, for all $n\ge n_m$, there is a collection $G_{m,n}(\alpha)$ of about $2^{n f(\alpha)}$ dyadic subcubes of $[0,1]^d$ such that for all $I\in G_{m,n}(\alpha)$ one has $\mu_{\alpha,f(\alpha)}(I)\approx 2^{-n\alpha}$, $\nu_{\alpha,f(\alpha)}(I)\approx 2^{-nf(\alpha)}$, and $\sum_{I\in G_{m,n}(\alpha)}\nu_{\alpha,f(\alpha)}(I)\in [1/2,1]$. 

For every integer $m\ge 2$, one considers $m$ dyadic closed  subcubes of $[0,1]^d$ of the same generation $n'_m$,   $L_{\alpha_1},\ldots,L_{\alpha_m}$, so that the  $2^{-n'_m/5}$ neighborhood  of  each $L_{\alpha_i}$ does not intersect any of the other $L_{\alpha_j}$.

The measure $\mu$ is constructed on a Cantor set $K=\bigcap_{m\ge 1}\bigcup_{I\in\mathbf{G}_m}$, where the $\mathbf{G}_m$ are families of closed dyadic subcubes of $[0,1]^d$ of generation $g_m$ tending to $\infty$ as $m\to\infty$, constructed recursively according to a scheme roughly as follows:

One obtains $\mathbf{G}_1$ by considering the measure $\mu_{\alpha_1,f(\alpha_1)}=\mu_{D,D}$,  an integer $N_1\ge n_1$ much bigger than $n'_2$ and setting $\mathbf{G}_1=G_{1,N_1}(\alpha_1)=G_{1,N_1}(D)$.  This yields the probability measure $\mu_1$ defined on  $\mathbf{G}_1$ as 
$$
\mu_1(I)=\frac{\mu_{D,D}(I)}{\Big (\sum_{I'\in \mathbf{G}_1}\mu_{D,D}(I')\Big )}.
$$ 
This measure satisfies $\mu_1(I)\approx 2^{-N_1 D}$.  

Suppose now that the set $\mathbf{G}_m$ has been constructed, as well as a probability measure $\mu_m$ on its elements.  One takes $N_{m+1}\ge n_{m+1}$ and integer much bigger than $\max(g_m,n'_{m+2})$, and for each $1\le i\le m+1$, one considers the measure $\mu_{\alpha_i,f(\alpha_i)}$ and the associated set $G_{m+1}(\alpha_i):=G_{m+1,N_{m+1}}(\alpha_i)$. For each $1\le i\le m+1$ and $I_m\in  \mathbf{G}_m$, one defines the set of the elements of $\mathbf{G}_{m+1}$ contained in $I_m$ as $\bigcup_{i=1}^{m+1}\mathbf{G}_{m+1}(I_m,\alpha_i)$, where $\mathbf{G}_{m+1}(I_m,\alpha_i)=\{I_m\cdot L_{\alpha_i}\cdot  I:I\in G_{m+1}(\alpha_i)\}$, and the concatenation $J\cdot J'$ of two closed subcubes of $[0,1]^d$ is obtained as the cube $f_J(J')$, where $f_J$ is the natural contracting similitude mapping $[0,1]^d$ onto $J$ (this operation is associative). One gets a probability measure $\mu_{m+1}$ on  $\mathbf{G}_{m+1}$ by setting, for $I\in G_{m+1}(\alpha_i)$:
\begin{equation}\label{mm1}
\mu_{m+1}(I_m\cdot L_{\alpha_i}\cdot I)= \mu_m(I_m) \frac{\mu_{\alpha_i,f(\alpha_i)}(I)}{\sum_{\alpha\in A_{m+1}}\sum_{I'\in G_{m+1}(\alpha)}\mu_{\alpha,f(\alpha)}(I')}.
\end{equation}
This makes it possible to define a Borel probability measure carried on $K$ and coinciding with $\mu_m$ over $\mathbf{G}_m$ for all $m\ge 1$.

Since $f(\alpha)<\alpha$ except for $\alpha=\alpha_1=D$,  if $N_{m+1}$ is  taken big enough, in \eqref{mm1} for each $i>1$ the contribution of the elements of $G_{m+1}(\alpha)$ is roughly $2^{N_{m+1}(f(\alpha)-\alpha)}$ hence is negligible so that the denominator is equivalent to the single contribution of $\sum_{I'\in G_{m+1}(D)}\mu_{D,D}(I')\in [1/2,1]$. Consequently, for $I_{m+1}\in \mathbf{G}_{m+1}$ of the form $I_m\cdot L_{\alpha_i}\cdot I$, $I\in G_{m+1}(\alpha_i)$, we have the following estimate:
\begin{equation}\label{est}
\mu(I_{m+1})\approx  \mu_m(I_m) \mu_{\alpha_i,f(\alpha_i)}(I)\approx  \mu_m(I_m) 2^{-\alpha_i N_{m+1}}\approx 2^{-\alpha_i g_{m+1}}
\end{equation}
because $g_m\ll N_{m+1}$. Also,  we have that $\#G_{m+1}(\alpha_i)\approx 2^{f(\alpha_i)N_{m+1}}$, hence 
$$
\#\{I\in\mathbf{G}_{m+1}: I\in \mathbf{G}_{m+1}(I_m,\alpha_i)\text{ for some }I_m\in\mathbf{G}_m\}=(\#\mathbf{G}_m) (\#G_{m+1}(\alpha_i)) \approx 2^{f(\alpha_i)g_{m+1}},
$$
again because $g_m\ll N_{m+1}$. The previous estimate and the continuity of $f$ essentially  yield that $f$ is an upper bound for  $\underline f^H$.  Combined with \eqref{est}, it shows that  at generation $m+1$, the mass of $\mu$ is essentially carried by the intervals $I_m\cdot L_{D}\cdot I$, $I\in G_{m+1}(D)$, since we have $1=\|\mu\|\approx \sum_{i=1}^{m+1} 2^{f(\alpha_i)g_{m+1}} 2^{-\alpha_i g_{m+1}}=\sum_{i=1}^{m+1} 2^{(f(\alpha_i)-\alpha_i)g_{m+1}} \approx 2^{(f(\alpha_1)-\alpha_1)g_{m+1}}=1$ (recall that $\alpha_1=f(\alpha_1)=D$). This can be strengthened to show that $\mu$ is exact $D$-dimensional. 

Another important fact is  the natural existence of a family of auxiliary measures used to find a sharp lower bound for $\underline f^H$: with each  $\widehat \beta=(\beta_m)_{m\ge 1}\in \prod_{m=1}^\infty A_m$ is associated the Cantor subset of $K$ defined as 
$$
K_{\widehat\beta}=\bigcap_{m\ge 1}\bigcup_{I\in \mathbf{G}_{\widehat\beta,m}}I,
$$
where $ \mathbf{G}_{\widehat\beta,m}$ is the subset of $\mathbf{G}_m$ obtained by selecting only the intervals of the construction for which one  considers the exponent $\beta_i\in A_i$ at step $i$ for all $1\le i\le m$. Using \eqref{est} and finer properties of the measures $\mu_{\alpha,\gamma}$ one can show that  $K_{\widehat\beta}\subset\underline E(\mu,\beta)$, where $\beta=\liminf_{m\to\infty} \beta_m$. Moreover, the  measures $\nu_{\beta_m, f(\beta_m)}$ can be used to  construct a nice  auxiliary probability measure $\nu_{\widehat\beta}$ carried by $K_{\widehat\beta}$. At first one defines recursively a sequence of measures  $(\nu_{\widehat\beta,m})_{m\ge 1}$ on the atoms  of the sets $\mathbf{G}_{\widehat\beta,m}$, $m\ge 1$, as follows: $\nu_{\widehat \beta,1}$ is the restriction of $\nu_{D,D}$ to $\mathbf{G}_{\widehat\beta,1}(=\mathbf{G}_1)$, and assuming that $\nu_{\widehat\beta,m}$ is constructed on $\mathbf{G}_{\widehat\beta,m}$, if $I_m\in \mathbf{G}_{\widehat\beta,m}$, for $I\in G_{m+1}(\beta_{m+1})$  one sets 
$$
\nu_{\widehat\beta,m+1}(I_m\cdot L_{\beta_{m+1}}\cdot I)= \nu_{\widehat\beta,m}(I_m) \frac{\nu_{\beta_{m+1},f(\beta_{m+1})}(I)}{\sum_{I'\in G_{m+1}(\beta_{m+1})}\nu_{\beta_{m+1},f(\beta_{m+1})}(I')}.
$$
This yields a Borel probability measure $\nu_{\widehat\beta}$ supported on $K_{\widehat \beta}$ such that  $\nu_{\widehat\beta}(I_m\cdot L_{\beta_{m+1}}\cdot I) =\nu_{\widehat\beta,m+1}(I_m\cdot L_{\beta_{m+1}}\cdot I)\approx  \nu_{\widehat\beta,m}(I_m)  \nu_{\beta_{m+1}, f(\beta_{m+1})}(I)$, so that $\nu_{\widehat\beta}(I_m\cdot L_{\beta_{m+1}}\cdot I)\approx \nu_{\beta_{m+1},f(\beta_{m+1})}(I)\approx 2^{-f(\beta_{m+1})g_{m+1}}$ (again since $g_m\ll N_{m+1}$). This can be strengthened to  $\dim_H(\nu_{\widehat\beta})=\liminf_{m\to\infty} f(\beta_m)$, hence $\dim_H  K_{\widehat\beta}\ge \liminf_{m\to\infty} f(\beta_m)$ by the mass distribution principle (see Section~\ref{MDP}). Finally, if $\beta\in[\alpha_{\min},\alpha_{\max}]$ and $\lim_{m\to\infty}\beta_m=\beta$, the continuity of $f$ yields $\underline f^H(\beta)=\dim_H\underline E(\mu,\beta)\ge f(\beta)$.

\subsection{Main statements, and comments}\label{State} 

New definitions and properties are needed to state our main results.

\subsubsection{Additional definitions and properties related to the multifractal formalism} \label{addi}

For $\mu\in \mathcal M^+_c(\R^d)$, recall that let $f^H_\mu$ and $ f^P_\mu$ stand for the Hausdorff spectrum $\alpha\in\R\cup\{\infty\}\mapsto \dim_H E(\mu,\alpha)$ and  the packing spectrum $\alpha\in\R\cup\{\infty\}\mapsto \dim_PE(\mu,\alpha)$ respectively. 

The functions defined below, as well as  some variants, are well known in the literature (\cite{Falcbook,BMP,Be,Olsen,LN}).  They naturally complete $\tau_\mu$ and $\tau_\mu^*$ to describe, in terms of large deviations, the asymptotic behavior of the distribution of the measure $\mu$ at small scales. They  also yield a finer multifractal formalism, which connects geometric properties of the sets $E(\mu,\alpha)$ to large deviations properties associated with $\mu$,  both from the Hausdorff and packing dimensions point of views. In Remark~\ref{link}  (Section~\ref{fullill}) we will explain the connection with another multifractal formalism emphasized in \cite{Olsen}, which is based on a purely geometric approach.

Define  also the upper $L^q$-spectrum of $\mu$ as
\begin{eqnarray*}
q\in\mathbb{R}\mapsto \overline\tau_\mu(q)&=&\limsup_{r\to 0^+}\frac{\log \sup\Big \{\sum_i \mu(B(x_i,r))^q\Big \}}{\log(r)}
\end{eqnarray*}
(this function is not concave in general), as well as  the lower and upper large  deviations spectra $\underline{f}^{\rm LD}_\mu$ and $\overline{f}^{\rm LD}_\mu$:
\begin{eqnarray*}
\alpha\in\R\mapsto\underline{f}^{\rm LD}_\mu(\alpha)&=&\lim_{\epsilon\to 0^+} \liminf_{r\to 0^+}\frac{\log \sup\#\Big \{i:r^{\alpha+\epsilon}\le  \mu(B(x_i,r))\le r^{\alpha-\epsilon}\Big \}}{-\log(r)},
\\
\alpha\in\R\mapsto\overline{ f}^{\rm LD}_\mu(\alpha)&=&\lim_{\epsilon\to 0^+} \limsup_{r\to 0^+}\frac{\log \sup \#\Big\{i:r^{\alpha+\epsilon}\le  \mu(B(x_i,r))\le r^{\alpha-\epsilon}\Big \}}{-\log(r)},
\end{eqnarray*}
\begin{eqnarray*}
\underline{f}^{\rm LD}_\mu(\infty)&=&\lim_{A\to\infty} \liminf_{r\to 0^+}\frac{\log \sup\#\Big \{i:  \mu(B(x_i,r))\le r^{A}\Big \}}{-\log(r)},
\\
\overline{ f}^{\rm LD}_\mu(\infty)&=&\lim_{A\to\infty} \limsup_{r\to 0^+}\frac{\log \sup\#\Big \{i:  \mu(B(x_i,r))\le r^{A}\Big \}}{-\log(r)},
\end{eqnarray*}
where the suprema are taken over all the centered packings of $\supp(\mu)$ by closed balls of radius  $r$. Notice that $0\le \underline\dim_B\supp(\mu)= -\overline\tau(0)\le d$, and $\overline \tau_\mu(1)=0$ (by the same arguments as for the equality $\tau_\mu(1)=0$, see \cite[Section 2.7]{Olsen} or \cite[Section 3]{LN}). 

One always has $\overline \tau_\mu^*\le \tau_\mu^*$, and
\begin{eqnarray}
&&\label{MuFo1}\forall\,\alpha\in\R\cup\{\infty\},\  f^H_\mu(\alpha)\le \underline{ f}^{\rm LD}_\mu(\alpha)\le  \overline\tau_\mu^*(\alpha)\le \max (\alpha, -\overline \tau_\mu(0))\le \max (\alpha,d) ,\\
&&\label{MuFo2}\forall\,\alpha\in\R\cup\{\infty\},\ 
 f^P_\mu(\alpha)\le\overline{ f}^{\rm LD}_\mu(\alpha)\le  \tau_\mu^*(\alpha)\le\max (\alpha, -\tau_\mu(0))\le \max (\alpha,d).
\end{eqnarray}
We will say that $\mu$ obeys the refined multifractal formalism at $\alpha\in\R\cup\{\infty\}$ if $f^H_\mu(\alpha)=\overline \tau_\mu^*(\alpha)$ and $f^P_\mu(\alpha)= \tau_\mu^*(\alpha)$. If $\alpha\in \mathrm{dom}(\tau_\mu^*)\setminus \mathrm{dom}(\overline\tau_\mu^*)$, one necessarily has $E(\mu,\alpha)=\emptyset$, so that in \eqref{MuFo2}, one can only expect the large deviations property $\overline{ f}^{\rm LD}_\mu(\alpha)=  \tau_\mu^*(\alpha)$ to hold. 

The inequalities $f^H_\mu(\alpha)\le \overline \tau_\mu^*(\alpha)$ and  $f^P_\mu(\alpha)\le  \tau_\mu^*(\alpha)$ are established  for $\alpha<\infty$ when $\mu$ is doubling in \cite[Section 2.7]{Olsen}.  The inequalities  $f^H_\mu(\alpha)\le \overline{ f}^{\rm LD}_\mu(\alpha)\le  \tau_\mu^*(\alpha)$ are established in \cite[Section 3]{LN} when $\alpha<\infty$. The other inequalities will be justified in Sections~\ref{proof33} and~\ref{proof34}.  

We notice that if $\mu$ is exact dimensional with dimension $D$, then $D=f^H(D)=\underline{ f}^{\rm LD}_\mu(D)= \overline\tau_\mu^*(D)=\overline{ f}^{\rm LD}_\mu(D)= \tau_\mu^*(D)$, and in the case where $\overline\tau_\mu$ is concave, we have  $D\in [\overline\tau_\mu'(1^+),\overline \tau_\mu'(1^-)](\subset [\tau_\mu'(1^+),\tau_\mu'(1^-)])$, since $\overline \tau_\mu(1)=0$ implies that $\overline \tau_\mu^*(\alpha)=\alpha$  if and only if  $\alpha\in  [\overline\tau_\mu'(1^+),\overline \tau_\mu'(1^-)]$.

Let us  now  describe  the possible behaviors of the $L^q$-spectrum and its Legendre transform. Before stating the corresponding propositions, we need to extend the notion of Legendre-Fentchel transform to functions  $f:\R\cup\{\infty\}\to \R\cup\{-\infty\}$. 

If $f:\R\cup\{\infty\}\to \R\cup\{-\infty\}$ and $\mathrm{dom}(f)\cap \R\neq\emptyset$, we define the  concave Legendre-Fenchel transform of $f$ as 
$$
f^*: q\in \R\mapsto \inf\{q\alpha-f(\alpha):\alpha\in\mbox{dom}(f)\},
$$
with the conventions $q \times \infty= \frac{q}{|q|}\times \infty$ if $q\neq 0$ and $0\times \infty=0$. 

Consequently, if $\infty\in \mathrm{dom}(f)$ and $f$ is bounded from above, then $0=\min (\mathrm{dom} (f^*))$ and $f^*(0)= -\max (\sup(f_{|\R}),f(\infty)$); moreover,  $f^*$ is concave over  $\mathrm{dom}(f^*)$, upper semi-continuous over $\mathrm{dom}(f^*)\setminus \{0\}$, and upper semi-continuous at $0$ only if  and only if $f(\infty)=\max (f)$.

\begin{pro} \label{condtau} Let $\mu\in \mathcal M^+_c(\R^d)$.  
\begin{enumerate}
\item  $\tau_\mu$ is concave, non-decreasing, $\tau_\mu(1)=0$, and  $-d\le \tau_\mu(0)=-\overline\dim_B\supp(\mu) \le 0$.

\item One has either $\mbox{\rm dom}(\tau_\mu) =\R$, or $\mbox{\rm dom}(\tau_\mu) =\R_+$, according to whether the exponent $\limsup_{r\to0^+}\displaystyle  \frac{\log (\inf\{\mu(B(x,r)):x\in\supp(\mu)\})}{\log(r)}$ is finite  or not. Moreover $\tau_\mu^*$ is non-negative on its domain. 
\end{enumerate}
\end{pro}

\begin{pro} \label{condtau1}
Suppose that $\tau:\R\to\R\cup\{-\infty\}$ satisfies the properties of the $L^q$-spectrum described in Proposition~\ref{condtau}.
\begin{enumerate}
\item Suppose that  $\mbox{\rm dom}(\tau) =\R$. Then $\mbox{\rm dom}(\tau^*)$ is the compact interval $I=[\tau'(\infty),\tau'(-\infty)]$, $\tau^*$ is concave and continuous on its domain, and $(\tau^*)^*=\tau$ on $\R$. 

\item Suppose that $\mbox{\rm dom}(\tau) =\R_+$. Then $\infty\in \mbox{\rm dom}(\tau^*)$ with $\tau^*(\infty)=-\tau(0)$ and: 
\begin{enumerate}
\item If $\tau(0)=0$ then $\tau=0$ over $\R_+$, $\mbox{\rm dom}(\tau^*)=\R_+\cup\{\infty\}$ and $ \tau^*=0$ over $\R_+\cup\{\infty\}$. 

\item If $\tau(0)<0$ and $\tau$ is continuous at $0^+$, then $\mbox{\rm dom}(\tau^*)$ is the interval $[\tau'(\infty),\infty]$, $\tau^*$ is concave, continuous,  and increasing over $[\tau'(\infty),\tau'(0^+))$,  $\tau^*(\alpha)=-\tau(0)=\tau^*(\infty)=-\tau(0)$ for all $\alpha\in [\tau'(0^+), \infty)$ and $\tau^*$ is continuous at $\infty$; there are two distinct behaviors according to whether $\tau'(0^+)<\infty$ or not.

\item If $\tau(0)<0$ and $\tau$ is discontinuous at $0^+$, then $\mbox{\rm dom}(\tau^*)$ is the interval $I=[\tau'(\infty),\infty]$. Moreover, $\tau^*(\alpha)=-\tau(0^+)<\tau^*(\infty)=-\tau(0)$ for all $\alpha\in [\lim_{q\to 0^+}\tau'(q^-), \infty)$, so that $\tau^*$ is concave and continuous on $[\tau'(\infty),\infty)$ and discontinuous at $\infty$ (there are also two cases, according to $\lim_{q\to 0^+}\tau'(q^-)=\infty$ or not). 

\item In all the previous cases, $(\tau^*)^*=\tau$ on $\R$.  
\end{enumerate}
\end{enumerate}
\end{pro}

Propositions~\ref{condtau}(1) is standard and proved for instance in \cite{LN} (Proposition 3.2). Propositions~\ref{condtau}(2) and \ref{condtau1}(1) are essentially restatements of Propositions 3.3--3.5 in \cite{LN}. However, for the reader convenience we will provide a proof in Section~\ref{proof3}, where the whole proofs of Propositions~\ref{condtau}(2) and \ref{condtau1} are given.

\subsubsection{Full illustration of the multifractal formalism. Complements to Theorem~\ref{cor}}\label{fullill}

When $\mu$ is a Gibbs measure on a conformal repeller or a self-similar measure on an attractor satisfying  suitable separation conditions, the Hausdorff and packing dimensions are also known for all the sets $E(\mu,\alpha,\beta)$: 
\begin{equation}\label{Eab}
\begin{split}
&\dim_H E(\mu,\alpha,\beta)=\min\{\tau_\mu^*(\gamma):\gamma\in[\alpha,\beta]\},\\
 &\dim_P E(\mu,\alpha,\beta)=\max\{\tau_\mu^*(\gamma):\gamma\in[\alpha,\beta]\},\\
 &\dim_H \underline E(\mu,\alpha)=\tau_\mu^*(\alpha)=  \dim_H \overline E(\mu,\alpha),\\
 &\dim_P \underline E(\mu,\alpha)=\max\{\tau_\mu^*(\gamma):\gamma\ge \alpha\},\  \dim_P \overline E(\mu,\alpha)=\max\{\tau_\mu^*(\gamma):\gamma\le \alpha\}.
 \end{split}
\end{equation}
for all $\alpha\le \beta\in\R_+\cup\{\infty\}$ (see \cite{Ol03,BOS,ZC12} and also \cite{BaSc00,FW} for closely related questions).  We notice that \eqref{Eab} implies that $\dim_H\supp(\mu)=\underline\dim_B \supp(\mu)=\dim_P\supp(\mu)=\overline\dim_B \supp(\mu)=-\tau_\mu(0)$.

It turns out that properties~\eqref{Eab} enter in our exhaustive illustration of the multifractal formalism.

\begin{thm}\label{thmmain}
Let $f\in\mathcal F(d)$. Suppose that $\mathrm{dom}(f)$ is  a non empty  closed subinterval of $[0,\infty]$ and $f$ is concave over $\mathcal{I}\cap\R_+$.

For each fixed point $D$ of $f$,  there exists $\mu\in\mathcal M^+_c(\R^d)$,  exact dimensional with dimension~$D$,  such that $\tau_\mu=f^*=\overline\tau_\mu$, $\tau_\mu^*=f$, and \eqref{Eab} holds for all $\alpha\le \beta\in \R_+\cup\{\infty\}$. Moreover,  the same properties hold if  $\mu$ is replaced by its restriction to any closed ball whose interior intersects $\supp(\mu)$. 
\end{thm}

The following corollary, of which Theorem~\ref{cor} is a consequence, then follows from the fact that if $\tau$ satisfies the properties of Proposition~\ref{condtau} (and so falls into the different situations described in Proposition~\ref{condtau1})(1) and (2), then $\tau^*$ satisfies the assumptions of Theorem~\ref{thmmain}.  
\begin{cor}\label{cor1}
Let $\tau:\R\to\R\cup \{-\infty\}$ be a  function satisfying properties (1) and (2) of Proposition~\ref{condtau}. Let $D\in[\tau'(1^+),\tau'(1^-)]$. 

There exists $\mu\in \mathcal M^+_c(\R^d)$, exact dimensional with dimension $D$,  satisfying \eqref{Eab} for all $\alpha\le \beta\in \R_+\cup\{\infty\}$ with $\tau_\mu=\tau=\overline\tau_\mu$. Moreover,  the same properties hold if  $\mu$ is replaced by its restriction to any closed ball whose interior intersects $\supp(\mu)$. \end{cor}

\begin{rem}
{\rm The behavior described in  Proposition~\ref{condtau1}(1) is illustrated, for instance, by Gibbs and weak Gibbs measures on conformal repellers (see \cite{Olsen,Pesin,FengO}). Such examples, which live on dynamical systems semi-conjugate to subshifts of finite type, cannot exhibit behaviors like those corresponding to Proposition~\ref{condtau1}(2). The behaviors described by  Proposition~\ref{condtau1}(2)(b) are illustrated  by some Gibbs measures on countable Markov shifts  and their geometric realizations \cite{Iommi}, which also obey the multifractal formalism, though in \cite{Iommi} the set $E(\mu,\infty)$ is not studied. The  fact that the behaviors described in Proposition~\ref{condtau1}(2)(a) and (c) be illustrated by measures obeing the mutifractal formalism seems to be new. We notice that the extension of the Legendre transform including $\infty$ in the domain in this case yields Legendre transforms which are not necessarily upper semi-continuous, like $\tau$ at $0$ in case (c). 
}\end{rem}

\begin{rem}{\rm 
Our results illustrate all the possible situations, in term of the function $\tau_\mu$, for which the measure $\mu$ is exact dimensional though $\tau_\mu'(1)$ does not exists. In  \cite{He2}, when $d=1$, for each $D\in (0,1)$ one finds  an exact dimensional measure $\mu$ with dimension $D$ and $L^q$-spectrum equal to $\min (q-1,0)$ over $\R_+$. It is also worth mentioning that in \cite{BatTes} one finds examples of  inhomogeneous Bernoulli measures over $[0,1]$ with an $L^q$-spectrum presenting countably many points of non differentiability over $[1,+\infty)$. 
}\end{rem}

In the previous results, due to \eqref{MuFo1} and  \eqref{MuFo2} we have $f^H_\mu=f^P_\mu=\underline f^{\mathrm{LD}}_\mu=\overline f^{\mathrm{LD}}_\mu$, which reflects  a strong homogeneity of the sets $E(\mu,\alpha)$. The purpose of the refined multifractal formalism is to describe situations irregular enough so that the Hausdorff and packing dimensions of $E(\mu,\alpha)$ differ for most of the $\alpha$. 

The next two results extend in a non trivial way the two previous ones , in particular by exhibiting a new formula for $\dim_H E(\mu,\alpha,\beta)$.  They invoke an extension of \eqref{MuFo1} and \eqref{MuFo2}, which illustrates  the following complement to the multifractal formalism:
If $0\le \alpha<\infty$ and $\alpha\le \beta \le\infty$, $1>r>0$, and $\epsilon>0$,  set 
$$
f_\mu(\alpha,\beta,\epsilon,r)=
\frac{\log \sup\#\Big \{i:r^{\beta+\epsilon}\le  \mu(B(x_i,r))\le r^{\alpha-\epsilon}\Big \}}{-\log(r)},
$$
where the suprema are taken over all the centered packing of $\supp(\mu)$ by closed balls of radius  $r$, and with the convention that $r^\infty=0$.  Then define 
$$
\underline {f}^{\rm{LD}}_\mu(\alpha,\beta)=
\lim_{\epsilon\to 0}\liminf_{r\to 0+}f_\mu(\alpha,\beta,\epsilon,r).
$$
Also, define 
$\underline {f}^{\rm{LD}}_\mu(\infty,\infty)=\underline {f}^{\rm{LD}}_\mu(\infty)$. 
\begin{pro}\label{MFcomp}
Let $\mu\in \mathcal M^+_c(\R^d)$. For any $0\le \alpha\le \beta \le \infty$, one has 
\begin{enumerate}
\item 
\begin{equation*}
\begin{split}
\dim_H E(\mu,\alpha,\beta)&\le f_H(\alpha,\beta):=\min (\overline{ f}^{\rm LD}_\mu(\alpha),\overline{ f}^{\rm LD}_\mu(\beta),\underline {f}^{\rm LD}_\mu(\alpha,\beta))\\
\dim_P E(\mu,\alpha,\beta)&\le  f_P(\alpha,\beta):=\max\{\overline{ f}^{\rm LD}_\mu(\alpha'):\alpha'\in[\alpha,\beta]\}
\end{split}.
\end{equation*}
\item 
\begin{equation*}
\begin{split}
& \dim_L \underline E(\mu,\alpha)\le  \sup\{ f_L(\alpha,\beta):\beta\ge \alpha\}Ê\text{ for }L\in\{H,P\},\\
& \dim_L \overline E(\mu,\alpha)\le  \sup\{f_L(\beta,\alpha):\:\beta\le \alpha\} \text{ for }L\in\{H,P\}
\end{split}.
\end{equation*}
\end{enumerate}
\end{pro}

\begin{thm}\label{thmHP}
Let $d\in\N_+$. Let $\mathcal{I}\subset \mathcal{J}$  be two  non empty closed subintervals of $[0,\infty]$.  Let $f$ and $g\in  \mathcal F(d)$ such that $\mathrm{dom} (f)= \mathcal{I}$, $\mathrm{dom} (g)= \mathcal{J}$, and $f\le g$.  Suppose also that $f$ and $g$ are concave  over $\mathcal{I}\cap\R_+$ and $\mathcal{J}\cap\R_+$ respectively.

For each  $D\in \mathrm{Fix}(f)$,  there exists $\mu\in \mathcal M^+_c(\R^d)$, exact dimensional with dimension $D$, such that  

\begin{enumerate}
\item $\mathrm{dom}(f^H_\mu)=\mathrm{dom}(f^P_\mu)=\mathrm{dom} (\underline{f}^{\rm LD}_\mu)=\mathcal{I}$ and  $\mathrm{dom} (\overline{f}^{\rm LD}_\mu)=\mathcal{J}$. 

\item For all $\alpha\in\mathcal{I}$, $f_\mu^H(\alpha)=\underline{f}^{\rm LD}_\mu(\alpha)=f(\alpha)$, $f_\mu^P(\alpha)=\overline{f}^{\rm LD}_\mu(\alpha)=g(\alpha)$ and $\overline{f}^{\rm LD}_\mu(\alpha)=g(\alpha)$ for $\alpha\in \mathcal{J}\setminus \mathcal{I}$. 

\item More generally, for all $\alpha\le \beta\in\R_+\cup\{\infty\}$, 
\begin{equation*}\label{Eab'}
\begin{split}
&\underline {f}^{\rm LD}_\mu(\alpha,\beta)= f(\alpha,\beta):=\max\{f(\alpha'):\alpha'\in[\alpha,\beta]\},\\&
\dim_H E(\mu,\alpha,\beta)
=\begin{cases}
\min (g(\alpha),g(\beta), f(\alpha,\beta))&\text{ if $[\alpha,\beta]\subset \mathcal{J}$ and $[\alpha,\beta]\cap \mathcal{I}\neq\emptyset$}\\
-\infty &\text{ otherwise}
\end{cases},\\
& \dim_P E(\mu,\alpha,\beta)=
\begin{cases}
\max\{g(\alpha'):\alpha'\in[\alpha,\beta]\}&\text{ if $[\alpha,\beta]\subset \mathcal{J}$ and $[\alpha,\beta]\cap \mathcal{I}\neq\emptyset$} \\
-\infty&\text{ otherwise}
\end{cases},\\
& \dim_H\underline E(\mu,\alpha)= \max\{\dim_H E(\mu,\alpha,\beta):\beta\ge \alpha\}=
 \min (g(\alpha), \max\{ f(\beta):\beta\ge\alpha\}),\\
 &\dim_H\overline E(\mu,\alpha)= \max\{\dim_H E(\mu,\beta,\alpha):\beta\le \alpha\}=
\min(g(\alpha), \max\{ f(\beta):\beta\le\alpha\}), \\
 &
 \dim_P\underline E(\mu,\alpha)=\max\{\dim_P E(\mu,\alpha,\beta):\beta\ge \alpha\}=
 \begin{cases}
 \max \{g(\beta):\beta\ge \alpha\}&\text{ if }\alpha\in [\min (\mathcal J),\max (\mathcal I)]\\
 -\infty &\text{ otherwise}
 \end{cases},\\
 &
 \dim_P\overline E(\mu,\alpha)=\max\{\dim_P E(\mu,\beta,\alpha):\beta\le \alpha\}=
 \begin{cases}
 \max \{g(\beta):\beta\le \alpha\}&\text{ if }\alpha\in [\min ( \mathcal I),\max  ( \mathcal J)]\\
 -\infty &\text{ otherwise}
 \end{cases}.
\end{split}
\end{equation*} 
\item $\overline\tau_\mu=f^*$,  $\overline\tau_\mu^*=f$, $\tau_\mu=g^*$ and $\tau_\mu^*=g$.
\end{enumerate}
Moreover, all the previous properties  hold if $\mu$ is replaced by its restriction to any closed ball whose interior intersects $\supp(\mu)$. 
\end{thm}

Notice that properties (2) and (4) of  the previous statement imply $\dim_H\supp(\mu)=\underline\dim_B \supp(\mu)=-\overline\tau(0)$ and $\dim_P\supp(\mu)=\overline\dim_B \supp(\mu)=-\tau(0)$, because $\max\{\dim_H\underline E(\alpha):\alpha\in \mathcal I\}=-\overline \tau(0)$ and $\max\{\dim_P\underline E(\alpha):\alpha\in \mathcal J\}=-\tau(0)$. 
\begin{cor}\label{corHP}
Let $\tau,\overline\tau:\R\to\R\cup \{-\infty\}$ be two  functions satisfying properties (1) and (2) of Proposition~\ref{condtau}, and such that  $\tau\le \overline\tau$.  

Let $D\in[\overline\tau'(1^+),\overline \tau'(1^-)]\subset [\tau'(1^+),\tau'(1^-)]$. There exists an  exact dimensional measure $\mu\in \mathcal M^+_c(\R^d)$ with dimension $D$  such that :
\begin{enumerate}
\item $\tau_\mu=\tau$ and $\overline\tau_\mu=\overline\tau$; 
\item $\dim_H\supp(\mu)=\underline\dim_B \supp(\mu)=-\overline\tau(0)$ and $\dim_P\supp(\mu)=\overline\dim_B \supp(\mu)=-\tau(0)$;

\item properties (1)-(3) of Theorem~\ref{thmHP} hold with $\mathcal{I}={\rm dom}(\overline \tau^*)$, $\mathcal{J}={\rm dom}(\tau^*)$, $f=\overline\tau^*$ and $g=\tau^*$. 
\end{enumerate}
Moreover, all the previous properties  hold if $\mu$ is replaced by its restriction to any closed ball whose interior intersects $\supp(\mu)$. 
\end{cor}

\begin{rem}\label{link}\rm {\bf Link with Olsen's multifractal formalism.} In \cite{Olsen}, Olsen introduces  three ``multifractal dimensions'' functions $b_\mu\le B_\mu\le \Lambda_\mu$  derived from ``multifractal'' generalizations of Hausdorff and packing measures associated with $\mu$ ($\Lambda_\mu$ and $B_\mu$ are convex, while $b_\mu$ may be not), so that $f^H(\alpha)\le (-b_\mu)^*(\alpha)$ and $f^P(\alpha)\le (-B_\mu)^*(\alpha) \ (\le (-\Lambda_\mu)^*(\alpha))$ for all $\alpha\in\R_+$; one can then say that Olsen's multifractal formalism holds at $\alpha\in \R_+\cup\{\infty\}$ if the previous inequalities are equalities (adding $\alpha=\infty$ in his formalism does not matter). The pair $\{b_\mu,B_\mu\}$ has a geometric meaning, while $\{\tau_\mu,\overline\tau_\mu\}$ relies on large deviations properties. 

This formalism has recently found new illustrations by  inhomogeneous Bernoulli measures on $[0,1]$ \cite{BP,Shen}, and it is particularly well suited to describe $\dim_H E(\mu,\alpha)$ for self-affine measures or Gibbs measures on self-affine Sierpinski carpets and sponges  \cite{Kin95,Ols98,BaMe07,BaFe} (the packing dimension of the sets $E(\mu,\alpha)$ in these situations remains and open question in general). 

Olsen pays a particular attention to compare the  pairs of functions $\{b_\mu,\Lambda_\mu\}$ and $\{\tau_\mu,\overline\tau_\mu\}$. He proves  $b_\mu\le-\overline\tau_\mu$ and $B_\mu\le \Lambda_\mu= -\tau_\mu$ when $\mu$ is doubling, which according to both multifractal formalisms inequalities, implies $(-b_\mu)^*(\alpha)=\overline\tau_\mu^*(\alpha)$ and $(-B_\mu^*)(\alpha)=(-\Lambda_\mu)^*(\alpha))=\tau_\mu^*(\alpha)$ when the refined multifractal formalism used in this paper holds at $\alpha$, hence the validity of his  formalism.  It turns out that even if the  measure $\mu$ we are going to construct to prove Corollary~\ref{corHP} are not doubling, they possess the weaker but close property that there exists a function $\epsilon(r)$ tending to $0^+$ as $r\to 0^+$ such that $\mu(B(x,2r))\le r^{-\epsilon(r)}\mu(B(x,r))$, uniformly in $x\in\supp(\mu)$ (see Section~{reduc}). This is enough for $b_\mu\le-\overline\tau_\mu$ and $B_\mu\le \Lambda_\mu= -\tau_\mu$ to hold, hence Olsen's multifractal formalism to be valid at each $\alpha$ of $\mathrm{dom}(\overline\tau_\mu^*)$. Moreover, using  the equality $(-b_\mu)^*=\overline\tau_\mu^*$ over $\mathrm{dom}(\overline\tau_\mu^*)$,  taking the Legendre-Fenchel transforms of these functions, and using the inequality $b_\mu\le -\overline\tau_\mu=-\overline \tau$, we can get $b_\mu=-\overline\tau_\mu=-\overline \tau$. Similarly  $B_\mu=\Lambda_\mu=-\tau_\mu=-\tau$. 

General upper bounds for $\dim_H \underline E(\mu,\alpha)$ and $\dim_H \overline E(\mu,\alpha)$ are given by \cite[Theorem 2.17(ii)(iii)]{Olsen}, namely $(-(b_\mu\square B_\mu))^*(\alpha)$ and  $(-(B_\mu\square b_\mu))^*(\alpha)$ respectively, where $b\square B$ equals $b$ over $(-\infty,0)$, $b(0)\land B(0)$ at 0, and $B$ over $(0,\infty)$. The previous remarks and the formulas obtained in Theorem~\ref{thmHP} for $\dim_H \underline E(\mu,\alpha)$ and $\dim_H \overline E(\mu,\alpha)$ show that for the measure $\mu$ we construct, the upper bounds estimates $(-(b_\mu\square B_\mu))^*(\alpha)$ and  $(-(B_\mu\square b_\mu))^*(\alpha)$  do provide the correct values for the Hausdorff dimensions. 
\end{rem}

\subsubsection{Prescription of the lower Hausdorff spectrum}  Theorem~\ref{short-thm1} can be precised as follows, according to the properties of the measure we construct: 
\begin{thm}\label{thm1}
Let $f\in\mathcal F(d)$. For each  $D\in \mathrm{Fix}(f)$,  there exists an (HM) measure $\mu\in\mathcal M^+_c(\R^d)$, exact dimensional with dimension $D$, such that $\underline f_\mu^H=f$. 

Moreover, $\mu$ can be constructed so that one has: (1) if $\alpha\in \mathrm{Fix}(f)$ then $f_\mu^H(\alpha)=\alpha$ and (2) if $\alpha\in\mathrm{dom}(f)\setminus \mathrm{Fix}(f)\neq\emptyset$ then $\underline E(\mu,\alpha)=E(\mu,\alpha,\infty)$, and properties (1) and (2)  hold if  $\mu$ is replaced by its restriction to any closed ball whose interior intersects $\supp(\mu)$. 
\end{thm}

\begin{rem}
{\rm It can be shown for the measure we construct that $\dim_P \underline E(\mu,\alpha)=\max\{f(\alpha'):\alpha'\ge \alpha\}$ for all $\alpha\in \R\cup\{\infty\}$. The Hausdorff and packing dimensions of the sets $\overline E(\mu,\alpha)$ and $E(\mu,\alpha,\beta)$ can also be computed; we leave these calculations, based on Corollaries~\ref{dimloc} and \ref{diminf}, to the reader.} 
\end{rem}

\begin{rem}\label{compBucSeu}{\rm (1) The prescription of the lower Hausdorff spectrum has also been studied in \cite{BucSeu13}, where the authors work on $\R$ and construct  (HM) continuous measures, not exact dimensional, but with upper Hausdorff dimension equal to 1, and whose support is equal to $[0,1]$, with a prescribed  lower Hausdorff spectrum  in the class  $\mathcal F$ of functions $f:\R_+\to [0,1]\cup\{-\infty\}$ which satisfy: $f(1)=1$, $\mathrm{dom}(f)$ is a closed subinterval of $[0,1]$ of the form $[\alpha ,1]$ such that $\alpha>0$, and $f_{|[\alpha,1)}=\max(g_{|[\alpha,1)},0)$, where (i) $g$ is the supremum of a sequence of functions $(g_n)_{n\ge 1}$, such that each $g_n$ is constant over its domain supposed to be a closed subinterval of $[0,1]$ and $g_n(\beta)\le \beta$ for all $\beta\in [0,1]$; (ii) $[\alpha,1]$ is the smallest closed interval containing the support of $g$.  It is also shown that for an (HM) measure to be supported by the whole interval $[0,1]$, it is necessary that the support of its lower Hausdorff spectrum contains an interval of the form $[\alpha,1]$, ($0\le \alpha\le 1$). 

The authors of  \cite{BucSeu13} also study the case of non (HM) measures. In this case, they construct non exact dimensional measures with upper Hausdorff dimension 1 whose support is equal to $[0,1]$, with a prescribed lower Hausdorff spectrum in the broader class  $\widetilde {\mathcal F}$ of functions $f$ which satisfy that $f(1)=1$, $0<\inf(\mathrm{dom}(f))$, and $f_{|\mathrm{dom}(f)\setminus\{1\}}=g_{|\mathrm{dom}(f)\setminus\{1\}}$, where $g$ satisfies property (i).  This includes all such functions $f$ for which $g$ is lower semi-continuous. Simultaneously, they also construct a non (HM) measure with lower Hausdorff spectrum given by $g$. 

(2) All the spectra defined in this paper make sense if measures are replaced by non negative functions of  subsets of $\R^d$ to which a notion of support is associated. This is the case for instance of Choquet capacities. In \cite{JLV-Vojak}, the prescription of the spectrum $\alpha\mapsto \dim_H E(C,\alpha)$ is studied, where $C$ is a Choquet capacity on subsets of $[0,1]$ but not a positive measure, which makes the situation easier to tract. The authors can find a capacity with spectrum given by $f$, for any function $f=\sup_{i\ge 1} f_i$, where the functions $f_i:\R\to[0,1]\cup\{-\infty\}$ are such that $\mathrm{dom}(f_i)$ is a non empty closed subset of $\R_+$, and either $f_i=0$ over  $\mathrm{dom}(f_i)$ or $f_i$ is invertible from $\mathrm{dom}(f_i)$ onto $f(\mathrm{dom}(f_i))$, with a continuous inverse (this class of function contains $\widetilde {\mathcal F}$). Moreover the capacity is (HM). 

In \cite{JLV-T}, the authors construct non (HM) non negative functions $C$ of subsets of $[0,1]$, which are not measures, for which the spectrum $\alpha\mapsto \lim_{\epsilon\to 0^+} \dim_H \bigcup_{s>0}\bigcap_{0<r<s}\{x\in\supp(C): r^{\alpha+\epsilon} \le C(B(x,r))\le r^{\alpha-\epsilon}\}$ is prescribed in the class  of upper semi-continuous functions $f:\R_+\mapsto [0,1]\cup\{-\infty\}$ with non empty compact domain. However,  the  spectrum which is prescribed is quite rough with respect to the Hausdorff spectrum. 

}
\end{rem}

\begin{rem}
{\rm It is worth mentioning that in this paper our constructions provide continuous measures even when their dimenson equals 0,  and are based on the properties of the simplest  multifractal measures, namely Bernoulli products. These properties are combined in  recursive concatenations (roughly described in Section~\ref{introduc} and more elaborated than those used for instance to lower bound the Hausdorff dimensions of the sets $E(\mu,\alpha)$ in the study of weak Gibbs measures) in order to converge asymptotically to a prescribed multifractal structure.

We will first prove Theorem~\ref{thm1} because its proof is shorter than that of Theorem~\ref{thmHP}, and it already contains some of the main ideas involved in  the proof of Theorem~\ref{thmHP}. However, none of the two proofs can be reduced to the other one regarding the computation of $\underline f^H_\mu$.}
\end{rem}

The paper is organized as follows.  Section~\ref{sec2} introduces preliminary information about Bernoulli measures. Section~\ref{proof1} is dedicated to the proof of Theorem~\ref{thm1}, Section~\ref{proof2} to the proof of Theorem~\ref{thmHP}, and Section~\ref{proof3} contains the proofs of  Propositions~\ref{condtau}, ~\ref{condtau1} and \ref{MFcomp}, as well as some inequalities in \eqref{MuFo1} and \eqref{MuFo2}. Section~\ref{MDP} gives a short account about the mass distribution principle.

We finish this introduction with the connection of Theorems~\ref{cor} and~\ref{short-thm1} with multifractal analysis of H\"older continuous functions. 

\subsection{Application to multifractal analysis of  H\"older continuous functions}\label{MAfonc} Multifractal analysis of functions has developed in parallel to multifractal analysis of measures, mainly under the impulse of Frish and Parisi's note about multifractality in fully developed turbulence~\cite{FrischParisi}, and with its own multifractal formalisms \cite{MAB,Jafsiam1,JAFFMEYER,JAFFJMPA,JAFFJMP}. These are based on the link between pointwise H\"older regularity and the wavelet expansions of H\"older continuous functions. 

Theorems~\ref{cor} and~\ref{short-thm1} can be used to build H\"older continuous wavelet series with prescribed upper semi-continuous lower Hausdorff spectra, and also to give a full illustration of the  multifractal formalism for H\"older continuous functions based on the  so-called wavelet leaders \cite{JAFFJMP},  according to the bridge made in \cite{BSw} between this formalism and the  multifractal formalism for measures. We will restrict ourselves to the case $d=1$.

To be more specific, recall first that if $F:\R\to\R$ is a bounded H\"older continuous function, for each $x_0\in\R$, one defines the pointwise H\"older exponent of $f$ at $x_0$ as 
$$
h_F(x_0)=\sup\{h\ge 0: \, \exists\, \text{$P\in\R[X]$},\  |F(x)-P(x-x_0)|=O(|x-x_0|^h)\text{ as }|x-x_0|\to 0\},
$$
where $|x-x_0|$ stands for the Euclidean norm of $x-x_0$.   This exponent is the counterpart for functions of the lower local dimension for measures. 

One usually call singularities spectrum of $F$ the mapping
$$
h\mapsto\dim_H \{x\in \R:  h_F(x)=h\} \quad (h\in \R\cup\infty)
$$
(we keep the terminology lower Hausdorff spectrum for a slightly different spectrum defined below). Notice that if $f$ is $\gamma$-H\"older, then $\{x\in \R:  h_F(x)=h\}=\emptyset$ if $h<\gamma$. 

We are going to restrict the study to $[0,1]$.  We fix a wavelet basis $\{\psi_I\}$ ($I$ describing all the dyadic subintervals of $\R$), so that the mother wavelet is in the Schwartz class (see \cite[Ch. 3]{Me}) and the $\psi_I$ are normalized to have the same supremum norm. 

Denoting $\{\lambda_I\}$ the collection of the wavelet coefficients of $F$ in the basis $\{\psi_I\}$, we define $L_I=\sup\{|\lambda_{I'}|\}$, the supremum being taken over all the dyadic intervals included either in $I$ or in the 2 dyadic intervals  of the same generation as $I$ neighboring $I$. Then, we define $\supp(F)$ has the closed set of those $x\in[0,1]$ such that  $|L(I_n(x))|>0$ for all $n\ge 1$, where $I_n(x)$ stands for the closure of the unique semi-open to the right dyadic cube of generation $n$ which contains $x$. According to \cite{JAFFJMP}, this set does not depend on the choice of $\psi$; moreover, for $x\in\supp(F)$, one has 
$$
h_F(x)=\liminf_{n\to\infty} \frac{\log_2(|L(I_n(x))|)}{-n}.
$$ 

For $h\in\R\cup\{\infty\}$, we set
$$ 
\underline E(F,h)=\{x\in \supp(F):  h_F(x)=h\}.
$$
The lower Hausdorff spectrum of $F$ is the mapping 
$$
\underline f_F^H:h\mapsto\dim_H \underline E(F,h)\quad(h\in\R\cup\{\infty\}).
$$
We say  that $F$ is  homogeneously multifractal (HM) if for any $h\in \R\cup\{\infty\}$, the Hausdorff dimension of $\underline E(F,h)\cap B$ does not depend on the ball $B$ whose interior intersects $\supp(F)$.

A basic idea \cite{BSw} to relate multifractal analysis of functions to that of measures is to consider wavelet series of the form
$$
F_{\mu,\gamma_1,\gamma_2}=\sum_{I\subset[0,1]} |I|^{\gamma_1}\mu(I)^{\gamma_2}\psi_I,
$$
where $|I|$ stands for the diameter of $I$, $\gamma_1\ge 0$, $\gamma_2>0$, $\mu\in\mathcal{M}^+_c(\R)$ with $\supp(\mu)\subset [0,1]$, and $\gamma=\gamma_1+ \gamma_2 \liminf_{n\to\infty} \frac{\log_2 (\max\{\mu(I):I\text{ dyadic }\subset [0,1], \,|I|= 2^{-n}\})}{-n}>0$, so that the function $F_{\mu,\gamma_1,\gamma_2}$ is $\beta$-H\"older continuous for all $0<\beta<\gamma$. Then, the study achieved in \cite{BSw} yields 
\begin{equation}\label{bridge}
\underline E(F_{\mu,\gamma_1,\gamma_2},h)=\underline E\Big (\mu,\frac{h-\gamma_1}{\gamma_2}\Big )
\end{equation}
 for all $h\in\R\cup\{\infty\}$,  so that any information about the multifractal structure of measures should transfer to a similar one for this class of wavelet series. In particular, it is clear from \eqref{bridge} that $\dim_H E(F_{\mu,\gamma_1,\gamma_2},h)\le \frac{h-\gamma_1}{\gamma_2}$.

 \subsubsection{Prescription of the lower Hausdorff spectrum} 
 
\begin{thm}
 Let $f:\R_+\cup\{\infty\}\to [0,1]\cup\{-\infty\}$ be upper semi-continuous. Suppose that  $\mathrm{dom}(f)$  is a closed subset $\mathcal{I}$ of $[0,\infty]$ such that $0<\min(\mathcal{I})<\infty$.
  There exists an (HM) H\"older continuous function $F$ such that $\underline f_F^H=f$. 
\end{thm}
\begin{proof} For $\lambda>0$ set $\theta(\lambda)=\sup\{f(h)/\lambda h: h\in \mathcal I\}$, with the convention $x:\infty=0$ for all $x\ge 0$. Since $f$ is upper semi-continuous and bounded,  $\theta(\lambda)$ is reached at some $h<\infty$. Moreover, the mapping $\theta$ is continuous, and we have $\theta (1/\min(\mathcal I)) \le 1$ by definition of $f$. Now we distinguish two cases. 

If $f\not\equiv 0$ over $\mathcal I$, then  $\theta(\lambda)$ tends to $\infty$ as $\lambda$ tends to $0^+$, so the continuity of $\theta$ yields $0<\lambda_0\le 1/\min(\mathcal I)$ such that $\theta(\lambda_0)=1$, hence $f(h)\le \lambda_0 h$ for all $h\in \mathcal I$, with equality at some $h$. Let $\widetilde f= f(\lambda_0^{-1}\cdot )$. By construction we have $\widetilde f\in\mathcal F(1)$. Put  $\widetilde f$ in Theorem~\ref{short-thm1} to get an (HM) measure in $\mathcal M^+_c(\R^1)$ supported on $[0,1]$ whose lower Hausdorff spectrum is given by $\widetilde f$. Then $F= F_{\mu,0,\lambda_0^{-1}}$ is H\"older continuous and  has $f$ as lower Hausdorff spectrum by~\eqref{bridge}. 

If $f\equiv 0$ on $\mathcal I$, then $\widetilde f=f(\cdot-\min(\mathcal{I})$ belongs to $\mathcal F(1)$ (with $0$ as unique fixed point). Put  $\widetilde f$ in Theorem~\ref{short-thm1} to get an (HM) measure in $\mathcal M^+_c(\R)$ supported on $[0,1]$ whose lower Hausdorff spectrum is given by $\widetilde f$. Then $F= F_{\mu,\min(\mathcal I),1}$ is H\"older continuous and  has $f$ as lower Hausdorff spectrum.
\end{proof}

\begin{rem}{\rm 
In  \cite{BucSeu13}, the measures described in Remark~\ref{compBucSeu}(1) are used to build (HM) functions of the form $F=F_{\mu,\gamma_1,\gamma_2}$ with $\supp(F)=[0,1]$. Previously in \cite{J2}, Jaffard build non (HM) wavelet series with prescribed spectrum in the class of functions $f: (0,\infty)\to [0,1]$  which are representable as the supremum of a
countable collection of step functions.}
\end{rem}
 
\subsubsection{Full illustration of the multifractal formalism} Our results also yield a full illustration of  the multifractal formalism for H\"older continuous  functions whose support is a subset of $[0,1]$. This requires some preliminary definitions and facts. 

If  $F=\sum_{I}\lambda_I\psi_I$ is a non trivial  such  function, i.e. $\emptyset\neq \supp(F)\subset [0,1]$,  denote by $T(q)$ the $L^q$-spectrum associated with the wavelet leaders $(L_I)_{I\subset [0,1]}$, i.e. the concave non decreasing function
$$
T_F(q)=\liminf_{n\to\infty}\frac{-1}{n} \log_2 \sum_{I\in G^*_n} L_I^q\quad (q\in \R), 
$$
where $G^*_n$ stands for the set of dyadic cubes $I$ of generation $n$ for which $L_I>0$.  Due to \cite{JAFFJMP} again, this function does not depend on the choice of $\{\psi_I\}$ if the mother wavelet is in the Schwartz class. Moreover, if $F$ takes the form $F_{\mu,\gamma_1,\gamma_2}$, one has almost immediately 
\begin{equation}\label{bridge2}
\tau_F(q)= \tau_\mu(\gamma_2 q)-\gamma_1 q\quad  ( \ q\in\R).
\end{equation}

From now on we discard the trivial situation for which $\lim_{n\to\infty}\frac{\log_2(\max\{L_I:I\in G_n^*\})}{-n}=\infty$, in which case we have $T_F=-\overline \dim_B\supp(F)\mathbf{1}_{\{0\}}+(-\infty)\mathbf{1}_{\R_-^*}+(\infty)\mathbf{1}_{\R_+^*}$ and $\supp(F)=\underline E(F,\infty)$. 

Now we have  $\liminf_{n\to\infty}\frac{\log_2(\max\{L_I:I\in G_n^*\})}{-n}<\infty$, so that there exists  $\beta\ge 0$ such that $T_F(q)\le \beta q$ for all $q\ge 0$, which ensures that $T_F$ takes values in $\R\cup\{-\infty\}$. 

We say that $F$ satisfies the multifractal formalism if  $\dim_H \underline E(F,h)=T^*_F(h)$ for all $h\in \R_+\cup\{\infty\}$. This is essentially the multifractal formalism considered in \cite{JAFFJMP}. One simple, but important, observation in \cite{BSw} is that from \eqref{bridge} and \eqref{bridge2} follows the fact that if $\mu\in\mathcal M_c^+(\R)$ is supported on $[0,1]$ and obeys the multifractal formalism for measures, then if $F_{\mu,\gamma_1,\gamma_2}$ is H\"older continuous, it obeys the multifractal formalism just defined above.

Let us now examine the features of the $L^q$-spectrum when the multifractal formalism holds. They consist in three properties denoted $(i)$, $(ii)$, $(iii)$ and explained below (in a first reading one can just read them, jumping the justifications, and directly go to Theorem~\ref{mff}).

We have the following properties:  $L_I=O(|I|^\alpha)$ for some $\alpha>0$ by the H\"older continuity assumption, hence 

$(i)$ there exists $\alpha>0$ and $c\in [0,1]$ (here $c=\overline \dim_B\supp(F)$) such that $T_F(q)\ge \alpha q -c$ for all $q\ge 0$.  Moreover, 

$(ii)$ $T_F$ satisfies the same  properties as $\tau$ in Proposition~\ref{condtau}(2), in particular $T_F^*$ is non negative over its domain, by the same arguments as in the proof of Proposition~\ref{condtau}(2). 

Due to $(i)$, we can define 
$$
q_0=\inf\{q\ge 0: T_F(q)>0\}.
$$
If $F$ satisfies the multifractal formalism, we must have the third property:

$(iii)$ Either $q_0>0$ or $T_F'(0^+)>0$ and $T_F(q)=T_F'(0^+)q$ for all $q>0$.

Let us justify this. If $q_0=0$, there exists $c'\in \R$ such that $T_F(q)=T_F'(0^+)q+c'$ for all $q> 0$, for otherwise $-\infty<T_F^*(T_F'(q^-))<0$ for all $q>0$ large enough so that $T_F'(q^-)<T_F'(0^+)$, while $T_F^*$ must be non negative over its domain. Also, since there exists $\beta>0$ such that $0<T_F(q)\le \beta q$ for $q>0$, we have $c'=0$.  If, moreover, $F$ satisfies the multifractal formalism, we must have $T_F'(0^+)>0$, otherwise $T_F=-\overline \dim_B\supp(F)$ over $\R^*_+$, and no H\"older continuous function $F$ can fulfill the multifractal formalism with $T_F$ as $L^q$-spectrum; indeed, this would imply $T_F^*(0)=\overline \dim_B\supp(F)\ge 0$ hence $\underline E(F_\mu,0)\neq\emptyset$. 
\begin{thm}\label{mff}
Suppose that a non decreasing concave function $T$ satisfies the above properties $(i)$--$(iii)$ necessary to be the $L^q$-spectrum of a H\"older continuous function whose support is a non empty subset of $[0,1]$. Then there exists an (HM) H\"older continuous  function  $F$ with $\supp(F)\subset [0,1]$, which satisfies the multifractal formalism with $T_F=T$.  
\end{thm}
\begin{proof}
Let $q_0=\inf\{q\ge 0: T(q)>0\}$. 
If  $q_0>0$, then $\tau(q)=T(q_0 q)$ satisfies the properties of Proposition~\ref{condtau}, so that it is the $L^q$-spectrum of an exact dimensional measure $\mu$ of dimension $D$, for any $D\in [q_0T'(q_0^+),q_0T'(q_0^-)]\subset [0,1]$ by Theorem~\ref{cor}. Moreover, the inequality $\tau_\mu(q)\ge \alpha q_0 q-c$ implies that $\tau_\mu^*(\beta)=-\infty$ for all $\beta<\alpha q_0$. Consequently, the function $F=F_{\mu,0,1/q_0}$ is $(\alpha-\epsilon)$-H\"older continuous for all $\epsilon>0$ and, and due to \eqref{bridge} and \eqref{bridge2}  it  fulfills the multifractal formalism for wavelet leaders with $T_F:q\mapsto \tau_\mu(q/q_0)=T(q)$.

If $q_0=0$,  the function defined as $\tau(q)=T(q)-T'(0^+)q$ satisfies the conditions required by Proposition~\ref{condtau}. Take the (HM) measure  $\mu$ associated with this function $\tau$ by Theorem~\ref{cor}. Then, the function $F_{\mu,T'(0^+),1}$  is $(T'(0^+)-\epsilon)$-H\"older continuous for all $\epsilon>0$ and and due to \eqref{bridge} and \eqref{bridge2} it fulfills the multifractal formalism for wavelet leaders,  with $T_F:q\mapsto \tau_\mu (q)+T'(0^+)q=T(q)$.
\end{proof}

\begin{rem}
{\rm In \cite{JAFFJMPA}, Jaffard uses a multifractal formalism associated with wavelet coefficients (not leaders). He introduces  a class of concave functions  such that to each element $\tau$ of this class he can associate a Baire space $V$ build on Besov spaces, so that generically an element of $V$ has a non decreasing Hausdorff spectrum obtained as the Legendre transform of $\tau$ computed by taking the infimum over a subdomain of $\R_+$. }
\end{rem}

The paper is organized as follows.  Section~\ref{sec2} introduces preliminary information about Bernoulli measures. Section~\ref{proof1} is dedicated to the proof of Theorem~\ref{thm1}, Section~\ref{proof2} to the proof of Theorem~\ref{thmHP}, and Section~\ref{proof3} contains the proofs of  Propositions~\ref{condtau}, ~\ref{condtau1} and \ref{MFcomp}, as well as some inequalities in \eqref{MuFo1} and \eqref{MuFo2}. Section~\ref{MDP} gives a short account about the mass distribution principle.

\section{Some notations, and preliminary facts about Bernoulli measures}\label{sec2}

\subsection{Notations}\label{nota}
For $n\ge 1$, define
$$
\mathcal{F}_n=\left \{\prod_{i=1}^d [k_i2^{-n},(k_i+1)2^{-n}]: 0\le k_i<2^n\right\}.
$$
If $x\in \R^d$ and $n\ge 0$ we denote by $I_n(x)$ the closure of the unique dyadique cube, semi-open to the right, of generation $n$, that contains $x$. 

Given two elements $I=\prod_{i=1}^d [k_i2^{-n},(k_i+1)2^{-n}]$ and $I'=\prod_{i=1}^d [k'_i2^{-n'},(k'_i+1)2^{-n'}]$ of $\bigcup_{p\ge 0}\mathcal{F}_p$, the concatenation $I\cdot I'$ of $I$ and $I'$ is defined as
\begin{equation}
\label{conc}
I\cdot I'=\prod_{i=1}^d [k_i2^{-n}+k_i'2^{-n-n'},k_i2^{-n}+ (k_i'+1)2^{-n-n'}]. 
\end{equation}

\medskip

If $J$ is a closed dyadic cube of generation $j$, we denote by $\mathcal{N}_1(n,J)$ the set made of $J$ and the $3^d-1$ dyadic cubes of generation $n$ neighboring $J$, and denote by $\mathcal{N}_2(n,J)$ the union of $\mathcal{N}_1(n,J)$ and the $5^d-3^d$ closed dyadic cubes of generations $n$ neighboring $\mathcal{N}_1(n,J)$. 

Fix $L_0$ a closed dyadic subcube of $[0,1]^d$ of generation 2 which does not touch $\partial [0,1]^d$. For each integer $k\ge 1$, we can fix a collection $\mathcal{L}(k)$ of $k$ closed dyadic cubes of generation $\ell(k)=\Big \lfloor \frac{\log_2(6^dk)}{d}\Big \rfloor +3$, all contained in $L_0$, such that the sets $\bigcup_{I\in \mathcal{N}_2(\ell(k),L)} I$, $L\in \mathcal{L}(k)$, are pairwise disjoint. This property will imply that the measure constructed in Section~\ref{constrmu2} is ``weakly'' doubling.

\medskip

If $\nu$ is a positive Borel measure supported on $[0,1]^d$ and $x$ belongs to the support of $\nu$, 
we set 
$$
d(\mu,x,n)=\frac{\log \nu(I_n(x))}{-n\log (2)}.
$$

\medskip

For the definitions of the $s$-dimensional Hausdorff and packing measures denoted respectively as $\mathcal H^s$ and $\mathcal P^s$ in this paper, the reader is referred to \cite{Falcbook} or \cite{Mat}. 

\subsection{Some facts about Bernoulli measures}

If $q\in[0,1]$, let $H(q)=-q\log_2(q)-(1-q)\log_2(1-q)$, with the convention $0\times\infty=0$.  Also, denote by $\nu_q$ the Bernoulli measure generated by $(q,1-q)$ on $[0,1]$. 

For each $0\le \gamma\le d$ and $\alpha\ge \gamma$ , we can fix $(p,q)=(p_{\alpha,\gamma},q_{\alpha,\gamma}) \in [0,1]^2$ such that 
$$
\begin{cases}
\alpha=-d\cdot(q\log_2(p)+(1-q)\log_2(1-p))\\
\gamma= d\cdot H(q).
\end{cases}
$$
Indeed, since $\gamma/d\in [0,1]$, there are clearly two solutions to $H(q)=\gamma/d$ in $[0,1]$ if $\gamma<d$ and only one if $\gamma=d$, namely $1/2$. Fix $q$ one solution. Now we seek for $p\in [0,1]$ such that $\alpha(p)=-d(q\log_2(p)+(1-q)\log_2(1-p))$ be equal to $\alpha$. If $q\in\{0,1\}$, this is immediate. Otherwise, the mapping $\alpha(p)$ decreases over $(0,q]$  from  $\infty$ to $\alpha(q)=\gamma$, and it increases on $[q,\infty)$ from $\gamma$ to $\infty$. So in this case  there is at least one and at most two solutions to $\alpha(p)=\alpha$ since we assumed that $\gamma\le \alpha$. 

\medskip

We will use the following classical fact, which is just a consequence of the strong law of large numbers. 
\begin{pro}
Suppose that $d=1$. Let $(p,q)\in [0,1]^2$. For $\nu_{q}$-almost every $x\in [0,1]$,
\begin{equation*}
\begin{split}
&\displaystyle\lim_{n\to \infty} d({\nu_p},x,n)= -q\log_2(p)-(1-q)\log_2(1-p),\\
&\displaystyle\lim_{n\to \infty} d({\nu_q},x,n)= H(q).
\end{split}
\end{equation*}
\end{pro}

\begin{cor}\label{dim}
For every $0\le \gamma\le d$ and $\alpha\ge\gamma$, for $\nu^{\otimes d}_{q_{\alpha,\gamma}}$-almost every $x\in [0,1]^d$
\begin{equation*}
\begin{split}
&\displaystyle\lim_{n\to \infty} d({\nu^{\otimes d}_{p_{\alpha,\gamma}}},x,n)=\alpha,\\
&\displaystyle\lim_{n\to \infty}d({\nu^{\otimes d}_{q_{\alpha,\gamma}}},x,n)=\gamma.
\end{split}
\end{equation*}
\end{cor}
Notice that if $\gamma>0$, both the $\nu^{\otimes d}_{p_{\alpha,\gamma}}$ mass and the $\nu^{\otimes d}_{q_{\alpha,\gamma}}$ mass of the boundaries of closed dyadic subcubes of $[0,1]^d$ vanish.

Now for  every $0\le \gamma\le d$, $\alpha\ge\gamma$, $n\in \N$ and $\epsilon>0$ define 
$$
E(\alpha,\gamma, n,\epsilon)=\left\{x\in [0,1)^d: \, \begin{cases} d({\nu^{\otimes d}_{p_{\alpha,\gamma}}},x,n) \in [\alpha-\epsilon,\alpha+\epsilon],\\
d({\nu^{\otimes d}_{q_{\alpha,\gamma}}},x,n)\in [\gamma-\epsilon,\gamma+\epsilon]
\end{cases}\right\}.
$$
Let $(\epsilon_m)_{m\ge 1}\in(0,1)^{\N_+}$ be a decreasing sequence converging to 0. 

By Corollary~\ref{dim}, for each $m\in \N_+$ we can fix an integer  $n^{0}_{m}(\alpha,\gamma)$ 
such that  
\begin{equation}\label{fund}
\nu^{\otimes d}_{q_{\alpha,\gamma}}(F_m(\alpha,\gamma))\ge 1/2, \ \mbox{ with } F_m(\alpha,\gamma)=\bigcap_{n\ge  n^0_{m}(\alpha,\gamma)} E(\alpha,\gamma,n,\epsilon_m/2).
\end{equation}
We notice that for all $n\ge n^0_{m}(\alpha,\gamma)$, since $\nu^{\otimes d}_{q_{\alpha,\gamma}}(F_m(\alpha,\gamma))\le 1$ and for each $I\in\mathcal F_n$ such that $I\cap F_m(\alpha,\gamma)\neq\emptyset$ we have $2^{-n(\gamma+\epsilon_m/2)}\le \nu^{\otimes d}_{q_{\alpha,\gamma}}(I)$,  we have 
$ \#\{I\in \mathcal F_n: I\cap F_m(\alpha,\gamma)\neq\emptyset\}\le 2^{n(\gamma+\epsilon_m/2)}$.

Then, let
$$
\widetilde E(\alpha,\gamma,n,\epsilon)=\left\{x\in F_m(\alpha,\gamma): \displaystyle\frac{\log \nu^{\otimes d}_{q_{\alpha,\gamma}}(I_n(x)\cap F_m(\alpha,\gamma))}{-n\log (2)}\in [\gamma-\epsilon_m,\gamma+\epsilon_m]
\right\}.
$$

We can find $n_m(\alpha,\gamma)\ge n^0_{m}(\alpha,\gamma)$ such that 

\begin{equation}\label{controlmuq}
\nu^{\otimes d}_{q_{\alpha,\gamma}}(  \widetilde F_m(\alpha,\gamma))\ge 1/2-1/2^{m}, \ \mbox{ with } \widetilde F_m(\alpha,\gamma)=\bigcap_{n\ge n_{m}(\alpha,\gamma)} \widetilde E(\alpha,\gamma,n,\epsilon_m).
\end{equation}
Indeed, for $n\ge \widetilde n^0_m(\alpha,\gamma)$ we have 
\begin{eqnarray}
\nonumber&&\nu^{\otimes d}_{q_{\alpha,\gamma}}(F_m(\alpha,\gamma)\setminus  \widetilde E(\alpha,\gamma,n,\epsilon_m))\\
\nonumber&=&\nu^{\otimes d}_{q_{\alpha,\gamma}}\big (\{x\in F_m(\alpha,\gamma): \nu^{\otimes d}_{q_{\alpha,\gamma}}(I_n(x)\cap F_m(\alpha,\gamma))
< 2^{-n(\gamma+\epsilon_m)}\}\big )\\
\nonumber&\le& \sum_{I\in\mathcal F_n: \  \nu^{\otimes d}_{q_{\alpha,\gamma}}(I\cap F_m(\alpha,\gamma))< 2^{-n(\gamma+\epsilon_m)}}\nu^{\otimes d}_{q_{\alpha,\gamma}}(I\cap F_m(\alpha,\gamma))\\
\nonumber&\le & ( \#\{I\in\mathcal F_n: I\cap F_m(\alpha,\gamma)\neq\emptyset\}) 2^{-n(\gamma+\epsilon_m)}\\
\label{111}&\le& 2^{n(\gamma+\epsilon_m/2)} 2^{-n(\gamma+\epsilon_m)}=2^{-n\epsilon_m/2},
\end{eqnarray}
so \eqref{controlmuq} follows if we  choose $n_m(\alpha,\gamma)$ such that $\sum_{n\ge n_m(\alpha,\gamma)}2^{-n\epsilon_m}\le 2^{-m}$.

We define 
\begin{equation}\label{muag}
\mu_{{\alpha,\gamma}}=\nu^{\otimes d}_{p_{\alpha,\gamma}}\mbox{ and }\nu_{\alpha,\gamma}=\nu^{\otimes d}_{q_{\alpha,\gamma}} (\cdot \cap F_m({\alpha,\gamma})).
\end{equation}

We can now gather a series of properties which will be used in the proofs of our main results.
\begin{fact}{\rm Let $m\in\N_+$, $0\le \gamma\le d$ and $\alpha\ge\gamma$. 

\noindent (1) If $N\ge n_m (\alpha,\gamma)$, by construction we have 
\begin{equation}\label{tm}
1/2\le \nu_{\alpha,\gamma} (F_m(\alpha,\gamma))=\sum_{I\in\mathcal F_N: I\cap F_m(\alpha,\gamma)\neq\emptyset}\nu_{\alpha,\gamma} (I)\le 1
\end{equation}
(notice that if $\alpha=\gamma$, by construction the above property  also holds for $\mu_{\alpha,\alpha}$). 
Consequently, since for each $I\in\mathcal F_N$ such that $I\cap F_m(\alpha,\gamma)\neq\emptyset$ we have $2^{-N(\gamma+\epsilon_m)}\le \nu_{\alpha,\gamma}(I)\le 2^{-N(\gamma-\epsilon_m)}$, we get
\begin{equation}\label{cardG}
2^{-1}2^{N(\gamma-\epsilon_m)}\le \#\{I\in \mathcal F_N: I\cap F_m(\alpha,\gamma)\neq\emptyset\}\le 2^{N(\gamma+\epsilon_m)}.
\end{equation}

By construction of $\widetilde F_m(\alpha,\gamma)$, we also have 
\begin{eqnarray}\label{0000}
\sum_{I\in \mathcal{F}_{N}:\, I\cap  \widetilde F_{m}(\alpha,\gamma)=\emptyset}\nu_{\alpha,\gamma}(I)\le 2^{-m}.
\end{eqnarray}
\noindent (2) If $J$ is a dyadic cube of generation $n\ge n_m(\alpha,\gamma)$, $J\cap F_m(\alpha,\gamma)\neq\emptyset$, and $N\ge n$, then  we have by construction 
\begin{equation}\label{0}
\nu_{{\alpha,\gamma}} (J)=\sum_{\substack{J\supset I\in \mathcal{F}_N,\\ I\cap F_m({\alpha,\gamma})\neq\emptyset} }\nu_{{\alpha,\gamma}} (I)\le |J|^{\gamma-\epsilon_m},
\end{equation}
and if, moreover, $J\cap \widetilde F_m({\alpha,\gamma})\neq\emptyset$, 
\begin{equation} \label{000}
|J|^{\gamma+\epsilon_m}\le \nu_{{\alpha,\gamma}} (J)=\sum_{\substack{J\supset I\in \mathcal{F}_N,\\ I\cap F_m({\alpha,\gamma})\neq\emptyset} }\nu_{{\alpha,\gamma}} (I)\le |J|^{\gamma-\epsilon_m}.
\end{equation}

\noindent (3)  If $J$ is a dyadic cube of generation $n\ge n_m(\alpha,\gamma)$, $J\cap F_m(\alpha,\gamma)\neq\emptyset$, and $N\ge n$, we also have 
\begin{equation}\label{00}
\sum_{\substack{J\supset I\in \mathcal{F}_N,\\ I\cap F_m({\alpha,\gamma})\neq\emptyset} }\mu_{{\alpha,\gamma}}(I) \le  \mu_{\alpha,\gamma}(J)\le |J|^{\alpha-\epsilon_m}.
\end{equation}
Also, \eqref{0} implies
\begin{equation*}
 \#\{J\supset I\in \mathcal F_N:  I\cap  F_m({\alpha,\gamma})\neq\emptyset \}\le |J|^{\gamma-\epsilon_m}2^{N(\gamma+\epsilon_m)},
\end{equation*}
hence
\begin{equation}\label{0''}
 \sum_{\substack{J\supset I\in \mathcal{F}_N,\\ I\cap  F_m({\alpha,\gamma})\neq\emptyset} } \mu_{{\alpha,\gamma}}(I)\le  2^{-n(\gamma-\epsilon_m)}2^{-N(\alpha-\gamma-2\epsilon_m)}.
\end{equation}

}
\end{fact}

\section{Prescription of lower Hausdorff spectra: Proof of Theorems~\ref{thm1}}\label{proof1}

\subsection{Construction of $\mu$ and estimates of its local dimension}\label{constmu1} 

Setting apart some important details omitted in the  outline provided in Section~\ref{introduc}, the proof we present for the general case will be more sophisticated because: (1) $f$ is only upper semi-continuous; (2) $f$ may have more than one fixed point; (3) $\mathrm{dom}(f)$ may contain $\infty$; (4) we will manage that all the sets $\underline E(\mu,\beta)$ are substantially big when they are not empty, this meaning that they contain a Cantor set. Referring to our sketch of proof, since we will use Bernoulli measures for $\mu_{\alpha_i,f(\alpha_i)}$ and $\nu_{\alpha_i,f(\alpha_i)}$, to getting such Cantor sets necessitates to avoid using couples  $(\alpha_i,f(\alpha_i))$ for which $f(\alpha_i)=0$ in the construction; indeed  in this case we know that the Bernoulli measure $\nu_{\alpha_i,f(\alpha_i)}$ is a Dirac mass.  These couples will be replaced by couples $(\alpha_i,\gamma_m(\alpha_i))$, where $0<\gamma_m(\alpha_i)$ tends to $0$ as $m\to\infty$. In particular, to illustrate the case where $f(0)=0$ and $f(\alpha)=-\infty$ for all $\alpha\neq0$, instead of taking $\mu=\delta_{x}$ for some $x\in \R^d$, we will construct a continuous measure on a Cantor set of dimension 0.

\subsubsection{Construction of the measure $\mu$} We will denote $\mathrm{dom}(f)$ by $\mathcal I$. Let $D\in\mathrm{Fix}(f)$. Due to our assumption requiring  the upper semi-continuity of $f$, there exists a dense countable subset $\Delta$ of $\mathcal{I}\setminus \{\infty\}$ such that for all $\alpha\in\mathcal{I}\setminus \{\infty\}$, there exists a sequence $(\alpha_n)_{n\ge 1}$ in $\Delta^{\N_+}$ such that $\lim_{n\to\infty}(\alpha_n,f(\alpha_n))=(\alpha,f(\alpha))$ (if $\infty\in \mathcal{I}$ and $f$ is continuous at $\infty$ the same holds at $(\infty,f(\infty))$, but we will not need this property). This important fact is elementary (see for instance \cite[Lemma 2]{JLV-T} for a proof when $\mathrm{dom}(f)$ is a compact subset of $\R$; the general case considered here then follows immediately). 

We fix once for all such a $\Delta$, and assume without loss of generality that it contains a dense subset of $\mathrm{Fix}(f)$.

Let $\alpha_{\min}=\min (\mathcal{I})$ and $\alpha_{\max}=\max (\mathcal{I})\in {\R}_+\cup\{\infty\}$. 

If $\Delta\setminus\{0,D\}\neq\emptyset$, enumerate its elements  in a sequence $(\alpha^\Delta_j)_{j\ge 1}$ (with necessary redundences if $\Delta\setminus\{0,D\}$ is finite), and for each $m\in \N$ set 
$$
\widetilde A_m=\big \{\alpha^\Delta_j:1\le j\le m,\ \alpha^\Delta_j\ge 4\epsilon_m\big\};
$$
otherwise set $\widetilde A_m=\emptyset$. Also set 
$$
\begin{cases}
D_m= 2\epsilon_m \text{ if } D=0, \text{ and } D_m=D \text{ otherwise},\\
\alpha_m(0)=D_m \text{ if } D=0 \text{ and }\alpha_m(0)=\min(2\epsilon_m,D)/2 \text{ otherwise}, \\
\alpha_m(\infty)=\big (\max (d,m, \max (\alpha^\Delta_j:1\le j\le m)\big )^2\text{ if }\alpha_{\max}=\infty.
\end{cases}
$$
Then let 
$$
A_m=\begin{cases}\widetilde A_m\cup\{D_m\}&\mbox{if } 0<\alpha_{\min}\le \alpha_{\max}<\infty\\
\widetilde A_m\cup\{D_m\}\cup \{\alpha_m(\infty)\}&\mbox{if }0<%
\alpha_{\min}\text{ and } \alpha_{\max}=\infty,\\
\widetilde A_m\cup\{D_m\}\cup \{\alpha_m(0)\}&\mbox{if }
\alpha_{\min}=0\text{ and } \alpha_{\max}<\infty,\\
\widetilde A_m\cup\{D_m\}\cup \{\alpha_m(0)\}\cup\{\alpha_m(\infty)\}&\mbox{if }
\alpha_{\min}=0\text{ and } \alpha_{\max}=\infty,
\end{cases}.
$$ 

For $\alpha\in A_m$ let 
$$
\gamma_m(\alpha)=\begin{cases}
f(\alpha)&\text{if }\alpha\in \widetilde A_m\text{ and }f(\alpha)>0\\
f(\infty)&\text{if }\alpha= \alpha_m(\infty)\text{ and }f(\infty)>0,\\
\alpha &\text{if }\alpha=D_m,\\
 \alpha_m(0)&\text{if }\alpha=  \alpha_m(0),\\
 \epsilon_m &\text{if }\alpha\in\widetilde A_m\text{ and }f(\alpha)=0,\\
\epsilon_m &\text{if }\alpha=\alpha_m(\infty)\text{ and }f(\infty)=0
\end{cases}.
$$
Notice that $\gamma_m(\alpha)\le \alpha$ for all $\alpha\in A_m$. For the values of $\alpha\in A_m\setminus \{0,D\}$ such that $f(\alpha)=0$, we choose $\gamma_m(\alpha)>0$ so that the measure $\nu_{{\alpha,\gamma_m(\alpha)}}$ be continuous but of dimension tending to 0 as $m\to\infty$, and in our construction no level set $\underline E(\mu,\beta)$ be supported on a countable set; indeed, with this choice every non empty such set will contain a Cantor set when $f(\beta)\ge 0$, which is a satisfactory common property. The choice of $\alpha_m(0)$ and $\gamma_m(\alpha_m(0))$, and that of $D_m$ when $D=0$, correspond to the same goal.

Using the definitions of Section~\ref{sec2},  for $\alpha\in A_m$, set 
\begin{equation*}
\mu_{\alpha}=\mu_{{\alpha,\gamma_m(\alpha)}} \text{ and }\nu_{\alpha}=\nu_{{\alpha,\gamma_m(\alpha)}} \text{ for  }\alpha\not\in \mathrm{Fix}(\gamma_m)\text{ and } \mu_{\alpha}=\nu_{\alpha}=\nu_{{\alpha,\alpha}} \text{ for  }\alpha\in \mathrm{Fix}(\gamma_m). 
\end{equation*}
Strickly speaking, $\mu_\alpha$ and $\nu_\alpha$ should be written $\mu_{m,\alpha}$ and $\nu_{m,\alpha}$, but for the sake of readability we will omit the index~$m$. 

Also, let 
$$
n_m=\max\{n_{m}(\alpha,\gamma_m(\alpha)): \alpha\in A_m\}.
$$

Now let $(N_m)_{n\in\N}$ be an increasing sequence of integers defined recursively   satisfying the following properties:
\begin{equation}\label{cond}
\begin{cases}\forall\, m\ge 1,\ N_m\ge n_m,\\
\max ((m +\#A_{m}+\max(A_m))^2,n_m)=o(\sqrt{N_{m-1}}) \mbox{ as }m\to\infty,\\
(\max (\{1\}\cup A_{m-1}))\displaystyle\sum_{i=1}^{m-1}N_i =o\big (\min (\{1\}\cup \gamma_m (A_m)) \sqrt{N_m}\big )\mbox{ as }m\to\infty.
\end{cases}
\end{equation}
For each $m\ge 1$ and $\alpha\in A_m$ set 
$$
G_m(\alpha)=\{I\in \mathcal{F}_{N_m}: I\cap F_{m}(\alpha,\gamma_m(\alpha))\neq\emptyset\}
$$ 
and 
$$
\rho_m(\alpha)=\begin{cases}
1&\text{ if }\alpha=D_m\\
2^{-m}/\#A_m&\text{ otherwise}
\end{cases}.
$$
Due to \eqref{tm} we have 
\begin{equation}\label{controlmuq00}
2^{-1}\le \sum_{I\in G_m(\alpha)}\nu_\alpha(I) \le 1 \quad (\forall\ \alpha\in A_m),
\end{equation}
 so
\begin{equation}\label{controlmuq1}
2^{-1}\le \sum_{I\in G_m(D_m)} \rho_m(D_m)\,(\mu_{D_m}(I)=\nu_{D_m}(I)) \le 1.
\end{equation}
Also,  

\begin{equation}\label{concentration}
\sum_{\alpha\in A_m\setminus  \{D_m\}} \sum_{I\in G_m(\alpha)}\rho_m(\alpha) \mu_{\alpha}(I)\le \sum_{\alpha\in A_m\setminus  \{D_m\}} \frac{2^{-m}}{\#A_m}\|\mu_\alpha\|\le 2^{-m}.
\end{equation}

For each $m\in \N_+$, we enumerate the elements of $A_m$ as $\alpha_{m,1},\ldots, \alpha_{m,\#A_m}$, we denote by  $\mathcal{L}_m$ the set of disjoint closed dyadic cubes $\mathcal{L}(\# A_m)$ defined in Section~\ref{nota}, and denote its elements as $L_{m,\alpha_{m,1}},\ldots, L_{m,\alpha_{m,\#A_m}}$.  We also denote $L_{m,\alpha_{m,i}}$ by  $L_{m,i}$ and $\ell(\#A_m)$ by~$\ell_m$.

We can start the construction of $\mu$. We will construct a Cantor set $K$ by defining recursively the sequence of families of cubes $({\bf G}_m)_{m\in\N}$ such that $K=\bigcap_{m\ge 1}\bigcup_{J\in {\bf G}_m}J$, and simultaneously a consistent sequence  of measures $\mu_m$ supported on $\bigcup_{J\in {\bf G}_m}J$ to get the desired  measure $\mu$ on $K$. 

Let 
$$
{\bf G}_1= \bigcup_{\alpha_{1,i}\in A_1}\{L_{1,i} I_{1,i}:I_{1,i}\in G_1(\alpha_{1,i})\}
$$ 
(the concatenation of intervals has been defined in \eqref{conc}). By construction, the interiors of the elements of  ${\bf G}_1$ are pairwise disjoint. Then define the measure $\mu_1$ is defined on ${\bf G}_1$ by
\begin{equation}\label{mu1}
\mu_1(L_{1,i} I_{1,i})= \frac{\rho_1(\alpha_{1,i})\mu_{{\alpha_{1,i}}}(I_{1,i})}{\displaystyle \sum_{\alpha\in A_1}\sum_{I\in G_1(\alpha)}\rho_1(\alpha)\mu_{{\alpha}}(I)}
\end{equation}
Combining  \eqref{controlmuq1} and  \eqref{concentration} we have 
$$
2^{-1}\le \sum_{\alpha\in A_1}\sum_{I\in G_1(\alpha)}\rho_1(\alpha)\mu_{{\alpha}}(I)\le 1+2^{-1}.
$$
Consequently, \eqref{mu1} yields
\begin{equation*}\label{mu1'}
  \frac{2}{3}\rho_1(\alpha_{1,i}) \mu_{{\alpha_{1,i}}}(I_{1,i})\le \mu_1(L_{1,i}I_{1,i})\le 2\rho_1(\alpha_{1,i})\mu_{{\alpha_{1,i}}}(I_{1,i}).
\end{equation*}

Then,  we define recursively for $m\ge 1$: 
$$
{\bf G}_{m+1}=\bigcup_{I_m\in {\bf G}_m}\bigcup_{\alpha_{m+1,i}\in A_{m+1}}  {\bf G}_{m+1}(I_m,\alpha_{m+1,i}), 
$$
where
$$
{\bf G}_{m+1}(I_m,\alpha_{m+1,i})=\big\{I_mL_{m+1,i} I_{m+1,i}:\, I_{m+1,i}\in G_{m+1}(\alpha_{m+1,i})\big \}
$$
and a measure $\mu_{m+1}$ on ${\bf G}_{m+1}$ by setting
\begin{equation*}\label{mum+1}
\mu_{m+1}(I_m L_{m+1,i} I_{m+1,i})=\mu_m(I_m) \frac{\rho_{m+1}(\alpha_{m+1,i})\mu_{{\alpha_{m+1,i}}}(I_{m+1,i})}{\displaystyle \sum_{\alpha\in A_{m+1}}\sum_{I\in G_{m+1}(\alpha)}\rho_{m+1}(\alpha)\mu_{{\alpha}}(I)},
\end{equation*}
which by construction satisfies
\begin{equation}\label{mum+1'}
  \frac{1}{1+2^{-m}}\le \frac{\mu_{m+1}(I_m L_{m+1,i} I_{m+1,i})}{\mu_m(I_m)\rho_{m+1}(\alpha_{m+1,i}) \mu_{{\alpha_{m+1,i}}}(I_{m+1,i})}\le 2
\end{equation}
by \eqref{controlmuq1} and  \eqref{concentration}.

Each measure $\mu_m$ can be trivially extended to a probability measure on $\mathcal{F}_{g_m}$, where $g_m=-\log_2 |I_m|$, that we still denote by $\mu_m$. This measure yields an absolutely continuous Borel measure on $[0,1]^d$, denoted by $\mu_m$ again, whose density with respect to the Lebesgue measure is given by $2^{d g_m}\mu_m(I)$ over each cube $I\in\mathcal F_{g_m}$. By construction,  the measures $\mu_m$ ($m\in\N$) weakly  converge to a Borel probability measure $ \mu$ on $[0,1]^d$, supported on  the Cantor set $K$ defined as
$$
K=\bigcap_{m\ge 1}\bigcup_{I\in {\bf G}_m}I,
$$
and satisfying $\mu(I)=\mu_m(I)$ for all $m\ge 1$ and $I\in \mathcal F_{g_m}$. Moreover, since $K$ does not intersect the boundary of any dyadic cube due to the definition of the sets $\mathcal{L}_m$, $\mu$ vanishes on  such a boundary.

\subsubsection{Estimates of the local dimension of $\mu$}\label{s-estim}
Let $x\in K$ and $n>g_1=N_1+\log_2 |L_{1,i_1}|^{-1}$. There exists a unique $m\in \N$ such that $g_m<n\le g_{m+1}$. By construction, we have $I_{g_{m}}(x)\in {\bf G}_{m}$ and $I_{g_{m+1}}(x)\in {\bf G}_{m+1}$, and  $I_{g_{m+1}}(x)\subset I_n(x)\subset I_{g_m}(x)$. Moreover, there exist a unique sequence of exponents $\alpha_1(x)\in A_1,\ldots,\alpha_m(x)\in A_m,\alpha_{m+1}(x)\in A_{m+1}$, and a unique sequence of pair of intervals $\{L_j,I_j\}_{1\le j\le m+1}$ such that 
\begin{equation}\label{Igm}
I_{g_m}(x)=L_{1}I_1\cdots L_{m}I_m\quad\text{and}\quad  I_{g_{m+1}}(x)=L_{1}I_1\cdots L_{m}I_mL_{m+1}I_{m+1},
\end{equation} 
with $I_{j}\in {G}_j(\alpha_j(x))$ and $L_j\in\mathcal{L}_j$  for each $1\le j\le m+1$. The intervals invoked in \eqref{Igm} will be used in the following statements.

\begin{pro}\label{localestimate}With the notations introduced above, there exists a positive sequence $(\delta_n)_{n\ge 1}$ converging to $0$ as $n\to\infty$ such that, uniformly in $x\in K$, for $g_m\le n\le g_{m+1}$ one has   
\begin{enumerate}
\item 
\begin{equation}\label{goodcontrols1}
\frac{  \mu(I_n(x))}{\displaystyle2^{-g_m \alpha_m(x)-(n- g_m)\alpha_{m+1}(x)}}\le 2^{n\delta_n};
\end{equation}
\item  if either $n=g_m$, or  $\alpha_{m+1}(x)\in \mathrm{Fix}(\gamma_{m+1})$ and $I_{m+1}\cap \widetilde F_{m+1}(\alpha_{m+1}(x),\alpha_{m+1}(x))\neq\emptyset$, then
\begin{equation}\label{goodcontrols2}
2^{-n\delta_n}\le \frac{  \mu(I_n(x))}{2^{-g_m \alpha_m(x)-(n- g_m)\alpha_{m+1}(x)}}\le 2^{n\delta_n},
\end{equation}
\item  if $g_m+\ell_{m+1}+n_{m+1}<n\le g_{m+1}$ then 
\begin{equation}\label{goodcontrols3}
\frac{  \mu(I_n(x))}{2^{-g_m \alpha_m(x)-(n- g_m)\gamma_{m+1}(\alpha_{m+1}(x))-N_{m+1}(\alpha_{m+1}(x)-\gamma_{m+1}(\alpha_{m+1}(x))}}\le 2^{n\delta_n+2N_{m+1}\epsilon_{m+1}}.
\end{equation}
\end{enumerate}
\end{pro}
\begin{cor}\label{dimloc}
For all $x\in K$ we have 
\begin{enumerate}
\item $\displaystyle \underline d(\mu,x)=\liminf_{m\to\infty}\alpha_m(x)$. 

\item   If for $m$ large enough we have $\alpha_m(x)\in \mathrm{Fix}(\gamma_m)$ and $I_m\cap \widetilde F_m(\alpha_m(x),\alpha_m(x))\neq\emptyset$,  then $\overline d(\mu,x)= \limsup_{m\to\infty}\alpha_m(x)$.

\item  If $\liminf_{m\to\infty}\alpha_m(x)\in \mathcal{I}\setminus \mathrm{Fix}(f)$  then  $\overline d(\mu,x)=\infty$. In particular, if $\alpha\in\mathcal{I}\setminus \mathrm{Fix}(f)$, then $\underline E(\mu,\alpha)=E(\mu,\alpha,\infty)$. 
\end{enumerate}
\end{cor}

\noindent
{\it Proof of Proposition~\ref{localestimate}.} We will write $\alpha_j$ for $\alpha_j(x)$. 

At first we notice that by construction, and due to \eqref{mum+1'} we have 
\begin{equation*}\label{e4-0}
 c_m \prod_{j=1}^m \mu_{{\alpha_j}}(I_j)\le \mu(I_{g_m(x)})\le 2^m \prod_{j=1}^m \mu_{{\alpha_j}}(I_j),
\end{equation*}
where $c_m=\prod_{j=1}^m (1+2^{-m})^{-1}\prod_{j=1}^m \rho_j(\alpha_j)\ge e^{-1} \prod_{j=1}^m \rho_j(\alpha_j)$. Then, due to the definition of $G_j(\alpha_j)$
\begin{equation}\label{e4}
c_m\exp \Big (-\sum_{j=1}^m (\alpha_j+\epsilon_j)N_j\log(2)\Big )\le \mu (I_{g_m}(x))\le 2^m\exp \Big (-\sum_{j=1}^m (\alpha_j-\epsilon_j)N_j\log(2)\Big ).
\end{equation}

Proof of (1) and (2): We distinguish two cases. Let $g'_m=g_m+\ell_{m+1}$.

{\bf Case 1:} $g_m<n\le g'_m+n_{m+1}$.  Write $I_{g'_m+n_{m+1}}(x)=I_{g_m(x)}L_{m+1} J_{n_{m+1}}$, where $J_{n_{m+1}}=I_{n_{m+1}}(\{2^{g'_m}x\})$, $\{t\}$ standing for the vector whose entries are the fractional part of the entries of $t$. We have $J_{n_{m+1}}\supset I_{m+1}$, $I_{m+1}\cap F_{m+1}(\alpha_{m+1},\gamma_{m+1}(\alpha))\neq\emptyset$, and the generation of $J_{n_{m+1}}$ is $n_{m+1}\ge n_{m+1}(\alpha_{m+1},\alpha_{m+1})$. 

We obviously have $\mu(I_n(x))\le \mu(I_{g_m}(x))$, and by construction if $\alpha_{m+1}\in \mathrm{Fix}(\gamma_{m+1})$ (remembering that $\mu_{\alpha_{m+1}}=\nu_{\alpha_{m+1}}$ and using the equality in \eqref{0}), 
\begin{eqnarray*}
\mu(I_{g_m}(x))&\ge& \mu(I_n(x))\ge \mu(I_{g'_m+n_{m+1}}(x))= \sum_{J_{n_{m+1}}\supset I\in G_{m+1}(\alpha_{m+1})}\mu (I_{g_m}(x)L_{m}I)\\
&=&\mu(I_{g_m}(x)) \frac{\displaystyle\sum_{J_{n_{m+1}}\supset I\in G_{m+1}(\alpha_{m+1})}\rho_{m+1}(\alpha_{m+1})\mu_{\alpha_{m+1}} (I)}{\displaystyle\sum_{\alpha\in A_{m+1} }\sum_{I\in G_{m+1}(\alpha)}\rho_{m+1}(\alpha)\mu_{{\alpha}} (I)}\\
&=&\frac{\rho_{m+1}(\alpha_{m+1})}{\displaystyle\sum_{\alpha\in A_{m+1} }\sum_{I\in G_{m+1}(\alpha)}\rho_{m+1}(\alpha)\mu_{{\alpha}} (I)}\mu(I_{g_m}(x)) \nu_{\alpha_{m+1}}(J_{n_{m+1}})\\
&\ge &\frac{\rho_{m+1}(\alpha_{m+1})}{1+2^{-(m+1)}}\mu(I_{g_m}(x))\nu_{\alpha_{m+1}}(J_{n_{m+1}})\quad (\text{we have used \eqref{controlmuq1} and  \eqref{concentration} again}).
\end{eqnarray*}
Combining this with \eqref{e4},  if $\alpha_{m+1}\in\mathrm{Fix}(\gamma_{m+1})$, we thus get 
 $$
 c_{m+1}\widetilde c^{-1}_{m}\nu_{\alpha_{m+1}}(J_{n_{m+1}})   \le \frac{\mu(I_n(x))}{\exp \Big (-\displaystyle\sum_{j=1}^m \alpha_j N_j\log(2)\Big )}\le \widetilde c_{m}2^m,
 $$ 
with $\widetilde c_m=\exp(\sum_{j=1}^mN_j\epsilon_j\log(2))$. If, moreover, $I_{m+1}\cap \widetilde F_{m+1}(\alpha_{m+1},\alpha_{m+1})\neq\emptyset$,  then $J_{n_{m+1}}\cap \widetilde F_{m+1}(\alpha_{m+1},\alpha_{m+1})\neq\emptyset$,   so due to \eqref{000}  we have 
$$
 c_{m+1}\widetilde c^{-1}_{m}2^{-n_{m+1}(\alpha_{m+1}+\epsilon_{m+1})}   \le \frac{\mu(I_n(x))}{\exp \Big (-\displaystyle\sum_{j=1}^m \alpha_j N_j\log(2)\Big )}\le \widetilde c_{m}2^m,
 $$ 
which finally yields 
\begin{equation*}
C_m^{-1}\le \frac{ \mu (I_n(x))}{\exp \Big (-\alpha_{m+1}(n-g_m)\log(2)-\displaystyle\sum_{j=1}^m \alpha_jN_j\log(2)\Big )}\le  C_m,
\end{equation*}
with $ C_m= c_{m+1}^{-1}\widetilde c_{m} 2^m 2^{2(\ell_{m+1} +n_{m+1})(\max(A_{m+1})+\epsilon_{m+1})}$. Moreover,  it is readily seen that the previous upper bound holds whatever be $\alpha_{m+1}$.

Now, due to the conditions \eqref{cond} we have imposed to the sequence $N_m$ and the definition of $\rho_j(\alpha_j)$, we have $\log(\widetilde C_m)=o(g_m)$ and $ \sup\Big\{\sum_{j=1}^{m-1}\alpha_j N_j:(\alpha_j)_{1\le j\le m-1}\in\prod_{j=1}^{m-1}A_j\Big\} =o(\min(A_m) g_m)$. Consequently, there exists a sequence $(\delta_n)_{n\in\N}$ converging to $0$ as $n\to\infty$, such that 
\eqref{goodcontrols1} and \eqref{goodcontrols2} hold 
uniformly in $x\in K$ and $g_m<n\le g'_m+n_{m+1}$. 

\medskip

{\bf Case 2:} $ g'_m+n_{m+1}<n\le g_{m+1}$. Write $I_n(x)=I_{g_m}(x) L_{m+1} J_{n-g'_m}$, where $J_{n-g'_m}=I_{n-g'_m}(\{2^{g'_m}x\})$. We have $J_{n-g'_m}\supset I_{m+1}$, $I_{m+1}\cap F_{m+1}(\alpha_{m+1},\gamma_{m+1}(\alpha))\neq\emptyset$, and the generation of $J_{n-g'_m}$ is $n-g'_m\ge n_{m+1}\ge  n_{m+1}(\alpha_{m+1},\alpha_{m+1})$. 

By construction:
\begin{eqnarray*}
\mu(I_n(x))&=&\sum_{J_{n-g'_m}\supset I\in G_{m+1}(\alpha_{m+1})}\mu (I_{g_m}(x)L_{m}I)\\
&=&\mu(I_{g_m}(x)) \frac{\displaystyle\sum_{J_{n-g'_m}\supset I\in G_{m+1}(\alpha_{m+1})}\rho_{m+1}(\alpha_{m+1})\mu_{\alpha_{m+1}} (I)}{\displaystyle\sum_{\alpha\in A_{m+1}} \sum_{I\in G_{m+1}(\alpha)}\rho_{m+1}(\alpha)\mu_{{\alpha}} (I)}\\
&\le &\frac{\rho_{m+1}(\alpha_{m+1})}{\displaystyle\sum_{\alpha\in A_{m+1}} \sum_{I\in G_{m+1}(\alpha)}\rho_{m+1}(\alpha)\mu_{{\alpha}} (I)}\mu(I_{g_m}(x))\mu_{\alpha_{m+1}}(J_{n-g'_m}) \quad \text{using \eqref{00}},
\end{eqnarray*}
with equality if $\alpha_{m+1}\in \mathrm{Fix}(\gamma_{m+1})$, remembering that in this case $\mu_{\alpha_{m+1}}=\nu_{\alpha_{m+1}}$, and using the equality in \eqref{0}. 
Consequently,
$$
  \mu(I_n(x))\le 2 \rho_{m+1}(\alpha_{m+1})\mu(I_{g_m}(x)) \mu_{\alpha_{m+1}}(J_{n-g'_m}),
 $$
and
$$
 \frac{\rho_{m+1}(\alpha_{m+1})}{1+2^{-(m+1)}} \mu(I_{g_m}(x)) \nu_{{\alpha_{m+1}}}(J_{n-g'_m})\le  \mu(I_n(x))\le 2 \rho_{m+1}(\alpha_{m+1})\mu(I_{g_m}(x)) \nu_{\alpha_{m+1}}(J_{n-g'_m})
 $$
if $\alpha_{m+1}\in \mathrm{Fix}(\gamma_{m+1})$ (we have used \eqref{controlmuq1} and  \eqref{concentration} again). Set  $C_{m,n}=\widetilde C_m 2^{(n-g'_m)\epsilon_{m+1}}$. The previous estimates combined with \eqref{e4} and the estimates  \eqref{00} and \eqref{000} of $\mu_{\alpha_{m+1}}(J_{n-g'_m})$ and $\nu_{\alpha_{m+1}}(J_{n-g'_m})$ respectively, yield:
\begin{equation*}
\frac{ \mu (I_n(x))}{\exp \Big (\displaystyle-\alpha_{m+1}(n-g_m)\log(2)-\sum_{j=1}^m \alpha_jN_j\log(2)\Big )}\le  C_{m,n},
\end{equation*}
and 
\begin{equation*}\label{e8}
C_{m,n}^{-1}\le\frac{ \mu (I_n(x))}{\displaystyle\exp \Big (-\alpha_{m+1}(n-g_m)\log(2)-\sum_{j=1}^m \alpha_jN_j\log(2)\Big )}\le  C_{m,n}
\end{equation*}
if $\alpha_{m+1}\in\mathrm{Fix}(\gamma_{m+1})$ and $I_{m+1} \cap\widetilde F_{m+1}(\alpha_{m+1},\alpha_{m+1})\neq\emptyset$. 
Then, due to \eqref{cond} again, the above sequence $(\delta_n)_{n\in\N}$ can be modified so that \eqref{goodcontrols1} also holds 
uniformly in $x\in K$, and $ g'_m+n_{m+1}<n\le g_{m+1}$.

\medskip

Proof of (3): it suffices to use \eqref{0''} instead of \eqref{00} to estimate $\mu_{\alpha_{m+1}}(J_{n-g'_m})$ in the previous upper bounds for $\mu (I_n(x))$.  $\square$

\medskip

\noindent
{\it Proof of Corollary~\ref{dimloc}.} (1) and (2) follow readily from \eqref{goodcontrols1} and \eqref{goodcontrols2}, and the fact that by construction the neighboring dyadic cubes of generation $n$ of $I_n(x)$ have a $\mu$-mass equal to 0 or for which the estimates of Proposition~\ref{localestimate}(1)(2) also hold. 

For (3), suppose that $\alpha=\liminf_{m\to\infty}\alpha_m(x)\in \mathcal{I}\setminus \mathrm{Fix}(f)$. Since $\mathrm{Fix}(f)$ is closed (because $f(\alpha)\le \alpha$ and $f$ is upper semi-continuous), there exists a subsequence $(m_k)_{k\ge 1}$ such that  $\alpha_{m_k+1}(x)$ converges to $\alpha$ and $\gamma_{m_k+1}(\alpha_{m_k+1}(x))$ converges to $f(\alpha)<\alpha)$. Take $n=n(k)=g_{m_k}+\ell_{m_k+1}+n_{m_k+1}+\sqrt{N_{m_k+1}}$. By construction $g_{m_k}=o(n)$ and $n-g_m\sim n\sim \sqrt{N_{m_k+1}}$, and using \eqref{goodcontrols3} we have 
$$
\frac{  \mu(I_n(x))}{2^{o(n)-n\gamma_{m_k+1}(\alpha_{m_k+1}(x))-n^2(\alpha_{m_k+1}(x)-\gamma_{m_k+1}(\alpha_{m_k+1}(x))}}\le 2^{n\delta_n+2n^2\epsilon_{m_k+1}},
$$
so 
$$
\frac{  \mu(I_n(x))}{-n\log(2)}\ge \frac{\gamma_{m_k+1}(\alpha_{m_k+1}(x))}{\log(2)}+n\frac{\alpha_{m_k+1}(x)-\gamma_{m_k+1}(\alpha_{m_k+1}(x))}{\log(2)}+o(n).
$$
Letting $k$ tend to $\infty$ yields the desired conclusion on $\overline d(\mu,x)$, again because the neighboring dyadic cubes of generation $n$ of $I_n(x)$ have a $\mu$-mass equal to 0 or with the same behavior as $\mu(I_n(x))$. Now, if $\alpha\in\mathcal{I}\setminus \mathrm{Fix}(f)$ and $x\in \underline E(\mu,\alpha)$, then $(\alpha_m(x))_{m\ge 1}$ must take infinitely many values in  $\mathcal{I}\setminus \mathrm{Fix}(f)$ since $ \mathrm{Fix}(f)$ is a closed set, and consequently $ \underline E(\mu,\alpha)\subset E(\mu,\alpha,\infty)$.

\subsection{Auxiliary measures and lower bound for the lower Hausdorff spectrum}
\subsubsection{Construction of auxiliary measures}

Let $\widehat \alpha=(\alpha_m)_{m\ge 1}\in \prod_{m=1}^\infty A_m$.

Now, we construct a measure $ \nu_{\widehat\alpha}$ as follows: Let 
$$
{\bf G}_{\widehat\alpha,1}=\{L_{1,\alpha_1}I_1: I_1\in G_1(\alpha_1)\}, 
$$
and define on ${\bf G}_{\widehat\alpha,1}$ the measure 
\begin{equation}\label{mu1beta}
 \nu_{\widehat\alpha,1}(L_{1,\alpha_1} I_1)= \frac{\nu_{{\alpha_1}}(I_1)}{\displaystyle\sum_{I\in G_1(\alpha_1)}\nu_{{\alpha_1}}(I)}.
\end{equation}
Due to \eqref{controlmuq00}, \eqref{mu1beta} yields
\begin{equation*}\label{mu1'beta}
  \nu_{{\alpha_1}}(I_1)\le  \nu_{\widehat\alpha,1}(L_{1,\alpha_1}I_1)\le 2\nu_{{\alpha_1}}(I_1).
\end{equation*}

Then, recursively we define for $m\ge 1$ 
$$
{\bf G}_{\widehat\alpha,m+1}=\bigcup_{I_m\in {\bf G}_{\widehat \alpha,m}}{\bf G}_{\widehat\alpha,m+1}(I_m,\alpha_{m+1}), 
$$
where
$$
{\bf G}_{\widehat\alpha,m+1}(I_m,\alpha_{m+1})=\{I_mL_{m+1,\alpha_{m+1}} I_{m+1}:\, I_{m+1}\in G_{m+1}(\alpha_{m+1})\}
$$
and a measure $ \nu_{\widehat\alpha,m+1}$ on   $ {\bf G}_{\widehat\alpha,m+1}$ by
\begin{equation*}\label{mum+1beta}
 \nu_{\widehat\alpha,m+1}(I_m L_{m+1,\alpha_{m+1}} I_{m+1})= \nu_{\widehat\alpha,m}(I_m) \frac{\nu_{{\alpha_{m+1}}}(I_{m+1})}{\displaystyle\sum_{I\in G_{m+1}(\alpha_{m+1})}\nu_{{\alpha_{m+1}}}(I)}.
\end{equation*}
Due to  \eqref{controlmuq00} we have 
\begin{equation}\label{mum+1'beta}
1 \le \frac{ \nu_{\widehat\alpha,m+1}(I_m L_{m+1,\alpha_{m+1}} I_{m+1})}{  \nu_{\widehat\alpha,m}(I_m)  \nu_{{\alpha_{m+1}}}(I_{m+1})}\le 2 .
\end{equation}
Each measure $ \nu_{\widehat\alpha,m}$ can be trivially extended into a probability measure on $\mathcal{F}_{g_m}$;  we still denote this measure by $ \nu_{\widehat\alpha,m}$. This measure yields an absolutely continuous Borel measure on $[0,1]^d$, denote by $ \nu_{\widehat\alpha,m}$ again, whose density with respect to the Lebesgue measure is given by $2^{d g_m}\nu_{\widehat\alpha,m}(I)$ over each cube $I\in\mathcal F_{g_m}$. By construction,  the measures $ \nu_{\widehat\alpha,m}$ ($m\in\N$) converge weakly  to a Borel probability measure $ \nu_{\widehat\alpha}$ on $[0,1]^d$, supported on  the Cantor set $K_{\widehat \alpha}$ defined as
$$
K\supset K_{\widehat \alpha}=\bigcap_{m\ge 1}\bigcup_{J\in {\bf G}_{\widehat\alpha,m}}J,
$$
and satisfying $\nu_{\widehat\alpha}(I)=\nu_{\widehat\alpha,m}(I)$ for all $m\ge 1$ and $I\in \mathcal F_{g_m}$.

\noindent{\it Estimation of the local dimension of $ \nu_{\widehat\alpha}$.}
We use the same notations as in Section~\ref{s-estim}. 
Let $x\in K_{\widehat \alpha}$ and $n>g_1=N_1+\ell_1$. There exists a unique $m\in \N$ such that $g_m<n\le g_{m+1}$. By construction, we have $I_{g_{m}}(x)\in {\bf G}_{\widehat\alpha,m}$ and $I_{g_{m+1}}(x)\in {\bf G}_{\widehat\alpha,m+1}$, and  $I_{g_{m+1}}(x)\subset I_n(x)\subset I_{g_m}(x)$. Moreover, we have \eqref{Igm}. 

\begin{pro}\label{localestimate2}There exists a positive sequence $(\delta_n)_{n\ge 1}$ converging to $0$ as $n\to\infty$ such that, for $\nu_{\widehat\alpha}$-almost every $x\in K_{\widehat \alpha}$, for $n$ large enough, 
\begin{equation}\label{goodcontbeta}
2^{-n\delta_n} \le \frac{ \nu_{\widehat\alpha}(I_n(x))}{ 2^{-g_m\gamma_m (\alpha_m)-(n- g_m)\gamma_{m+1}(\alpha_{m+1})}}\le 2^{n\delta_n}.
\end{equation}
\end{pro}
The previous proposition, the fact that by construction the neighboring dyadic cubes of generation $n$ of $I_n(x)$ have a $\nu_{\widehat\alpha}$-mass equal to 0 or for which the estimates \eqref{goodcontbeta} hold,  and the mass distribution principle (Section~\ref{MDP}) yield the Hausdorff and packing dimensions of~$ \nu_{\widehat\alpha}$:
\begin{cor}\label{diminf}
We have $\underline d( \nu_{\widehat\alpha},x)= \liminf_{m\to\infty}\gamma_m(\alpha_m)$ and $\overline d( \nu_{\widehat\alpha},x)= \limsup_{m\to\infty}\gamma_m (\alpha_m)$ for $\nu_{\widehat\alpha}$-almost every $x$. Consequently, $\dim_H (\nu_{\widehat\alpha})=\liminf_{m\to\infty}\gamma(\alpha_m)$ and $\dim_P (\nu_{\widehat\alpha})=\limsup_{m\to\infty}\gamma(\alpha_m)$.
\end{cor}
\noindent{\it Proof of Proposition~\ref{localestimate2}.} At first, we prove the following fact: there exists  $\widetilde K_{\widehat \alpha}\subset K_{\widehat \alpha}$, of full  $\nu_{\widehat\alpha}$-measure, such that for all  $x\in \widetilde K_{\widehat \alpha}$, for $m$ large enough, we have $I_m\cap \widetilde F_m(\alpha_m,\gamma_m(\alpha_m))\neq\emptyset$, with $I_m$ defined as in \eqref{Igm}. 

Indeed, due to the multiplicative structure of $\nu_{\widehat\alpha}$, by construction we have  
\begin{eqnarray*}
&&\nu_{\widehat\alpha}(\{x\in K_{\widehat \alpha}: I_{m} \cap \widetilde F_{m}(\alpha_{m},\gamma_m(\alpha_{m}))=\emptyset\})\\
&=& \sum_{I\in \mathbf{G}_{\widehat\alpha,m-1}} \nu_{\widehat\alpha,m-1}(I) \frac{\displaystyle\sum_{J\in \mathcal{F}_{N_{m}},\, J\cap \widetilde F_{m}(\alpha_{m},\gamma_m(\alpha_{m}))=\emptyset}\nu_{\alpha_{m}}(J)}{\displaystyle\sum_{J\in \mathcal{F}_{N_{m}},\, J\cap F_{m}(\alpha_{m},\gamma_m(\alpha_{m}))\neq\emptyset}\nu_{{\alpha_{m}}}(J)}\\
&\le & \sum_{I\in \mathbf{G}_{\widehat\alpha,m-1}} \nu_{\widehat\alpha,m-1}(I) \frac{2^{-m}}{2^{-1}}= 2^{-(m-1)},
\end{eqnarray*}
where we have used the left hand side of \eqref{controlmuq00}, and \eqref{0000}. Then, by the Borel-Cantelli lemma, we have the desired conclusion. 
 
Now the proof of the proposition  follows from   lines similar  to those used to prove Proposition~\ref{localestimate}: 

Let $x\in \widetilde K_{\widehat \alpha}$. At first we notice that by construction, and due to \eqref{mum+1'beta} we have 
\begin{equation}\label{e4beta}
\prod_{j=1}^m \nu_{{\alpha_j}}(I_j)\le  \nu_{\widehat\alpha}(I_{g_m(x)})\le 2^{m} \prod_{j=1}^m \nu_{{\alpha_j}}(I_j).
\end{equation}
Then, due to the definition of $G_j(\alpha_j)$, 
\begin{equation*}\label{e4''}
\widetilde c_m^{-1}\le\frac{  \nu_{\widehat\alpha} (I_{g_m}(x))}{\exp \Big (-\sum_{j=1}^m \gamma_j(\alpha_j)N_j\log(2)\Big )}\le  2^m\widetilde c_m
\end{equation*}
where $\widetilde c_m$ is defined as in the proof of Proposition~\ref{localestimate}.

We distinguish two cases.

{\bf Case 1:} $g_m<n\le g'_m+n_{m+1}$.  Write $I_{g'_m+n_{m+1}}(x)=I_{g_m(x)}L_{m+1} J_{n_{m+1}}$, where $J_{n_{m+1}}=I_{n_{m+1}}(\{2^{g'_m}x\})$. 
We have 
\begin{eqnarray*}
\nu_{\widehat\alpha}(I_{g_m}(x))&\ge& \nu_{\widehat\alpha}(I_n(x))\ge \nu_{\widehat\alpha}(I_{g'_m+n_{m+1}}(x))= \sum_{J_{n_{m+1}}\supset I\in G_{m+1}(\alpha_{m+1})}\nu_{\widehat\alpha} (I_{g_m}(x)L_{m}I)\\
&=&\nu_{\widehat\alpha}(I_{g_m}(x)) \frac{\displaystyle\sum_{J_{n_{m+1}}\supset I\in G_{m+1}(\alpha_{m+1})}\nu_{\alpha_{m+1}} (I)}{\displaystyle\sum_{I\in G_{m+1}(\alpha_{m+1})}\nu_{\alpha_{m+1}} (I)}\\
& =&\frac{\nu_{\widehat\alpha}(I_{g_m}(x))\nu_{\alpha_{m+1}}(J_{n_{m+1}})}{\displaystyle\sum_{I\in G_{m+1}(\alpha_{m+1})}\nu_{\alpha_{m+1}} (I)} \ge \nu_{\widehat\alpha}(I_{g_m}(x))\nu_{\alpha_{m+1}}(J_{n_{m+1}}),
\end{eqnarray*}
where we have used the right hand side of \eqref{controlmuq00}. For $m$ large enough so that $I_{m+1}\cap\widetilde F_{m+1}(\alpha_{m+1},\gamma_{m+1}(\alpha_{m+1}))\neq\emptyset$ we have $J_{n_{m+1}}\cap\widetilde F_{m+1}(\alpha_{m+1},\gamma_{m+1}(\alpha_{m+1}))\neq\emptyset$ and the generation of $J_{n_{m+1}}$ is $n_{m+1}\ge n(\alpha_{m+1},\gamma_{m+1}(\alpha_{m+1}))$, so \eqref{000} holds for $J_{n_{m+1}}$, and this combined with  \eqref{e4beta} and \eqref{000} yields
 $$
\widetilde c_{m}^{-1} 2^{-n_{m+1}(\gamma_{m+1}(\alpha_{m+1})+\epsilon_{m+1})}  \le  \frac{\nu_{\widehat\alpha} (I_n(x))}{\exp \Big (-\sum_{j=1}^m \gamma_j(\alpha_j)N_j\log(2)\Big )}\le 2^m \widetilde c_m.
 $$ 
Consequently, 
\begin{equation}\label{e6}
 \widetilde C_m\le \frac{ \nu_{\widehat\alpha} (I_n(x))}{\exp \Big (-\gamma_{m+1}(\alpha_{m+1})(n-g_m)\log(2)-\sum_{j=1}^m \gamma_j(\alpha_j)N_j\log(2)\Big )}\le  \widetilde C_m,
\end{equation}
with $\widetilde C_m= \widetilde c_{m} 2^{m}  2^{2(\ell_{m+1} +n_{m+1})(\max(\gamma_{m+1}(A_{m+1}))+\epsilon_{m+1})}$. 

Due to the conditions \eqref{cond} we have imposed to the sequence $(N_m)_{m\ge 1}$, we have $\log(\widetilde C_m)=o(g_m)$ and $ \sup\Big\{\sum_{j=1}^{m-1}\gamma_j(\alpha_j) N_j:(\alpha_j)_{1\le j\le m-1}\in\prod_{j=1}^{m-1}A_j\Big\} =o(\min(\gamma_m(A_m)) g_m)$. Consequently, there exists a sequence $(\delta_n)_{n\in\N}$ converging to $0$ as $n\to\infty$, such that \eqref{goodcontbeta} holds 
for all $x\in \widetilde K_{\widehat \alpha}$, for $m$ large enough and  and $g_m<n\le g'_m+n_{m+1}$. 

\medskip

{\bf Case 2:} $ g'_m+n_{m+1}<n\le g_{m+1}$. Write $I_n(x)=I_{g_m}(x) L_{m+1} J_{n-g'_m}$, where $J_{n-g'_m}=I_{n-g'_m}(\{2^{g'_m}x\})$. We have  
\begin{eqnarray*}
 \nu_{\widehat\alpha}(I_n(x))&=&\sum_{J_{n-g'_m}\supset I\in G_{m+1}(\alpha_{m+1})} \nu_{\widehat\alpha} (I_{g_m}(x)L_{m}I)\\
&=& \nu_{\widehat\alpha}(I_{g_m}(x)) \frac{\displaystyle \sum_{J_{n-g'_m}\supset I\in G_{m+1}(\alpha_{m+1})}\nu_{{\alpha_{m+1}}} (I)}{\displaystyle\sum_{I\in G_{m+1}(\alpha_{m+1})}\nu_{{\alpha_{m+1}}} (I)},
\end{eqnarray*}
hence
$$
 \nu_{\widehat\alpha}(I_{g_m}(x)) \nu_{\alpha_{m+1}}(J_{n-g'_m})\le  \nu_{\widehat\alpha}(I_n(x))\le 2 \nu_{\widehat\alpha}(I_{g_m}(x)) \nu_{\alpha_{m+1}}(J_{n-g'_m})
$$
(we have used \eqref{controlmuq00}). For $m$ large enough we have $I_{m+1}\cap \widetilde F_{m+1}(\alpha_{m+1},\gamma_{m+1}(\alpha_{m+1}))\neq\emptyset$. This implies $J_{n-g'_m}\cap \widetilde F_{m+1}(\alpha_{m+1},\alpha_{m+1})\neq\emptyset$. Since, moreover,  $n-g'_m\ge  n(\alpha_{m+1},\gamma_{m+1}(\alpha_{m+1}))$, \eqref{000} holds for $J _{n-g'_m}$ and the previous estimates  combined with \eqref{e4beta}   yield
\begin{equation*}\label{e7}
\widetilde C_{m,n}^{-1}\le \frac{  \nu_{\widehat\alpha} (I_n(x))}{\exp \Big (-\gamma_{m+1}(\alpha_{m+1})(n-g_m)\log(2)-\sum_{j=1}^m \gamma_j(\alpha_j)N_j\log(2)\Big )}\le  \widetilde C_{m,n},
\end{equation*}
where 
$$
\widetilde C_{m,n}=\widetilde C_{m}2^{\max(\gamma_{m+1}(A_{m+1})) (g'_m-g_m)}2^{(n-g'_m)\epsilon_{m+1}} .
$$
Then, due to \eqref{cond}, the above sequence $(\delta_n)_{n\in\N}$ can be modified so that \eqref{goodcontrols1} also holds 
for all $x\in \widetilde K_{\widehat \alpha}$, for $m$ large enough and   $ g'_m+n_m<n\le g_{m+1}$. $\square$

\subsubsection{Lower bound for the lower Hausdorff spectrum}
\begin{pro}
For any closed ball $B$ whose interior intersects $K=\supp(\mu)$, we have $\dim_H (B\cap \underline E(\mu,\alpha))\ge f(\alpha)$ for all $\alpha\in\mathcal{I}$.
\end{pro}
\begin{proof}
Fix $B$, a closed ball whose interior intersects $K$. There exists $m_0\in\N$, a sequence $(\alpha_{j})_{1\le j\le m_0}\in\prod_{j=1}^{m_0}A_j$, as well as pairs of dyadic cubes $(\{L_{j},  {I}_{\alpha_{j}}\}_{1\le j\le m_0}$, with $L_j\in\mathcal{L}_j$ and $ {I}_{\alpha_{j}}\in G_j(\alpha_j)$ such that 
$$
L_1 {I}_{\alpha_{1}}\cdots L_{m_0} {I}_{\alpha_{m_0}}\subset B.
$$

Now fix $\alpha\in\mathcal{I}$. 

If $\alpha\in \mathcal{I}\setminus\{0,D,\infty\}$ then for each $m> m_0$, fix $\alpha_m\in A_m\setminus\{D_m\}$, so that $\lim_{m\to\infty} (\gamma_m(\alpha_m)=f(\alpha_m))=f(\alpha)$. If, moreover, $\alpha\in \mathrm{Fix}(f)\setminus D$, take $\alpha_m\in  \mathrm{Fix}(f)$. 

If $\alpha=0$ then for each $m> m_0$ let $\alpha_m=\alpha_m(0)$. We have  $\lim_{m\to\infty} (\gamma_m(\alpha_m)=\alpha_m)=0=f(\alpha)$. 

If $\alpha=D$  then for each $m>m_0$ let $\alpha_m=D_m$. We have $\lim_{m\to\infty} (\gamma_m(\alpha_m)=\alpha_m)=\alpha=f(\alpha)$. 

If $\alpha=\infty$ then for each $m>m_0$ let $\alpha_m=\alpha_m(\infty)$. We have $\gamma_m(\alpha_m)=f(\infty)$ if $f(\infty)>0$ and $\lim_{m\to\infty}(\gamma_m(\alpha_m)=\epsilon_m)=0=f(\infty)$ otherwise. 

Let $\widehat \alpha=(\alpha_m)_{m\ge 1}$, and consider the measure $\nu_{\widehat \alpha}$ constructed in the previous section. This measure is supported on the set $\widetilde K_{\widehat \alpha}\subset K$ exhibited at the beginning of  Proposition~\ref{localestimate2} proof, and by construction $\nu_{\widehat \alpha}(B\cap \widetilde K_{\widehat \alpha} )>0$, so due to Corollary~\ref{diminf} we have $\dim_H\widetilde  K_{\widehat \alpha}\cap B\ge f(\alpha)$. 

Moreover, due to Corollary~\ref{dimloc}(1), we have $\widetilde  K_{\widehat \alpha}\subset \underline E(\mu,\alpha)$, so $\dim_H (B\cap \underline E(\mu,\alpha))\ge f(\alpha)$. If, moreover, $\alpha\in  \mathrm{Fix}(f)$, then by Corollary~\ref{dimloc}(2), our choice of $(\alpha_m)_{m\ge 1}$ implies that $\widetilde K_{\widehat \alpha}\subset  E(\mu,\alpha)$, so $\dim_H (B\cap E(\mu,\alpha))=f(\alpha)$. 
\end{proof}

\subsection{Upper bound for the lower Hausdorff spectrum}
Recall that $f$ is upper semi-continuous. Also, its domain, denoted by $\mathcal{I}$ is a closed subset of $[0,\infty]$, so due to Corollary~\ref{dimloc}(1), if $\alpha\not\in \mathcal I$ then  $\underline E(\mu,\alpha)=\emptyset$, hence $\dim_H\underline E(\mu,\alpha)=-\infty=f(\alpha)$.

Let $\alpha\in\mathcal{I}\setminus\{\infty\}$. Fix $\eta>0$. Then choose $\delta>0$ such that $f(\beta)\le f(\alpha)+\eta$ if $|\beta-\alpha|\le \delta$. Due to Corollary~\ref{dimloc}(1), $\underline E(\mu,\alpha)=\{x\in K: \liminf_{m\to\infty} \alpha_m(x)=\alpha\}$, so 
$$
\underline E(\mu,\alpha)\subset \bigcap_{M\in\N}\bigcup_{m\ge M}\bigcup_{I_m\in {\bf G}_m}\bigcup_{\substack{\alpha'\in A_{m+1}:\\
\alpha'\in [\alpha-\delta,\alpha+\delta]}}\bigcup_{I\in  {G}_{m+1}(\alpha')}I_mL_{m+1,\alpha'} I.
$$
For a fixed $m\ge 1$, all the cubes $I_mL_{m+1,\alpha'} I$ are of the same generation $g_{m+1}$. Moreover the set $\mathcal C_m$ of these cubes has a cardinality 
$$
\#\mathcal C_m=\sum_{I_m\in{\bf G}_m} \sum_{\substack{\alpha'\in A_{m+1}:\\
\alpha'\in [\alpha-\delta,\alpha+\delta]}} \#{G}_{m+1}(\alpha')\le (\#{\bf G}_m) (\# A_{m+1} )\max_{\substack{\alpha'\in A_{m+1}:\\
\alpha'\in [\alpha-\delta,\alpha+\delta]}} \#{G}_{m+1}(\alpha').
$$ 
We can deduce from \eqref{cardG} and the definition of $\gamma_{m+1}$ that
\begin{equation*}\label{card1 }
  \#{G}_{m+1}(\alpha')\le 2^{N_{m+1}(f(\alpha')+3\epsilon_{m+1})}.
\end{equation*}
Thus, by using the upper semi-continuity of $f$ we get 
$$
\#\mathcal C_m \le (\#{\bf G}_m )(\# A_{m+1} ) 2^{N_{m+1}(f(\alpha)+\eta+3\epsilon_{m+1})}.
$$
Moreover, 
\begin{equation*}\label{card2}
\#{\bf G}_m =\prod_{j=1}^m \Big (\sum_{\beta\in A_j} \#{G}_{j}(\beta)\Big ).
\end{equation*}
Now, we deduce from \eqref{cond} that $\log (\#{\bf G}_m)+\log (\# A_{m+1} ) =o((f(\alpha+\eta)N_{m+1})$, and finally obtain 
$$
\limsup_{m\to\infty}\frac{\log (\# \mathcal C_m)}{N_{m+1} \log(2)}\le f(\alpha)+\eta.
$$
Noting that $\lim_{m\to\infty} g_{m+1}/N_{m+1}=1$, this is enough to conclude that for all $\epsilon>0$ we have 
$$
\sum_{M\in\N}\sum_{m\ge M}\sum_{I_m\in {\bf G}_m}\sum_{\substack{\alpha'\in A_{m+1}:\\
\alpha'\in [\alpha-\delta,\alpha+\delta]}}\sum_{I\in  {G}_{m+1}(\alpha')}|I_mL_{m+1,\alpha'} I|^{f(\alpha)+\eta+\epsilon}<\infty,
$$
so $ \mathcal H^{f(\alpha)+\eta+\epsilon}(\underline E(\mu,\alpha))=\lim_{m\to\infty} \mathcal H^{f(\alpha)+\eta+\epsilon}_{2^{-g_{m+1}}} (\underline E(\mu,\alpha))=0$, and 
$\dim_H \underline E(\mu,\alpha)\le  f(\alpha)+\eta+\epsilon$. Since this holds for all $\eta>0$ and $\epsilon>0$, we get the conclusion.

\smallskip

Now suppose that $\alpha=\infty\in \mathcal{I}$. Let $\eta>0$ and $A>0$ such that $f(\alpha)\le f(\infty)+\eta$ for $\alpha\ge A$ ($f$ is upper semi-continuous over $I$). We have 
$$
E(\mu,\infty)\subset \bigcap_{M\in\N}\bigcup_{m\ge M}\bigcup_{I_m\in {\bf G}_m}\bigcup_{\alpha'\in A_{m+1}:\\
\alpha'\ge A}\bigcup_{I\in  {G}_{m+1}(\alpha')}I_mL_{m+1,\alpha'} I
$$
and following the same lines as for the case $\alpha<\infty$ yields $\dim_H E(\mu,\infty)\le f(\infty)$, since $f$ is upper-semi continuous at $\infty$.

\subsection{$\mu$ is exact dimensional, with dimension $D$}

By construction, for each $m\ge 2$, we have 
\begin{eqnarray*}
&&
\mu(\{x\in K: \alpha_{m}(x)\neq D_{m}\})\\
&=& \sum_{I\in \mathbf{G}_{m-1}}\sum_{\alpha\in A_{m}\setminus \{D_{m}\}}\sum_{J\in G_{m}(\alpha)}\mu(I L_{m,\alpha}J)\\
&=& \sum_{I\in \mathbf{G}_{m-1}}\mu(I )\frac{ \sum_{\alpha\in A_{m }\setminus \{D_{m }\}}\sum_{J\in {G}_{m }(\alpha)}\rho_m(\alpha) \mu_{\alpha}(J)}{ \sum_{\alpha\in A_{m}}\sum_{J\in {G}_{m}(\alpha)}\rho_m(\alpha) \mu_{\alpha}(J)}\\
&\le& \sum_{I\in \mathbf{G}_{m-1}}\mu(I_m) \frac{2^{-m}}{2^{-1}}=2^{-(m-1)},
\end{eqnarray*}
where we have used \eqref{controlmuq1} and  \eqref{concentration}. By the Borel-Cantelli lemma we get that $\mu$-almost everywhere, $\alpha_m(x)=D_{m}$ for $m$ large enough. Moreover, denoting by $I_m$ the same interval as in \eqref{Igm}, we have 
\begin{eqnarray*}
&&\mu(\{x\in K: \alpha_m(x)=D_m,\,  I_{m} \cap \widetilde F_{m}(D_{m},D_{m})=\emptyset\} )\\
&=& \sum_{I\in \mathbf{G}_{\widehat\alpha,m-1}} \mu_{m-1}(I) \frac{\displaystyle\sum_{J\in \mathcal{F}_{N_{m}},\, J \cap \widetilde F_{m}(D_{m},D_{m})=\emptyset}\nu_{D_{m}}(J)}{ \sum_{\alpha\in A_{m}}\sum_{J\in {G}_{m}(\alpha)}\rho_m(\alpha) \mu_{\alpha}(J)}\\
&\le & \sum_{I\in \mathbf{G}_{\widehat\alpha,m-1}} \mu_{m-1}(I) \frac{2^{-m}}{2^{-1}}= 2^{-(m-1)},
\end{eqnarray*}
where we have used \eqref{0000} and \eqref{controlmuq1}. Then, a new application of the Borel-Cantelli lemma implies that for $\mu$-almost every $x$, for $m$ large enough, we have both $\alpha_m(x)=D_m$ and $ I_{m} \cap \widetilde F_{m}(D_{m},D_{m})\neq\emptyset$. We can then conclude from \eqref{goodcontrols2} and Corollary~\ref{dimloc}(1)(2) that  $\mu$ is exact dimensional with dimension $D$, since $D_m$ converges to $D$ as $m\to\infty$. 

%

\section{Full illustration of the multifractal formalism: Proof of Theorem~\ref{thmHP}}\label{proof2}

\subsection{Construction of $\mu$}\label{constrmu2}
We will modify the scheme used in the proof of Theorem~\ref{thm1} by repeating recursively, for all $m\ge 1$, for some integers  $R^f_m$ anf $R^g_m$ to be specified, $R^f_m$ times the $m$th step with $\gamma_m(\alpha)$  approximating $f(\alpha)$, followed by $R^g_m$ times the $m$th step with $\gamma_m(\alpha)$  approximating $g(\alpha)$; this will make it possible to both guaranty the non emptyness of the sets $E(\mu,\alpha)$, $\alpha\in\mathcal{I}$, and the control of the difference between the associated Hausdorff and packing spectra. Additional conditions on $\gamma_m(\alpha)$ will be also needed to obtain an exact dimensional measure.

Let $D$ be a fixed point of $f$ (it is automatically a fixed point of $g$). Due to our assumption on the upper semi-continuity of $f$, there exists a countable subset $\Delta_f$ of $\mathcal{I}\setminus \{\infty\}$ such that for all $\alpha\in\mathcal{I}\setminus \{\infty\}$, there exists a sequence $(\alpha_n)_{n\ge 1}$ in $\Delta_f^{\N_+}$ such that $\lim_{n\to\infty}(\alpha_n,f(\alpha_n))=(\alpha,f(\alpha))$. Similarly,  there exists a countable subset $\Delta_g$ of $\mathcal{J}\setminus \{\infty\}$ such that for all $\alpha\in \mathcal{J}\setminus \{\infty\}$, there exists a sequence $(\alpha_n)_{n\ge 1}$ in $\Delta_g^{\N_+}$ such that $\lim_{n\to\infty}(\alpha_n,g(\alpha_n))=(\alpha,g(\alpha))$.

Fix once for all such $\Delta_f$ and $\Delta_g$ and we can assume that they contain $D$. Also take
$$
\epsilon_m=(m+1)^{-2}.
$$

Set  $\alpha^f_{\min}=\min (\mathcal{I})$, $\alpha^f_{\max}=\max (\mathcal{I})\in {\R}_+\cup\{\infty\}$, $\alpha^g_{\min}=\min (\mathcal{J})$, and $\alpha^g_{\max}=\max (\mathcal{J})\in {\R}_+\cup\{\infty\}$.
For $h\in\{f,g\}$, if $\Delta_h\setminus\{0,D\}\neq\emptyset$, enumerate the elements of  $\Delta_h\setminus\{0,D\}$ in a sequence $(\alpha^{\Delta_h}_j)_{j\ge 1}$, and  for each $m\in \N$ set 
$$
\widetilde A^h_m=\big \{\alpha^{\Delta_h}_j:1\le j\le m,\ \alpha^{\Delta_h}_j\ge 4\epsilon_m^{1/3}\big\}; 
$$
otherwise, set $\widetilde A^h_m=\emptyset$. Also set 
$$
\begin{cases}
D_m= 2\epsilon_m^{1/3}\text{ if } D=0, \text{ and } D_m=D \text{ otherwise},\\
\alpha^h_m(0)=D_m \text{ if } D=0 \text{ and }\alpha_m(0)=\min(2\epsilon_m^{1/3},D)/2 \text{ otherwise},\\
\alpha^h_m(\infty)=\big (\max (d,m, \max (\alpha^{\Delta_h}_j:1\le j\le m)\big )^2 \text{ if }\alpha^h_{\max}=\infty
\end{cases}.
$$
Then let 
$$
A^h_m=\begin{cases}\widetilde A^h_m\cup\{D_m\}&\mbox{if } 0<\alpha^h_{\min}\le \alpha^h_{\max}<\infty\\
\widetilde A^h_m\cup\{D_m\}\cup \{\alpha^h_m(\infty)\}&\mbox{if }0<%
\alpha^h_{\min}\text{ and } \alpha^h_{\max}=\infty,\\
\widetilde A^h_m\cup\{D_m\}\cup \{\alpha^h_m(0)\}&\mbox{if }
\alpha^h_{\min}=0\text{ and } \alpha^h_{\max}<\infty,\\
\widetilde A^h_m\cup\{D_m\}\cup \{\alpha^h_m(0)\}\cup\{\alpha^h_m(\infty)\}&\mbox{if }
\alpha^h_{\min}=0\text{ and } \alpha^h_{\max}=\infty,
\end{cases}.
$$ 

For $\alpha\in A^h_m$ let 
$$
\widetilde {\gamma}^h_m(\alpha)=\begin{cases}
h(\alpha)&\text{if }\alpha\in \widetilde A^h_m\text{ and }h(\alpha)>0\\
h(\infty)&\text{if }\alpha= \alpha^h_m(\infty)\text{ and }h(\infty)>0,\\
\alpha &\text{if }\alpha=D_m,\\
 \alpha_m(0)&\text{if }\alpha=  \alpha^h_m(0),\\
 \epsilon_m^{1/3}&\text{if }\alpha\in\widetilde A^h_m\text{ and }h(\alpha)=0,\\
\epsilon_m^{1/3}&\text{if }\alpha=\alpha^h_m(\infty)\text{ and }h(\infty)=0
\end{cases}
$$
(notice that $f(\infty)=0$ (resp. $g(\infty)=0$) only if $f=0$ (resp. $g=0$) over $\mathcal{I}$ (resp. $\mathcal{J}$) due to our assumptions on $f$ (resp. $g$)). Then, set (this will be used to entail that $\mu$ is exact dimensional)  
$$
\gamma^h_m(\alpha)=\begin{cases}
\alpha=\widetilde \gamma^h_m(\alpha)&\text{if } \alpha=D_m\\
 (1-\theta^h_m)\widetilde \gamma^h_m(\alpha)&\text{otherwise}
 \end{cases},
$$
where $\theta^h_m\in (0,1)$ tends to $0$ slowly enough so that 
\begin{equation}\label{thetam}
\forall\,\alpha\in A_m^h\setminus \{D_m\},\ 
\begin{cases}
\alpha-\gamma^h_m(\alpha)=\alpha-(1-\theta^h_m)\widetilde\gamma^h_m(\alpha)\ge D_m\sqrt{\epsilon_m}\\
(1-\sqrt{\epsilon_m})\alpha-\gamma^h_m(\alpha)\ge 0
\end{cases}
.
\end{equation}
 We then have in particular  $\gamma_m^h(\alpha)<\alpha$ for all $\alpha\in A_m^h\setminus \{D_m\}$.
 
\medskip

Using the definitions of Section~\ref{sec2}, for $\alpha\in A^h_m$ set 
\begin{equation*}\label{defmunu}
\mu_{\alpha}=\mu_{{\alpha,\gamma^h_m(\alpha)}} \text{ and }\nu_{\alpha}=\nu_{{\alpha,\gamma^h_m(\alpha)}} \text{ for  }\alpha\in A_m^h\setminus \{D_m\}\text{ and } \mu_{D_m}=\nu_{D_m}=\nu_{{D_m,D_m}} \text{ for  }\alpha=D_m.
\end{equation*}
Strickly speaking, $\mu_\alpha$ and $\nu_\alpha$ should be written $\mu^h_{m,\alpha}$ and $\nu^h_{m,\alpha}$, but for the sake of readability we will omit the indices $m$ and $h$.

Also, set 
$$
n^h_m=\max\{n_{m}(\alpha,\gamma^h_m(\alpha)): \alpha\in A^h_m\}.
$$

Now let $(R^f_m)_{m\ge 1}$,  $(R^g_m)_{m\ge 1}$ and $(N_m)_{m\ge 1}$ be three increasing sequences of positive integers defined recursively  and satisfying the following properties:
\begin{equation}\label{condHP}
\begin{cases}(1) \ \forall\, m\ge 1, e^{m+1}\le R^f_{m}\le  R^g_m\le R_{m+1}^f;\\

(2) \ \forall\, m\ge 1,\ N_m\ge \max (n^h_m, \exp (\#A_m^h),\exp(m)) \text{ for }h\in\{f,g\},\text{ and, as }m\to\infty:\\
(3)\  (\max (\{1\}\cup A_m^f\cup A_m^g))N_m=o\Big ((\min (\{1\}\cup  A_{m-1}^f\cup A_{m-1}^g))\displaystyle\sum_{i=1}^{m-1}(R^f_i+R^g_i)N_i\Big );\\
(4)\  (\max (\{1\}\cup A_{m-1}^f\cup A_{m-1}^g))\displaystyle\sum_{i=1}^{m-1}(R^f_i+R^g_i)N_i =o\big (\min (\{1\}\cup\gamma_m(A^{f}_m))R^f_{m}N_m\big );\\
(5)\  (\max (\{1\}\cup A_m^f\cup A_m^g))\displaystyle(R^f_{m}N_m+\sum_{i=1}^{m-1}(R^f_i+R^g_i)N_i) =o\big (\min (\{1\}\cup\gamma_m(A^{g}_m))R^g_{m}N_m\big ).
\end{cases}
\end{equation}

Then, for $\alpha\in A^h_m$ set 
$$
G^h_m(\alpha)=\{I\in \mathcal{F}_{N_m}: I\cap F_{m}(\alpha,\gamma^h_m(\alpha))\neq\emptyset\}
$$ 
and 
\begin{equation}\label{rhof}
\rho_m^h(\alpha)=\begin{cases}
1&\text{ if }\alpha=D_m\\
(2^{-m}/\#A^h_m)^2&\text{ otherwise}
\end{cases}.
\end{equation}
 
We enumerate the elements of $A^h_m$ as $\alpha^h_{m,i}$, $1\le i\le \#A^h_m$ and denote by $L^h_{m,\alpha_i}$, $1\le i\le\#A^h_m$,  the disjoint closed dyadic cubes of generation $\ell^h_m=\ell(\#A_m^h)$ of the set $\mathcal{L}_m=\mathcal{L}(\#A_m^h)$ defined in Section~\ref{nota}.  We also denote $L^h_{m,\alpha_i}$ by~$L^h_{m,i}$.

Let us start the construction of $\mu$. We consider the same measure $\mu_1$ on $\mathbf{G}_1$ as in Section~\ref{constmu1}, except that take the sets $G_1^f(\alpha)$,  $\alpha\in A_1^f$, instead of  the sets $G_1(\alpha)$,  $\alpha\in A_m$,  and the collection $\mathcal{L}^f_1$ instead of $\mathcal{L}_1$.  Then, for $1\le s\le R^f_1-1$ we define recursively

$$
{\bf G}_{s+1}=\bigcup_{I_s\in {\bf G}_s}\bigcup_{\alpha_{1,i}\in A^f_{1}}  {\bf G}^f_{1}(I_s,\alpha_{1,i}), 
$$
where
$$
{\bf G}_{1}(I_s,\alpha_{1,i})=\big\{I_sL^f_{1,i} I_{1,i}:\, I_{1,i}\in G^f_{1}(\alpha_{1,i})\big \}
$$
and a measure $\mu_{s+1}$ on $ {\bf G}_{s+1}$ as
\begin{equation*}
\mu_{s+1}(I_s L^f_{1,i} I_{1,i})=\mu_s(I_s) \frac{\rho_1^f(\alpha_{1,i})\mu_{\alpha_{1,i}}(I_{1,i})}{\sum_{\alpha\in A_{1}}\sum_{I\in G^f_{1}(\alpha)}\rho_1^f(\alpha)\mu_{{\alpha}}(I)}.
\end{equation*}
Then, define recursively a sequence $ ({\bf G}_s)_{s\ge 1}$of sets of intervals of the same generation  and a sequence of measures $(\mu_s)_{s\ge 1}$ as follows: 

For all $m\ge 1$, and $R^f_{m}+\sum_{j=1}^{m-1} R^f_{j}+R^g_{j}<s+1\le  \sum_{j=1}^{m} R^f_{j}+R^g_{j}$, 
$$
{\bf G}_{s+1}=\bigcup_{I_s\in {\bf G}_s}\bigcup_{\alpha_{m,i}\in A^g_{m}}  {\bf G}_{m}(I_s,\alpha_{m,i}), 
$$
where
 $$
{\bf G}_{m}(I_s,\alpha_{m,i})=\big\{I_sL^g_{m,i} I_{m,i}:\, I_{m,i}\in G^g_{m}(\alpha_{m,i})\big \}
$$
and the measure $\mu_{s+1}$ on $ {\bf G}_{s+1}$ is defined as 
\begin{equation*}
\mu_{s+1}(I_s L^g_{m,i} I_{m,i})=\mu_s(I_s) \frac{\rho_m^g(\alpha_{m,i})\mu_{{\alpha_{m,i}}}(I_{m,i})}{\sum_{\alpha\in A^g_{m}}\sum_{I\in G^g_{m}(\alpha)}\rho_m^f(\alpha)\mu_{{\alpha}}(I)}.
\end{equation*}

For all $m\ge 2$, and $\sum_{j=1}^{m-1} R^f_{j}+R^g_{j}<s+1\le  R^f_{m}+\sum_{j=1}^{m} R^f_{j}+R^g_{j}$, 
$$
{\bf G}_{s+1}=\bigcup_{I_s\in {\bf G}_s}\bigcup_{\alpha_{m,i}\in A^f_{m}}  {\bf G}_{m}(I_s,\alpha_{1,i}), 
$$
where
 $$
{\bf G}_{m}(I_s,\alpha_{m,i})=\big\{I_sL^f_{m,i} I_{m,i}:\, I_{m,i}\in G^f_{m}(\alpha_{m,i})\big \}
$$
and the measure $\mu_{s+1}$ on $ {\bf G}_{s+1}$ is defined as 
\begin{equation*}
\mu_{s+1}(I_s L^f_{m,i} I_{m,i})=\mu_s(I_s) \frac{\rho_m^f(\alpha_{m,i})\mu_{{\alpha_{m,i}}}(I_{m,i})}{\sum_{\alpha\in A^f_{m}}\sum_{I\in G^f_{m}(\alpha)}\rho_m^f(\alpha)\mu_{{\alpha}}(I)}.
\end{equation*}

This yields (in the same way as in Section~\ref{constmu1}) a Borel probability measure $\mu$ supported on 
$$
K=\bigcap_{s\ge 1} \bigcup_{I\in\mathbf{G}_s} I
$$
such that $\mu(I)=\mu_s(I)$ for all $s\ge 1$ and $I\in\mathbf{G}_s$. 
\medskip

For each $m\ge 1$, we define 
$$
s_m= R^f_{m}+\sum_{i=1}^{m-1} R^f_{i}+R^g_{i} \quad\text{and}\quad s'_m=\sum_{i=1}^{m} R^f_{i}+R^g_{i}.
$$
Then, for $s\ge 1$, we denote by $n(s)$ the generation of the cubes belonging to $\mathbf{G}_s$, i.e.
$$
n(s)=(s-s'_{m-1})(N_m+\ell^f_m)+\sum_{i=1}^{m-1} R^f_{i}(N_i+\ell^f_i)+R^g_{i}(N_i+\ell^g_i)
$$
if  $s'_{m-1}<s\le s_m$, and 
 $$
 n(s)= (s-s_{m})(N_m+\ell^g_m)+R^f_m(N_m+\ell^f_m)+\sum_{i=1}^{m-1} R^f_{i}(N_i+\ell^f_i)+R^g_{i}(N_i+\ell^g_i)
$$ 
if $s_m<s\le s'_m$.
 
 The following property, which follows immediately from \eqref{condHP}  will be useful. If $s'_{m-1}<s\le s_m$, set 
 \begin{equation}\label{n'1}
 n'(s)=(s-s'_{m-1})N_m+\sum_{i=1}^{m-1} R^f_iN_i+R^g_{i}N_i,
\end{equation}
 and if $s_m<s\le s'_m$, set 
 \begin{equation*}\label{n'2}
 n'(s)=(s-s_m)N_m+R^f_mN_m+\sum_{i=1}^{m-1} R^f_iN_i+R^g_{i}N_i.
\end{equation*}
 \begin{pro}\label{n'} If $s_{m-1}<s\le s_m$, then
$ \max (1, \max(A_m^f\cup A^g_m)\sum_{i=1}^m R^f_m\ell^f_m+R^g_m\ell^g_m=o(n'(s))$. In particular, $n'(s)\sim n(s)$ as $s\to\infty$.   
 \end{pro}
 \begin{rem}\label{neigh}
By construction, due to our choice for the cubes $L^h_{m,i}$, if $J=I_{s-1} L^h_{m,i} I_{m,i}\in \mathbf{G}_{s}$ and $I$ is a dyadic cube in   $\mathcal{N}_2(n(s),J)$ (see Section~\ref{nota} for the definition),  either $\mu(I)=0$ or  $I$ takes the form $I_{s-1} L^h_{m,i} I'_{m,i}$ and $\mu(J)2^{-\epsilon(s) n(s)}\le \mu(I)\le \mu(J)2^{\epsilon(s) n(s)}$, where $\epsilon(s)$ is independent of $J$ and the sequence $(\epsilon(s))_{s\ge 1}$  tends to  0 as $s\to\infty$. 
\end{rem}
\subsection{Reduction of the problem} In this section we explain why the measure $\mu$ constructed in the previous section has the nice property that it is possible to replace centered balls by dyadic cubes in all the sets and quantities involved in the multifractal formalism without modifying them (it is of course impossible to do so for any measure in $\mathcal M^+_c(\R^d))$. 

\medskip

Let us start with two properties which easily follow from \eqref{condHP}(1--3) by construction:
\begin{eqnarray*}
\lim_{s\to\infty}\frac{n(s+1)}{n(s)}&=&1\\
\text{and}\quad \lim_{s\to\infty} \sup_{x\in K} \Big|\frac{\log_2(\mu(I_{n(s)}(x))}{n(s)}-\frac{\log_2(\mu(I_{n(s+1)}(x))}{n(s+1)}\Big |&=&0.
\end{eqnarray*} 
Moreover, if $2^{-n(s+1)}< r\le 2^{-n(s)}$ we have 
$$
I_{n(s+1)}(x)\subset B(x,r)\subset B(x,2r)\subset   \bigcup_{I\in \mathcal{N}_2(n(s),I_{n(s)}(x))} I,
$$
and  if $I\in \mathcal{N}_2(n(s),I_{n(s)}(x))$, either  $ \mu(I_{n(s)}) e^{-\epsilon(s) n(s)}\le \mu(I)\le \mu(I_{n(s)}) e^{\epsilon(s) n(s)}$ or  $\mu(I)=0$,  where $\epsilon(s)$ is independent of $x$ and $I$, by  Remark~\ref{neigh}.

It follows that for all $x\in K$ we have 
\begin{equation}\label{dimlocid}
\underline d(\mu,x)=\liminf_{s\to\infty} \frac{\log_2(\mu(I_{n(s)}(x))}{-n(s)}\text{ and } \overline d(\mu,x)=\limsup_{s\to\infty} \frac{\log_2(\mu(I_{n(s)}(x))}{-n(s)},
\end{equation}
and 
$\mu$ is weakly doubling in the sense that there exists a function $\widetilde\epsilon(r)$ tending to $0^+$ as $r\to 0^+$ such that 
$$
\forall\, x\in K,\ \mu(B(x,2r))\le r^{-\widetilde \epsilon(r)}\mu(B(x,r)).
$$
Also, the above properties and standard covering arguments (see in particular \cite[Section 4.6]{Olsen} where doubling measures are used) yield, for all $q\in\R$,
$$
\tau_\mu(q)= \liminf_{s\to \infty}   \frac{\log_2 \sum_{I_s\in \mathbf{G}_s}\mu(I_s)^q}{-n(s)}
\text{ and }
\overline \tau_\mu(q)=\limsup_{s\to \infty}   \frac{\log_2 \sum_{I_s\in \mathbf{G}_s}\mu(I_s)^q}{-n(s)}.
$$

Similarly,  for $\alpha\in \R_+$ we have
\begin{eqnarray*}
\underline{f}^{\rm LD}_\mu(\alpha)&=&\lim_{\epsilon\to 0^+} \liminf_{s\to \infty}\frac{\log_2 \#\Big \{I\in \mathbf{G}_s:2^{-n(s)(\alpha+\epsilon)}\le  \mu(I)\le2^{-n(s)(\alpha-\epsilon)}\Big\}}{n(s)},
\\
\overline{ f}^{\rm LD}_\mu(\alpha)&=&\lim_{\epsilon\to 0^+}\limsup_{s\to \infty}\frac{\log_2 \#\Big \{I\in \mathbf{G}_s:2^{-n(s)(\alpha+\epsilon)}\le  \mu(I)\le2^{-n(s)(\alpha-\epsilon)}\Big\}}{n(s)},
\end{eqnarray*}
and \begin{eqnarray*}
\underline{f}^{\rm LD}_\mu(\infty)&=&\lim_{A\to\infty} \liminf_{s\to \infty}\frac{\log_2 \#\Big \{I\in \mathbf{G}_s: \mu(I)\le 2^{-n(s)A}\Big\}}{n(s)},\\
\overline{ f}^{\rm LD}_\mu(\infty)&=&\lim_{A\to\infty} \limsup_{s\to \infty}\frac{\log_2 \#\Big \{I\in \mathbf{G}_s: \mu(I)\le 2^{-n(s)A}\Big\}}{n(s)},
\end{eqnarray*}
and for $0\le \alpha\le \beta\le \infty$ such that $(\alpha,\beta)\neq (\infty,\infty)$, 
$$
\underline{f}^{\rm LD}_\mu(\alpha,\beta)=\lim_{\epsilon\to 0^+} \liminf_{s\to \infty}\frac{\log_2 \#\Big \{I\in \mathbf{G}_s:2^{-n(s)(\beta+\epsilon)}\le  \mu(I)\le2^{-n(s)(\alpha-\epsilon)}\Big\}}{n(s)}.
$$

Finally, due to the multiplicative nature of the construction of $\mu$, defining for each $m\ge 1$ and $q\in \R$
$$
T^f_{m}(q)= \log_2\frac{\sum_{\alpha\in A_m^f} \sum_{I\in G^f_m(\alpha)} \rho^f_m(\alpha)^q\mu_\alpha(I)^q}{\Big (\sum_{\alpha\in A_m^f} \sum_{I\in G^f_m(\alpha)} \rho^f_m(\alpha)\mu_\alpha(I)\Big)^q}
$$
and 
$$
T^g_{m}(q)= \log_2\frac{\sum_{\alpha\in A_m^g} \sum_{I\in G^g_m(\alpha)} \rho^g_m(\alpha)^q\mu_\alpha(I)^q}{\Big (\sum_{\alpha\in A_m^g} \sum_{I\in G^g_m(\alpha)} \rho^g_m(\alpha)\mu_\alpha(I)\Big)^q},
$$
we have for $s\ge 1$ and $q\in\R$
$$
\log_2 \sum_{I_s\in \mathbf{G}_s}\mu(I_s)^q=\sum_{i=1}^{m-1} R^f_{i} T^f_{i}(q)+R^g_{i} T^g_{i}(q)+(s-s'_{m-1})\, T^f_{m}(q)
$$ 
if $s'_{m-1}<s\le s_m$
and 
$$
\log_2 \sum_{I_s\in \mathbf{G}_s}\mu(I_s)^q=\Big (\sum_{i=1}^{m-1} R^f_{i} T^f_{i}(q)+R^g_{i} T^g_{i}(q)\Big )+ R^f_{m}T^f_{m}(q)+ (s-s_m)\,T^g_{m}(q)
$$ 
if 
$
s_m<s\le s'_m$. It follows that
\begin{equation}\label{toto}
\tau_\mu(q)= \liminf_{S\ni s\to \infty}   \frac{\log_2 \sum_{I_s\in \mathbf{G}_s}\mu(I_s)^q}{-n(s)}\text{ and }\overline \tau_\mu(q)=\limsup_{S\ni s\to \infty}   \frac{\log_2 \sum_{I_s\in \mathbf{G}_s}\mu(I_s)^q}{-n(s)},
\end{equation}
where 
$S=\Big \{s_m,s'_m: m\ge 1\Big \}$.

\subsection{Local dimension estimates for $\mu$, auxiliary measures, and lower bounds for the different  spectra}
\subsubsection{Local dimension estimates for the measure $\mu$}

Let $x\in K$ and $s\ge 1$. 

If $s'_{m-1}<s\le s_m$, for all $1\le i\le m-1$, there are, uniquely determined, $R^f_i$ elements $(\alpha^f_{i,j}(x))_{1\le j\le R^f_i}\in (A^f_i)^{R^f_i}$, $R_i^g$ elements $(\alpha^g_{i,j}(x))_{1\le j\le R_i^g}\in (A^g_i)^{R^g_i}$, and $s'=s-s'_{m-1}$ elements $(\alpha^f_{m,j}(x))_{1\le j\le s'}\in (A^f_{m})^{s'}$, and for each exponent $\alpha^h_{i,j}(x)$ of this collection a unique element $I_{\alpha^h_{i,j}(x)}$ of $G^h_i(\alpha^h_{i,j}(x))$ and a unique element $L^h_{i,\alpha^h_{i,j}(x)}$ of $\mathcal{L}_i^h$ such that 
\begin{equation*}
\label{Ig1}I_{n(s)}(x)=\Big (\odot_{i=1}^{m-1} (\odot_{j=1}^{R^f_i}L^f_{j,\alpha^f_{1,j}(x)}I_{\alpha^f_{1,j}(x)})\cdot (\odot_{j=1}^{R^g_i}L^g_{j,\alpha^g_{1,j}(x)}I_{\alpha^g_{1,j}(x)})\Big )\cdot  (\odot_{j=1}^{s'}L^f_{m,\alpha^f_{m,j}(x)}I_{\alpha^f_{m,j}(x)}),
\end{equation*}
(where the notation $\odot_{p=1}^q I_p$ stands for the cube obtained by concatenation $I_1\cdot I_2\cdots I_q$ of the cubes $I_1,\ldots,I_q$). 

By construction, using the analogue of  \eqref{mum+1'} we get, writing $\alpha^h_{i,j}$ for  $\alpha^h_{i,j}(x)$, 
\begin{multline*}
\Big (\frac{1}{1+2^{-m}}\Big )^{s'} \Big (\prod_{i=1}^{m-1}\Big (\frac{1}{1+2^{-i}}\Big )^{R^f_i+R^g_i}\Big)\prod_{j=1}^{s'} \rho_m^f(\alpha^f_{m,j})\Big (\prod_{i=1}^{m-1} \Big (\prod_{j=1}^{R^f_i}\rho_i^f({\alpha^f_{i,j}})\Big ) \Big (\prod_{j=1}^{R^g_i}\rho_i^g({\alpha^g_{i,j}})\Big )\\ \le \frac{\mu(I_{n(s)})}{\prod_{j=1}^{s'}\mu_{{\alpha^f_{m,j}}}(I_{\alpha^f_{m,j})}) \Big (\prod_{i=1}^{m-1} \Big (\prod_{j=1}^{R^f_i}\mu_{{\alpha^f_{i,j}}}(I_{\alpha^f_{i,j}}) \Big ) \Big (\prod_{j=1}^{R^g_i}\mu_{{\alpha^g_{i,j}}}(I_{\alpha^g_{i,j}})}\Big )\\
\le 2^{m s'+\sum_{i=1}^{m-1} i (R^f_i+R^g_i)}.
\end{multline*}
Consequently, due to the fact that $I_{\alpha^h_{i,j}(x)}\in G^h_i(\alpha^h_{i,j}(x))$ for $h\in\{f,g\}$, using \eqref{condHP}(1)(2) and Proposition~\ref{n'} we get 
\begin{equation}\label{dimlocmu0}
2^{-\epsilon(s)n(s)} \le \frac{\mu(I_{n(s)}(x))}{2^{-n(s)\alpha_s(x)}}\le 2^{\epsilon(s)n(s)}
\end{equation}
with 
$$
\alpha_{s}(x)=\frac{\sum_{j=1}^{s'}N_{m}\alpha^f_{m,j}(x)+\Big (\sum_{i=1}^{m-1}  \sum_{j=1}^{R^f_i}N_i \alpha^f_{i,j}(x)+ \sum_{j=1}^{R^g_i}N_i\alpha^g_{i,j}(x)\Big )}{s'N_m+\sum_{i=1}^{m-1} R^f_{i}N_i+R^g_{i}N_i}$$
and $\lim_{s\to\infty}\epsilon(s)=0$. Due to \eqref{condHP}(3-5), 
\begin{equation}\label{dls1}
\text{if $s=s_m$, we have}\quad \alpha_{s}(x)=\frac{\sum_{j=1}^{R^f_m}N_{m}\alpha^f_{m,j}(x)}{R^f_mN_m}+\epsilon'(s),
\end{equation}
with $\lim_{s\to\infty}\epsilon'(s)=0$ uniformly in $x$. 

If now $s_m<s\le s'_m$ and $s'=s-s'_m$, with similar notations we get 
\begin{eqnarray}
\label{Ig2}I_{n(s)}(x)=\Big (\odot_{i=1}^{m-1} (\odot_{j=1}^{R^f_i}L^f_{j,\alpha^f_{1,j}(x)}I_{\alpha^f_{1,j}(x)})\cdot (\odot_{j=1}^{R^g_i}L^g_{j,\alpha^g_{1,j}(x)}I_{\alpha^g_{1,j}(x)})\Big )\\
\nonumber \cdot  (\odot_{j=1}^{R_m^f}L^f_{m,\alpha^f_{m,j}(x)}I_{\alpha^f_{m,j}(x)})\cdot  (\odot_{j=1}^{s'}L^g_{m,\alpha^g_{m,j}(x)}I_{\alpha^g_{m,j}(x)}) 
\end{eqnarray}
and 
\begin{equation}\label{dimlocmu0'}
2^{-\epsilon(s)n(s)} \le \frac{\mu(I_{n(s)}(x))}{2^{-n(s)\alpha_s(x)}}\le 2^{\epsilon(s)n(s)}
\end{equation}
with 
$$
\alpha_{s}(x)=\frac{\sum_{j=1}^{s'}N_{m}\alpha^g_{m,j}(x)+\sum_{j=1}^{R^f_m} N_{m}\alpha^f_{m,j}(x)+\sum_{i=1}^{m-1}  \sum_{j=1}^{R^f_i}N_i \alpha^f_{i,j}(x)+ \sum_{j=1}^{R^g_i}N_i \alpha^g_{i,j}(x)}{s'N_m+R^f_mN_m+\sum_{i=1}^{m-1} R^f_iN_i+R^g_{i}N_i}$$
and $\lim_{s\to\infty}\epsilon(s)=0$. In addition, 
\begin{equation}\label{dls2}
\text{if $s=s'_m$, we have}\quad \alpha_{s}(x)=\frac{\sum_{j=1}^{R^g_m}N_{m}\alpha^g_{m,j}(x)}{R^g_mN_m}+\epsilon'(s),
\end{equation}
with $\lim_{s\to\infty}\epsilon'(s)=0$ uniformly in $x$.

\subsubsection{Auxiliary measures}\label{secaux}
Let $\widehat \alpha=(\alpha^f_1,\alpha^g_1,\ldots,\alpha^f_m,\alpha^g_m,\ldots)\in \prod_{m=1}^\infty A^f_m\times A^g_m$. 

We construct a measure $ \nu_{\widehat\alpha}$ as follows: Let 
$$
{\bf G}_{\widehat\alpha,1}=\{L^f_{1,\alpha^f_1}I_1: I_1\in G^f_1(\alpha^f_1)\}, 
$$
and define 
\begin{equation*}
 \nu_{\widehat\alpha,1}(L^f_{1,\alpha^f_1} I_1)= \frac{\nu_{{\alpha^f_1}}(I_1)}{\displaystyle\sum_{I\in G^f_1(\alpha^f_1)}\nu_{{\alpha^f_1}}(I)}.
\end{equation*}

Then, for $1\le s\le R^f_1-1$ we define recursively

$$
{\bf G}_{\widehat\alpha,s+1}=\bigcup_{I_s\in {\bf G}_{\widehat\alpha,s} }{\bf G}_{\widehat\alpha,s+1}(I_s,\alpha^f_1), 
$$
where
$$
{\bf G}_{\widehat\alpha,s+1}(I_s,\alpha^f_1)=\big\{I_sL^f_{1,\alpha^f_1} I_1:\, I_{1}\in G^f_{1}(\alpha^f_1)\big \}
$$
and a measure $\nu_{\widehat \alpha,s+1}$ on ${\bf G}_{\widehat\alpha,s+1}$ as
\begin{equation*}
\nu_{\widehat \alpha,s+1}(I_s L^f_{1,\alpha^f_1} I_{1})=\nu_{\widehat \alpha,s}(I_s) \frac{\nu_{\alpha^f_{1}}(I_1)}{\sum_{I\in G^f_{1}(\alpha_1^f)}\nu_{\alpha^f_1}(I)}.
\end{equation*}
Then, define recursively a sequence $ ({\bf G}_{\widehat \alpha,s})_{s\ge 1}$ of sets of intervals of the same generation  and a sequence of measures $(\nu_{\widehat \alpha,s})_{s\ge 1}$ as follows: 

For all $m\ge 1$, and $s_m<s+1\le s'_m$, 
$$
{\bf G}_{\widehat \alpha,s+1}=\bigcup_{I_s\in {\bf G}_{\widehat \alpha,s}} {\bf G}_{\widehat \alpha,s+1}(I_s,\alpha^g_m), 
$$
where
 $$
{\bf G}_{\widehat \alpha,s+1}(I_s,\alpha^g_m)=\big\{I_sL^g_{m,\alpha_m^g} I_{m}:\, I_m\in G^g_{m}(\alpha_m^g)\big \}
$$
and the measure $\nu_{\widehat \alpha,s+1}$ on $ {\bf G}_{\widehat \alpha,s+1}$ is defined as 
\begin{equation*}\label{nus+1}
\nu_{\widehat \alpha,s+1}(I_s L^g_{m,\alpha^g_m} I_{m})=\nu_{\widehat \alpha,s}(I_s) \frac{\nu_{{\alpha^g_{m}}}(I_{m})}{\sum_{I\in G^g_{m}(\alpha_m^g)}\nu_{{\alpha^g_m}}(I)}.
\end{equation*}

For all $m\ge 2$, and $s'_{m-1}<s+1\le  s_m$, 
$$
{\bf G}_{\widehat \alpha,s+1}=\bigcup_{I_s\in {\bf G}_{\widehat \alpha,s}} {\bf G}_{\widehat \alpha,s+1}(I_s,\alpha^f_m), 
$$
where
 $$
{\bf G}_{\widehat \alpha,s+1}(I_s,\alpha^f_m)=\big\{I_sL^f_{m,\alpha_m^f} I_{m}:\, I_m\in G^f_{m}(\alpha_m^f)\big \}
$$
and the measure $\nu_{\widehat \alpha,s+1}$ on $ {\bf G}_{\widehat \alpha,s+1}$ is defined as 
\begin{equation*}\label{nus+1'}
\nu_{\widehat \alpha,s+1}(I_s L^f_{m,\alpha^f_m} I_{m})=\nu_{\widehat \alpha,s}(I_s) \frac{\nu_{{\alpha^f_{m}}}(I_{m})}{\sum_{I\in G^f_{m}(\alpha_m^f)}\nu_{{\alpha^f _m}}(I)}.
\end{equation*}
This yields a Borel probability measure $\nu_{\widehat \alpha}$ supported on 
$$
K\supset K_{\widehat\alpha}=\bigcap_{s\ge 1} \bigcup_{I\in{\bf G}_{\widehat \alpha,s}} I
$$
and such that $\nu_{\widehat \alpha}(I)=\nu_{\widehat \alpha,s}(I)$ for all $s\ge 1$ and $I\in\mathbf{G}_{\widehat \alpha,s}$. 
Moreover, estimates similar to those used to control the local dimension of $\mu$ show that there exists a positive  sequence $(\epsilon (s))_{s\ge 1}$ such that $\lim_{s\to\infty} \epsilon(s)=0$ and for all $x\in  K_{\widehat\alpha}$ and $s\ge 1$, if  $s'_{m-1}<s\le s_m$ and $s'=s-s'_{m-1}$, we have 
\begin{equation}\label{dimlocnu0}
2^{-\epsilon(s)n(s)} \le \frac{\nu_{\widehat\alpha}(I_{n(s)}(x))}{2^{-n(s)\gamma_s(x)}}\le 2^{\epsilon(s)n(s)}
\end{equation}
with 
$$
\gamma_{s}(x)=\frac{\sum_{j=1}^{s'}N_{m}\gamma^f_m(\alpha^f_{m,j}(x))+\Big (\sum_{i=1}^{m-1}  \sum_{j=1}^{R^f_i}N_i\gamma^f_i( \alpha^f_{i,j}(x))+ \sum_{j=1}^{R^g_i}N_i\gamma^g_i(\alpha^g_{i,j}(x))\Big )}{s'N_m+\sum_{i=1}^{m-1} R^f_iN_i+R^g_{i}N_i},
$$
and if $s_m<s\le s'_m$ and  $s'=s-s_m$, we have 
\begin{equation}\label{dimlocnu0'}
2^{-\epsilon(s)n(s)} \le \frac{\nu_{\widehat\alpha}(I_{n(s)}(x))}{2^{-n(s)\gamma_s(x)}}\le 2^{\epsilon(s)n(s)}
\end{equation}
with 
$$
\gamma_{s}(x)=\frac{\displaystyle \sum_{j=1}^{s'}N_{m}\gamma^g_m(\alpha^g_{m,j}(x))+\sum_{j=1}^{R^f_m} N_{m}\gamma^f_m(\alpha^f_{m,j}(x))+\sum_{i=1}^{m-1}  \sum_{j=1}^{R^f_i}N_i \gamma^f_i(\alpha^f_{i,j}(x))+ \sum_{j=1}^{R^g_i}N_i \gamma^g_i(\alpha^g_{i,j}(x))}{s'N_m+R^f_mN_m+\sum_{i=1}^{m-1} R^f_iN_i+R^g_{i}N_i}.
$$
Moreover, since by construction for $x\in K_{\widehat\alpha}$, we have $\alpha^f_{i,j}(x)=\alpha^f_i$ and $\alpha^g_{i,j}(x)=\alpha^g_i$ for all $1\le j\le R^f_i$ and  $1\le j\le R^g_i$ respectively, from \eqref{dimlocmu0} and \eqref{dimlocnu0} we get 
\begin{equation}\label{expo}
\begin{split}
\alpha_s(x)&= \frac{{s'}N_{m}\alpha^f_{m}+\sum_{i=1}^{m-1}R^f_iN_i \alpha^f_{i}+ {R^g_i}N_i\alpha^g_{i}}{s'N_m+\sum_{i=1}^{m-1} R^f_iN_i+R^g_{i}N_i}\\
\gamma_s(x)&= \frac{{s'}N_{m}\gamma^f_m(\alpha^f_{m})+\sum_{i=1}^{m-1} {R^f_i}N_i\gamma^f_i( \alpha^f_i)+ {R^g_i}N_i\gamma^g_i(\alpha^g_{i})}{s'N_m+\sum_{i=1}^{m-1} R^f_iN_i+R^g_{i}N_i}  
\end{split}
\end{equation}
if  $s'_{m-1}<s\le s_m$, and from \eqref{dimlocmu0'} and \eqref{dimlocnu0'} we get 
\begin{equation}\label{expo'}
\begin{split}
\alpha_s(x)&= \frac{{s'}N_{m}\alpha^g_{m}+R_m^f N_m \alpha^f_m+\sum_{i=1}^{m-1}R^f_iN_i \alpha^f_{i}+ {R^g_i}N_i\alpha^g_{i}}{s'N_m+R^f_mN_m+\sum_{i=1}^{m-1} R^f_iN_i+R^g_{i}N_i}\\
\gamma_s(x)&= \frac{{s'}N_{m}\gamma^g_m(\alpha^g_{m})+R_m^f N_{m}\gamma^f_m(\alpha^g_{m})+\sum_{i=1}^{m-1} {R^f_i}N_i\gamma^f_i( \alpha^f_i)+ {R^g_i}N_i\gamma^g_i(\alpha^g_{i})}{s'N_m+R^f_mN_m+\sum_{i=1}^{m-1} R^f_iN_i+R^g_{i}N_i}  
\end{split}
\end{equation}
if $s_m<s\le s'_m$.

We also have 
\begin{equation}\label{dimlocid2}
\displaystyle \underline d(\nu_{\widehat\alpha},x)=\liminf_{s\to\infty} \frac{\log_2(\nu_{\widehat\alpha}(I_{n(s)}(x))}{-n(s)}\quad\text{and}\quad\displaystyle\overline d(\nu_{\widehat\alpha},x)=\limsup_{s\to\infty} \frac{\log_2(\nu_{\widehat\alpha}(I_{n(s)}(x))}{-n(s)}
\end{equation}
for all $x\in K_{\widehat\alpha}$, for the same reasons as those leading to \eqref{dimlocid}. 

\subsubsection{Lower bounds for the dimensions}
Suppose that we have proven that for all $0\le\alpha\le \beta\le \infty$ we have $\underline f^{\mathrm{LD}}_\mu(\alpha,\beta)= \max\{f(\alpha'):\alpha'\in[\alpha,\beta]\}$, a property which will be  established in Section~\ref{fab}. Then, the lower bounds of Theorem~\ref{thmHP}(3) follow readily from the mass distribution principle (see Section~\ref{MDP}), property \eqref{dimlocid2},  and the following proposition, which is a direct consequence of the estimates \eqref{expo} and \eqref{expo'}, and the assumptions \eqref{condHP}(4-5). 
\begin{pro}\label{LB}
With the notations of the previous section, fix $ 0\le \alpha\le \beta\le \infty$ such that  $[\alpha,\beta]\subset \mathcal{J}$ and $[\alpha,\beta]\cap \mathcal{I}\neq\emptyset$. Let $\alpha'=\mathrm{argmax}(f_{|[\alpha,\beta]})$ and $\beta'=\mathrm{argmax}(g_{|[\alpha,\beta]})$. 

Fix a sequence $\widehat\alpha=(\alpha^f_1,\alpha^g_1,\ldots,\alpha^f_m,\alpha^g_m,\ldots)\in \prod_{m=1}^\infty A^f_m\times A^g_m$ such that 
$\lim_{m\to\infty}\alpha^f_m=\alpha'$, $\lim_{m\to\infty}\gamma(\alpha^f_m)=\gamma(\alpha')$, $\lim_{m\to\infty}\alpha^g_{3m-2}=\alpha$, $\lim_{m\to\infty}\gamma(\alpha^g_{3m-2})=g(\alpha)$, $\lim_{m\to\infty}\alpha^g_{3m-1}=\beta$, $\lim_{m\to\infty}\gamma(\alpha^g_{3m-1})=g(\beta)$, $\lim_{m\to\infty}\alpha^g_{3m}=\beta'$, and $\lim_{m\to\infty}\gamma(\alpha^g_{3m})=g(\beta')$. Then for all $x\in K_{\widehat\alpha}$, one has 
\begin{equation*}
\begin{split}
\underline d(\mu,x)=\alpha,\  \overline d(\mu,x)=\beta,\ 
\underline d(\nu_{\widehat\alpha},x)=\min\{f(\alpha'),g(\alpha),g(\beta)\},\text{ and }\overline d(\nu_{\widehat\alpha},x)=g(\beta').
\end{split}
\end{equation*}
Consequently, $\nu_{\widehat\alpha}(E(\mu,\alpha,\beta))=1$,  $\dim_H  \nu_{\widehat\alpha}= \min\{f(\alpha'),g(\alpha),g(\beta)\}$ and $\dim_P \nu_{\widehat\alpha}=g(\beta')$, so that $\dim_H E(\mu,\alpha,\beta)\ge \min\{f(\alpha'),g(\alpha),g(\beta)\}$ and $\dim_H E(\mu,\alpha,\beta)\ge f(\beta')$. 
\end{pro}

\subsection{Large deviations spectra and $L^q$-spectra}
\subsubsection{The large deviations spectra ${\underline{f}}^{\mathrm {LD}}_\mu(\alpha)$ and $\overline f^{\mathrm {LD}}_\mu(\alpha)$}\label{fa} It is clear from the construction of $\mu$,  (\ref{dimlocmu0},\ref{dimlocmu0'}) and \eqref{condHP} that $\underline{f}^{\rm LD}_\mu(\alpha)=-\infty$ if $\alpha\not\in \mathcal I$ (take $s=s_m$ and use \eqref{dls1}) and $\overline{ f}^{\rm LD}_\mu(\alpha)=-\infty$ if $\alpha\not \in \mathcal{J}$ (use \eqref{dimlocmu0} and \eqref{dimlocmu0'}). Moreover, by \eqref{MuFo1} we have $\underline{f}^{\rm LD}_\mu(\alpha)\ge \dim_H E(\mu,\alpha)$, hence $\underline{f}^{\rm LD}_\mu(\alpha)\ge f(\alpha)$ for all $\alpha\in\mathcal{I}$ by Proposition~\ref{LB}. Similarly,  $\overline{f}^{\rm LD}_\mu(\alpha)\ge \dim_P E(\mu,\alpha)$, hence Proposition~\ref{LB} yields $\overline{f}^{\rm LD}_\mu(\alpha)\ge g(\alpha)$ for all $\alpha\in\mathcal{I}$.  

Let us show that $\overline{f}^{\rm LD}_\mu(\alpha)\le g(\alpha)$ for $\alpha\in \mathcal {J}$. 

Suppose first that  $\alpha\in \mathcal {J}\setminus\{\infty\}$. Fix $\eta>0$ and $\epsilon_\eta>0$ such that $g(\alpha')\le g(\alpha)+\eta$ if $\alpha'\in[\alpha-2\epsilon_\eta,\alpha+2\epsilon_\eta]$.  Fix $\epsilon\in (0,\epsilon_\eta)$. If $s'_{m-1}<s\le s_m$, due to \eqref{dimlocmu0}, for $s$ large enough, if  $I_s\in \mathbf{G}_s$ satisfies $2^{-n(s)(\alpha+\epsilon)}\le  \mu(I_s)\le 2^{-n(s)(\alpha-\epsilon)}$, then for any $x\in I_s=I_{n(s)}(x)$ we have $\alpha-2\epsilon\le \alpha_s:= \alpha_s(x)\le \alpha+2\epsilon$,
and the exponents $\alpha^f_{i,j}(x)$ and $\alpha^g_{i,j}(x)$ do not depend on $x\in I_s$. 

Due to the multiplicative structure of the construction of $\mu$, for each such collection of exponents $\{\alpha^f_{i,j}, \alpha^g_{i,j}\}$, the set $\mathbf{G}_s(\{\alpha^f_{i,j}, \alpha^g_{i,j}\})$ of those dyadic cubes $I_s\in \mathbf{G}_s$ such that  $\alpha^h_{i,j}(x)=\alpha^h_{i,j}$ for all $x\in I_s$ and $h\in\{f,g\}$  is such that  (setting $s'=s-s'_{m-1}$)
\begin{eqnarray*}
\#\mathbf{G}_s(\{\alpha^f_{i,j}, \alpha^g_{i,j}\})&=& \prod_{j=1}^{s'} \#G^f_{m}(\alpha^f_{m,j}) \prod_{i=1}^{m-1}\Big (\prod_{j=1}^{R^f_i} \#G^f_{i}(\alpha^f_{i,j})\Big )\prod_{j=1}^{R^g_i} \#G^g_{j}(\alpha^g_{i,j})\\
&\le &\prod_{j=1}^{s'}2^{N_m(\gamma^f_m(\alpha^f_{m,j})+\epsilon_m)} \prod_{i=1}^{m-1}\Big (\prod_{j=1}^{R^f_i}2^{N_i(\gamma^f_i(\alpha^f_{i,j})+\epsilon_i)}\Big )\prod_{j=1}^{R^g_i} 2^{N_i(\gamma^g_i(\alpha^g_{i,j})+\epsilon_i)}
\end{eqnarray*}
 Since for $\alpha\in A^h_i$ by construction we have $\gamma^h_i(\alpha)\le h(\alpha)+2\epsilon_i^{1/3}$, we get 
\begin{eqnarray*}
\#\mathbf{G}_s(\{\alpha^f_{i,j}, \alpha^g_{i,j}\})&\le & 2^{\eta(s)n'(s)}\prod_{j=1}^{s'}2^{N_m f(\alpha^f_{m,j})} \prod_{i=1}^{m-1}\Big (\prod_{j=1}^{R^f_i}2^{N_if(\alpha^f_{i,j})}\Big )\prod_{j=1}^{R^g_i} 2^{N_i g(\alpha^g_{i,j})}\\
&\le &2^{\eta(s)n'(s)}\prod_{j=1}^{s'}2^{N_m g(\alpha^f_{m,j})} \prod_{i=1}^{m-1}\Big (\prod_{j=1}^{R^f_i}2^{N_ig(\alpha^f_{i,j})}\Big )\prod_{j=1}^{R^g_i} 2^{N_i g(\alpha^g_{i,j})}
\end{eqnarray*}
with $\lim_{s\to\infty}\eta(s)=0$ and $n'(s)$ defined like in \eqref{n'1}. 
Now recall that $g$ is concave over $\mathcal{J}\setminus\{\infty\}$. Thus 
\begin{multline*}
\sum_{j=1}^{s'} N_{m} g(\alpha^f_{m,j})+\sum_{j=1}^{R^f_i}N_i g(\alpha^f_{i,j})+\sum_{j=1}^{R^g_i}N_i g(\alpha^g_{i,j})\\
\le n'(s) g\left (\frac{ \sum_{j=1}^{s'} N_m\alpha_{m,j} +\sum_{i=1}^{m-1}\sum_{j=1}^{R^f_i}N_i\alpha^f_{i,j}+\sum_{j=1}^{R^g_i}N_i\alpha^g_{i,j}}{n'(s)}\right )\\
=n'(s)g(\alpha_s)\le n'(s) (g(\alpha)+\eta).
\end{multline*}
Consequently, due to Proposition~\ref{n'} we get 
$$
 \#\mathbf{G}_s(\{\alpha^f_{i,j}, \alpha^g_{i,j}\})\le 2^{(g(\alpha)+\eta+\eta(s))n(s)}.
$$
Moreover, the number of such collections cannot exceed the total number of  possible realisations of such a family when the condition $\alpha_s\in[\alpha-\epsilon_\eta,\alpha+\epsilon_\eta]$ is dropped, which by construction is equal to   $(\#A^f_{m})^{s'} \prod_{i=1}^{m-1}(\# A^f_i)^{R^f_i}(\# A^g_i)^{R^g_i}=2^{\eta'(s)n(s)}$, with $\lim_{s\to\infty}\eta'(s)=0$, by \eqref{condHP}(2)(3).   We can conclude that 
$$
\#\Big \{I_s\in\mathbf{G}_s: 2^{-n(s)(\alpha+\epsilon)}\le  \mu(I_s)\le 2^{-n(s)(\alpha-\epsilon)}\Big \}\le 2^{(g(\alpha)+\eta+\eta(s)+\eta'(s))n(s)}.
$$
The same estimates hold if $s_m<s\le s'_m$, and this yields $\overline{f}^{\rm LD}_\mu(\alpha)\le f(\alpha)+\eta$. Since this holds for all $\eta>0$, the have the desired conclusion.

Now suppose that $\alpha=\infty\in \mathcal {J}$. Since $g(\infty)\ge \sup\{g(\alpha):\alpha\in \mathcal {J}\setminus\{\infty\}\}$, with the same notations as above, the only change is that for any $A>0$ we must consider those intervals in $\mathbf{G}_s$ such that $\alpha_s\ge A$, and conditioning on the realization of $\{\alpha^f_{i,j}(x),\alpha^g_{i,j}(x)\}$, the same calculations as above yield, even without using the concavity of $g$,  $\#\Big \{I_s\in\mathbf{G}_s:   \mu(I_s)\le 2^{-n(s) A}\Big \}\le 2^{(g(\infty)+\eta+\eta(s)+\eta'(s))n(s)}$, hence the result. 

Let us prove that $\underline{f}^{\rm LD}_\mu(\alpha)\le f(\alpha)$ for $\alpha\in\mathcal{I}$. 

Suppose first that $\alpha<\infty$. Fix $\eta>0$ and $\epsilon_\eta>0$ such that $f(\alpha')\le f(\alpha)+\eta$ if $\alpha'\in[\alpha-2\epsilon_\eta,\alpha+2\epsilon_\eta]$.  Fix $\epsilon\in (0,\epsilon_\eta)$.  Suppose that $s=s_m$. Due to \eqref{dls1}, for $s$ large enough, if  $I_s\in \mathbf{G}_s$ satisfies $2^{-n(s)(\alpha+\epsilon)}\le  \mu(I_s)\le 2^{-n(s)(\alpha-\epsilon)}$, then for any $x\in I_s=I_{n(s)}(x)$ we have $\alpha-2\epsilon\le \widetilde \alpha_s:= \frac{\sum_{j=1}^{R^f_m}N_{m}\alpha^f_{m,j}(x)}{R^f_mN_m}\le \alpha+2\epsilon$. Also, given $\{\alpha^f_{i,j}, \alpha^g_{i,j}\}$, the set $\mathbf{G}_s(\{\alpha^f_{i,j}, \alpha^g_{i,j}\})$ of those dyadic cubes $I_s\in \mathbf{G}_s$ such that  $\alpha^h_{i,j}(x)=\alpha^h_{i,j}$ for all $x\in I_s$ and $h\in\{f,g\}$ is bounded as above by $\prod_{j=1}^{R^f_m}2^{N_m(\gamma^f_m(\alpha^f_{m,j})+\epsilon_m)} \prod_{i=1}^{m-1}\Big (\prod_{j=1}^{R^f_i}2^{N_i(\gamma^f_i(\alpha^f_{i,j})+\epsilon_i)}\Big )\prod_{j=1}^{R^g_i} 2^{N_i(\gamma^g_i(\alpha^g_{i,j})+\epsilon_i)}$, which yields 
$$
 \#\mathbf{G}_s(\{\alpha^f_{i,j}, \alpha^g_{i,j}\})\le 2^{o(R^f_mN_m)}2^{\sum_{j=1}^{R^f_m}N_m \gamma^f_m(\alpha^f_{m,j})}\le 2^{o(R^f_mN_m)}2^{\sum_{j=1}^{R^f_m}N_m f(\alpha^f_{m,j})}
 $$
 due to \eqref{condHP}(4). Using the concavity of $f$ this implies
$$
  \#\mathbf{G}_s(\{\alpha^f_{i,j}, \alpha^g_{i,j}\})\le 2^{R^f_mN_m(f(\widetilde \alpha_s)+o(1))}\le 2^{R^f_mN_m(f(\alpha)+2\eta)}
$$
  for $s$ large enough. Finally, counting the possible number of collections $\{\alpha^f_{i,j}, \alpha^g_{i,j}\}$ as above yields 
 \begin{equation}\label{fa'}
\#\Big \{I_s\in\mathbf{G}_{s}: 2^{-n(s)(\alpha+\epsilon)}\le  \mu(I_s)\le 2^{-n(s)(\alpha-\epsilon)}\Big \}\le 2^{(f(\alpha)+3\eta)n(s)}
 \end{equation}
for $s=s_m$ large enough.  Since $\eta$ is arbitrary, this is enough to conclude. 

If $\alpha=\infty$, the previous estimates with $\gamma^f_m(\alpha_{m,j})$ bounded by $f(\infty)+2\epsilon_m$ yield the conclusion.

\medskip

It remains to prove that $\overline{f}^{\rm LD}_\mu(\alpha)\ge g(\alpha)$ for $\alpha\in \mathcal {J}\setminus \mathcal {I}$. In fact the argument is valid for all $\alpha\in \mathcal {J}$. Let $ \alpha\in \mathcal {J}$. Suppose first that $\alpha<\infty$.  Choose a sequence $\widehat \alpha=(\alpha^f_1,\alpha^g_1,\ldots,\alpha^f_m,\alpha^g_m,\ldots)\in \prod_{m=1}^\infty A^f_m\times A^g_m$ such that $\lim_{m\to\infty}\alpha^g_m=\alpha$ and $\lim_{m\to\infty} \gamma^g_m(\alpha^g_m)=g(\alpha)$.  We leave the reader check, that for there exists a sequence $\epsilon''(s)$ converging to $0$ as $s=s'_m$  tends to $\infty$ such that, for all $x\in K_{\widehat\alpha}$ (the Cantor set constructed in Section~\ref{secaux}), we have $\alpha_s(x)=\alpha+\epsilon(s)$ (this uses  \eqref{dls2}) and $\#\{\mathbf{G}_{\widehat\alpha,s}=\prod_{i=1}^m(\#G^f(i,\alpha^f))^{R^f_i}(\#G^f(i,\alpha^g))^{R^f_i}\ge 2^{n(s)(g(\alpha)-\epsilon''(s))}$ (this uses the left hand side of \eqref{cardG}, and \eqref{condHP}(5)). This implies that  $\overline{f}^{\rm LD}_\mu(\alpha)\ge g(\alpha)$. 

If $\alpha=\infty$, the same argument with $\alpha^g_m=\alpha^g_m(\infty)$ yields the desired lower bound. 

We notice that a similar argument would give another proof of $\underline{f}^{\rm LD}_\mu(\alpha)\ge f(\alpha)$ for $\alpha\in\mathcal{I}$. 
\begin{rem}\label{crucialrem}
In fact, the arguments developed in this section provide us with the following precious information: for $\alpha\in \R_+$ we have
\begin{eqnarray*}
f(\alpha)&=&\lim_{\epsilon\to 0^+} \liminf_{m\to \infty}\frac{\log_2 \#\Big \{I\in \mathbf{G}_{s_m}:2^{-n(s_m)(\alpha+\epsilon)}\le  \mu(I)\le2^{-n(s_m)(\alpha-\epsilon)}\Big\}}{n(s_m)}\\
&=&\lim_{\epsilon\to 0^+} \limsup_{m\to \infty}\frac{\log_2 \#\Big \{I\in \mathbf{G}_{s_m}:2^{-n(s_m)(\alpha+\epsilon)}\le  \mu(I)\le2^{-n(s_m)(\alpha-\epsilon)}\Big\}}{n(s_m)};\\
\\
g(\alpha)&=&\lim_{\epsilon\to 0^+} \liminf_{m\to \infty}\frac{\log_2 \#\Big \{I\in \mathbf{G}_{s'_m}:2^{-n(s'_m)(\alpha+\epsilon)}\le  \mu(I)\le2^{-n(s'_m)(\alpha-\epsilon)}\Big\}}{n(s'_m)}\\
&=&\lim_{\epsilon\to 0^+} \limsup_{m\to \infty}\frac{\log_2 \#\Big \{I\in \mathbf{G}_{s'_m}:2^{-n(s'_m)(\alpha+\epsilon)}\le  \mu(I)\le2^{-n(s'_m)(\alpha-\epsilon)}\Big\}}{n(s'_m)},
\end{eqnarray*}
and we also have 
\begin{eqnarray*}
f(\infty)&=&\lim_{A\to\infty} \liminf_{m\to \infty}\frac{\log_2 \#\Big \{I\in \mathbf{G}_{s_m}:  \mu(I)\le2^{-n(s_m)A}\Big\}}{n(s_m)}\\
&=&\lim_{A\to\infty} \limsup_{m\to \infty}\frac{\log_2 \#\Big \{I\in \mathbf{G}_{s_m}: \mu(I)\le2^{-n(s_m)A}\Big\}}{n(s_m)}
\end{eqnarray*}
and 
\begin{eqnarray*}
g(\infty)&=&\lim_{A\to\infty}\liminf_{m\to \infty}\frac{\log_2 \#\Big \{I\in \mathbf{G}_{s'_m}: \mu(I)\le2^{-n(s'_m)A}\Big\}}{n(s'_m)}\\
&=&\lim_{A\to\infty} \limsup_{m\to \infty}\frac{\log_2 \#\Big \{I\in \mathbf{G}_{s'_m}: \mu(I)\le2^{-n(s'_m)A}\Big\}}{n(s'_m)}.
\end{eqnarray*}
\end{rem}

\subsubsection{The functions $\tau_\mu$ and $\overline \tau_\mu$} 

It follows from Remark~\ref{crucialrem} and standard large deviations estimates similar to those used for instance in the proof of \cite[Theorem 4.2]{Riedi}, that for all $q\in\R$ we have 
\begin{equation*}
\lim_{m\to\infty} \frac{\log_2 \sum_{I\in \mathbf{G}_{s_m}}\mu(I)^q}{-n(s_m)}=f^*(q)\quad\text{and}\quad \lim_{m\to\infty} \frac{\log_2 \sum_{I\in \mathbf{G}_{s'_m}}\mu(I)^q}{-n(s'_m)}=g^*(q)
\end{equation*}
(in case the domains of $f$ and $g$ are compact, this is a direct consequence of Laplace-Varadhan's integral lemma  \cite[Theorem 4.3.1]{De-Zei}). Consequently, due to \eqref{toto} we get $\tau_\mu=g^*$ and $\overline\tau_\mu=f^*$. Moreover, it is direct from the  duality between upper semi-continuous concave functions \cite[Theorem 12.2, Corollary 12.2.2]{Roc} that if $\infty\not\in \mathrm{dom}(f)$, we have  $\tau_\mu^*=g$ and $\overline\tau_\mu^*=f$, and if $\infty\in \mathrm{dom}(f)$, then $\mathrm{dom}(\tau_\mu=f^*)=\R_+$, and $\tau_\mu^*$ coincides with $f$ over $\R$. Then, by our definition of the extended concave conjugate function, we have  $\tau_\mu^*(\infty)=-\tau_\mu(0)=f(\infty)$.  

\subsubsection{The large deviations spectrum $\underline{f}^{\mathrm {LD}}_\mu(\alpha,\beta)$}\label{fab}

Recall that for $0\le \alpha\le\beta\le \infty$, $f(\alpha,\beta)$ is defined as $\max\{f(\alpha'):\alpha'\in[\alpha,\beta]\}$.

If $[\alpha,\beta]\subset \R_+\cup\{\infty\}\setminus \mathcal {I}$, due to \eqref{dls1}, we have $\underline f^{\mathrm {LD}}_\mu(\alpha,\beta)=-\infty=f(\alpha,\beta)$.

 Assume now  that  $[\alpha,\beta]\cap  \mathcal {I}\neq\emptyset$, and $(\alpha,\beta)\neq (\infty,\infty)$.  Denote the interval $[\alpha,\beta]\cap  \mathcal {I}$ by $[\alpha_1,\beta_1]$, and notice that $f(\alpha,\beta)=f(\alpha_1,\beta_1)$.
 
  Suppose first that $\beta_1<\infty$. Let $\eta>0$ and for each $\alpha'\in  [\alpha_1,\beta_1]$ let $\epsilon(\alpha')>0$ such that $f(\beta)\le f(\alpha')+\eta$ for all $\beta\in[\alpha'-\epsilon(\alpha'),\alpha'+\epsilon(\alpha')]$. There exists $\alpha'_1,\ldots,\alpha'_N$ in $ [\alpha_1,\beta_1]$ such that $ [\alpha_1,\beta_1]\subset \bigcup_{i=1}^N [\alpha_i'-\epsilon(\alpha'_i),\alpha'_i-\epsilon(\alpha'_i)]$. Let $\epsilon\le \min\{\epsilon(\alpha'_i):1\le i\le N\}$. Property  \eqref{dls1} implies that for $m$ large enough, if $I\in\mathbf{G}_{s_m}$ and $2^{-n(s_m)(\beta+\epsilon)} \le \mu(I)\le 2^{-n(s_m)(\alpha-\epsilon)}$, since there exists $x\in K$ such that $I=I_{n(s_m)}(x)$, in fact there exists $1\le i\le N$ such that $2^{-n(s_m)(\alpha'_i+\epsilon(\alpha'_i))} \le \mu(I)\le 2^{-n(s_m)(\alpha'_i-\epsilon(\alpha'_i))}$. Then, the estimate \eqref{fa'} achieved in Section~\ref{fa} yields
 $\#\Big \{I\in\mathbf{G}_{s_m}: 2^{-n(s_m)(\alpha'_i+\epsilon(\alpha'_i))} \le \mu(I)\le 2^{-n(s_m)(\alpha'_i-\epsilon(\alpha'_i))}\Big \}\le 2^{(f(\alpha'_i)+3\eta)n(s_m)}$ for $m$ large enough.  
This implies that for $m$ large enough 
$$
\#\Big \{I\in\mathbf{G}_{s_m}:  2^{-n(s_m)(\beta+\epsilon)} \le \mu(I)\le 2^{-n(s_m)(\alpha-\epsilon)}\Big\}\le \sum_{i=1}^N2^{(f(\alpha'_i)+3\eta)n(s_m)}\le N 2^{(f(\alpha,\beta)+3\eta)n(s_m)}.
$$
If follows that $\underline f^{\mathrm {LD}}_\mu(\alpha,\beta)\le f(\alpha,\beta)+3\eta$ for all $\eta>0$, hence the desired upper bound $f(\alpha,\beta)$ for $\underline f^{\mathrm {LD}}_\mu(\alpha,\beta)$. 

For the lower bound, we just use the fact that by Proposition~\ref{LB} we know that if $\alpha'=\mathrm{argmax}(f_{|[\alpha,\beta]})$ we have $f(\alpha')\le \dim_H E(\mu,\alpha')$,  and on the other hand $ \dim_H E(\mu,\alpha')\le \underline f_\mu^{\mathrm{LD}}(\alpha')\le \underline f_\mu^{\mathrm{LD}}(\alpha,\beta)$, the last inequality being obvious.  

If $\beta_1=\infty$,  the upper bound for $ \underline f_\mu^{\mathrm{LD}}(\alpha,\beta)$ just comes from the observation already done in Section~\ref{fa} that if $s=s_m$, we have $\#\mathbf{G}_s\le 2^{n(s)(f(\infty)+o(1))}$, and the lower bound comes from the lower bound $f(\infty)$ for $\dim_HE(\mu,\infty)$. 

\subsection{Upper bounds for the different spectra}
It is a direct application of Proposition~\ref{MFcomp}, using the fact that $\underline f^{\mathrm {LD}}_\mu(\alpha,\beta)=f(\alpha,\beta)$. 

\subsection{$\mu$ is exact dimensional, with dimension $D$}
Fix $\epsilon>0$. For each $s\ge 1$, if $s'_{m-1}< s\le s_m$ and $s'=s-s'_{m-1}$, an application of Markov's inequality yields, for any $\eta>0$:
\begin{eqnarray*}
&&\mu\Big (E_{s,+}=\Big \{x\in K: \frac{\mu(I_{n(s)}(x))}{\mu(I_{n(s'_{m-1})}(x))}\ge 2^{-(n'(s)-n'(s'_{m-1})) (D_m-\epsilon)}\Big \}\Big )\\
&\le&\sum_{I_{n(s)}\in \mathbf{G}_s}\mu(I_{n(s)}) \Big (\frac{\mu(I_{n(s)})}{\mu(I_{n(s'_{m-1})})}\Big )^\eta 2^{\eta (n'(s)-n'(s'_{m-1})) (D_m-\epsilon)}\\
&=& \sum_{I_{n(s)}\in \mathbf{G}_s}\mu(I_{n(s'_{m-1})}) \Big (\frac{\mu(I_{n(s)})}{\mu(I_{n(s'_{m-1})})}\Big )^{1+\eta} 2^{\eta (n'(s)-n'(s'_{m-1})) (D_m-\epsilon)},
\end{eqnarray*}
where $I_{n(s'_{m-1})}$ stands for the unique element of $\mathbf{G}_{s'_{m-1}}$ containing $I_{n(s)}$. We notice that by construction, given $I_{n(s'_{m-1})}$ in $\mathbf{G}_{s'_{m-1}}$, the distribution of the collection $\Big \{\frac{\mu(I_{n(s)})}{\mu(I_{n(s'_{m-1})})}\Big\}$, where $I_{n(s)}\in \mathbf{G}_s$ and $I_{n(s)}\subset I_{n(s'_{m-1})}$, is independent of $I_{n(s'_{m-1})}$, and  taking into account the fact that between the steps $s'_{m-1}$ and $s$ one uses $s'$ times the same motive in the recursion defining $\mu$, we can get   
\begin{eqnarray*}
\mu(E_{s,+})&\le &
\frac{ \Big (\sum_{\alpha\in A^f_m}\sum_{I\in G^f(m,\alpha)}\rho_m^f(\alpha)^{1+\eta}\mu_\alpha(I)^{1+\eta}\Big )^{s'}}{\Big (\sum_{\alpha\in A^f_m}\sum_{I\in G^f(m,\alpha)}\rho_m^f(\alpha)\mu_\alpha(I)\Big )^{(1+\eta)s'}} 2^{\eta (n'(s)-n'(s'_{m-1})) (D_m-\epsilon)}\\
&\le& 2^{(1+\eta)s'}\Big (\sum_{\alpha\in A^f_m}\sum_{I\in G^f(m,\alpha)}\rho_m^f(\alpha)^{1+\eta}\mu_\alpha(I)^{1+\eta}\Big )^{s'}2^{\eta (n'(s)-n'(s'_{m-1})) (D_m-\epsilon)}.
\end{eqnarray*}

We have 
\begin{multline*}
\sum_{I\in G^f(m,D_m)}\rho_m^f(D_m)^{1+\eta}\mu_{D_m}(I)^{1+\eta}=\sum_{I\in G^f(m,D_m)}\mu_{D_m}(I)^{1+\eta}\\ \le (\#G^f(m,D_m))2^{-N_m(D_m-\epsilon_m)(1+\eta)}\le 2^{N_m(D_m+\epsilon_m)}2^{-N_m(D_m-\epsilon_m)(1+\eta)}\\
=2^{-N_mD_m\eta}2^{N_m\epsilon_m(2+\eta)}.
\end{multline*}
On the other hand, if $A^f_m\setminus \{D_m\}\neq\emptyset$, fixing one of its elements $\alpha_0$, we have 
\begin{multline*}
\sum_{\alpha\in A^f_m\setminus\{D_m\}}\sum_{I\in G^f(m,\alpha)}\rho_m^f(\alpha)^{1+\eta}\mu_\alpha(I)^{1+\eta}\le \sum_{\alpha\in A^f_m\setminus\{D_m\}}\rho_m^f(\alpha)^{1+\eta}(\#G^f(m,\alpha))2^{-N_m(\alpha-\epsilon_m)(1+\eta)}\\\le \sum_{\alpha\in A^f_m\setminus\{D_m\}}\rho_m^f(\alpha)^{1+\eta}2^{N_m(\gamma^f_m(\alpha)+\epsilon_m)}2^{-N_m(\alpha-\epsilon_m)(1+\eta)}\\\le  \rho_m^f(\alpha_0)^{1+\eta}(\#A_m^f) (\sup_{\alpha\in A^f_m\setminus\{D_m\}} 2^{N_m(\gamma^f_m(\alpha) -\alpha)}) 2^{N_m\epsilon_m(2+\eta)}.
\end{multline*}
Now, take $\eta=\eta_m=\sqrt{\epsilon_m}$. Due to the definition \eqref{rhof}  of $\rho_m^f$ we have $\rho_m^f(\alpha_0)^{1+\eta}(\#A_m^f) \le 1$, and due to  \eqref{thetam}, we have $\sup_{\alpha\in A^f_m\setminus\{D_m\}} 2^{N_m(\gamma^f_m(\alpha) -\alpha)}\le 2^{-N_mD_m\eta_m}$.  Finally,
\begin{eqnarray}
\nonumber\mu (E_{s,+})
\nonumber&\le& 2^{s'}\cdot 2^{(1+\eta_m)s'} \big (2^{-N_mD_m\eta_m}2^{N_m\epsilon_m(2+\eta_m)}\big )^{s'} 2^{\eta_m (n'(s)-n'(s'_{m-1})) (D_m-\epsilon)}\\
\nonumber&=& 2^{s'}\cdot 2^{(1+\eta_m)s'} \big (2^{-N_mD_m\eta_m}2^{N_m\epsilon_m(2+\eta_m)}\big )^{s'} 2^{\eta_m N_m s' (D_m-\epsilon)}
\le  2^{3s'}2^{-N_m s'\eta_m(\epsilon-3\eta_m)}. 
\end{eqnarray}
Also, using a similar estimate as above, and  with the same choice $\eta=\eta_m=\sqrt{\epsilon_m}$, we have 
\begin{eqnarray*}
&&\mu\Big (\Big \{x\in K: \frac{\mu(I_{n(s)}(x))}{\mu(I_{n(s'_{m-1})}(x))}\le 2^{-(n'(s)-n'(s'_{m-1})) (D_m+\epsilon)}\Big \}\Big )\\
&\le&\frac{ \Big (\sum_{\alpha\in A^f_m}\sum_{I\in G^f(m,\alpha)}\rho_m^f(\alpha)^{1-\eta}\mu_\alpha(I)^{1-\eta}\Big )^{s'}}{\Big (\sum_{\alpha\in A^f_m}\sum_{I\in G^f(m,\alpha)}\rho_m^f(\alpha)\mu_\alpha(I)\Big )^{(1-\eta)s'}} 2^{-\eta (n'(s)-n'(s'_{m-1})) (D_m+\epsilon)}\\
&\le& 2^{(1-\eta)s'}\Big (\sum_{\alpha\in A^f_m}\sum_{I\in G^f(m,\alpha)}\rho_m^f(\alpha)^{1-\eta}\mu_\alpha(I)^{1-\eta}\Big )^{s'}2^{-\eta (n'(s)-n'(s'_{m-1})) (D_m+\epsilon)}.
\end{eqnarray*}
On the one hand, we have 
\begin{multline*}
\sum_{I\in G^f(m,D_m)}\rho_m^f(D_m)^{1-\eta}\mu_{D_m}(I)^{1-\eta}=\sum_{I\in G^f(m,D_m)}\mu_{D_m}(I)^{1-\eta}\\ \le 2^{N_m(D_m+\epsilon_m)}2^{-N_m(D_m-\epsilon_m)(1-\eta)}=2^{N_mD_m\eta}2^{2N_m\epsilon_m}.
\end{multline*}
On the other hand, if $A^f_m\setminus \{D_m\}\neq\emptyset$, fixing one of its elements $\alpha_0$, we have 
\begin{multline*}
\sum_{\alpha\in A^f_m\setminus\{D_m\}}\sum_{I\in G^f(m,\alpha)}\rho_m^f(\alpha)^{1-\eta}\mu_\alpha(I)^{1-\eta}\\\le \sum_{\alpha\in A^f_m\setminus\{D_m\}}\rho_m^f(\alpha)^{1-\eta}2^{N_m(\gamma^f_m(\alpha)+\epsilon_m)}2^{-N_m(\alpha-\epsilon_m)(1-\eta)}\\\le  \rho_m^f(\alpha_0)^{1-\eta}(\#A_m^f) (\sup_{\alpha\in A^f_m\setminus\{D_m\}} 2^{N_m(\gamma^f_m(\alpha) -(1-\eta)\alpha)}) 2^{2 N_m\epsilon_m}.
\end{multline*}
 Due to the definition \eqref{rhof}  of $\rho_m^f$, since $\sqrt{\epsilon_m}=(m+1)^{-1}\le 1/2$,  we have $\rho_m^f(\alpha_0)^{1-\eta_m}(\#A_m^f) \le 1$, and due to  \eqref{thetam}, we have $\sup_{\alpha\in A^f_m\setminus\{D_m\}} 2^{N_m(\gamma^f_m(\alpha) -(1-\eta_m\alpha)}\le 1$.  Consequently, the previous estimates yield
\begin{equation*}\label{muex'}
 \mu\Big (E_{s,-}=\Big \{x\in K: \frac{\mu(I_{n(s)}(x))}{\mu(I_{n(s'_{m-1})}(x))}\le 2^{-(n'(s)-n'(s'_{m-1})) (D_m+\epsilon)}\Big \}\Big )\le 2^{3s'}2^{-N_m s'\eta_m(\epsilon-3\eta_m)}.
 \end{equation*}
 
Similarly, if $s_m<s\le s'_m$ and $s'=s-s_m$ we can get 
\begin{equation*}\label{muex3}
\begin{split}
&\mu\Big (E_{s,+}=\Big \{x\in K: \frac{\mu(I_{n(s)}(x))}{\mu(I_{n(s_{m})}(x))}\ge 2^{-(n'(s)-n'(s_{m})) (D_m-\epsilon)}\Big \}\Big )\le 2^{3s'}2^{-N_m s'\eta_m(\epsilon-3\eta_m)}\\
 &\mu\Big (E_{s,-}=\Big \{x\in K: \frac{\mu(I_{n(s)}(x))}{\mu(I_{n(s_{m})}(x))}\le 2^{-(n'(s)-n'(s_{m})) (D_m+\epsilon)}\Big \}\Big )\le 2^{3s'}2^{-N_m s'\eta_m(\epsilon-3\eta_m)}.
\end{split}
\end{equation*}

Finally, for $m_0$ big enough so that $3\eta_m\le \epsilon/2$, and $N_m\eta_m\epsilon/2>4$ (remember that $N_m\ge e^m$ and $\eta_m=(m+1)^{-1}$), we have 
\begin{multline*}
\sum_{m\ge m_0}\sum_{s'_{m-1}<s\le s_m}\mu(E_{s,+}\cup E_{s,-)}+\sum_{s_{m}<s\le s'_m}\mu(E_{s,+}\cup E_{s,-)} \\\le 2 \sum_{m\ge m_0}\sum_{s'=1}^{R^f_m} 2^{3s'}2^{-N_m s'\eta_m\epsilon/2}+\sum_{s'=1}^{R^g_m} 2^{3s'}2^{-N_m s'\eta_m\epsilon/2}\\
\le 2 \sum_{m\ge m_0} 16 \cdot \frac{2^{-N_m\eta_m\epsilon/2}}{1-8\cdot 2^{-N_m\eta_m\epsilon/2}}\le 64 \sum_{m\ge m_0} 2^{-N_m\eta_m\epsilon/2}<\infty.
\end{multline*}
By the Borel-Cantelli lemma, we deduce that for $\mu$-almost every $x$, there exists  an integer $m_x$ such that for all $m\ge m_x$, for all $s'_{m-1}<s\le s_m$ one has 
$$
2^{-(n'(s)-n'(s'_{m-1})) (D_m+\epsilon)}\le \frac{\mu(I_{n(s)}(x))}{\mu(I_{n(s'_{m-1})}(x))}\le 2^{-(n'(s)-n'(s'_{m-1})) (D_m-\epsilon)},
$$
and for all $s_{m}<s\le s'_m$ one has 
$$
 2^{-(n'(s)-n'(s_{m})) (D_m+\epsilon)}\le \frac{\mu(I_{n(s)}(x))}{\mu(I_{n(s_{m})}(x))}\le 2^{-(n'(s)-n'(s_{m}))(D_m-\epsilon)}.
 $$
Using these inequalities telescopically and noting that $D_m$ converges to $D$ as $m\to\infty$ and $n'(s)\sim n(s)$ as $s\to\infty$ we get $D-\epsilon\le \underline d(\mu,x)\le\overline d(\mu,x) \le D+\epsilon$ for $\mu$-almost every $x$. Since $\epsilon$ is arbitrary, we have the desired exact dimensionality.

\subsection{Restrictions of $\mu$} Let $B$ be a closed ball whose interior intersects $\supp(\mu)$ at some point $x$. Let $\nu=\mu_{|B}$. We naturally have the inequalities $\dim_L(\nu,\alpha,\beta)\le \dim_L(\mu,\alpha,\beta)$ for $L\in\{H,P\}$ and $0\le\alpha\le \beta\le \infty$, as well as $\underline f_\nu^{\mathrm{LD}}(\alpha)\le\underline f_\mu^{\mathrm{LD}}(\alpha)$,  $\overline f_\nu^{\mathrm{LD}}(\alpha)\le\overline f_\mu^{\mathrm{LD}}(\alpha)$, $\underline f_\nu^{\mathrm{LD}}(\alpha,\beta)\le\underline f_\mu^{\mathrm{LD}}(\alpha,\beta)$,  $\tau_\nu\ge \tau_\mu$ and $\overline \tau_\nu\ge \overline\tau_\mu$. On the other hand, denoting by $s_0$ the smallest integer such that  $I_{n(s_0)}(x)\subset \mathrm{Int} (B)$,  we can modify the first terms of  the sequence $\widehat \alpha$ of Proposition~\ref{LB} so that  $I_{n(s_0)}(x)\in\mathcal{G}_{\widehat\alpha,s_0}$, hence  $\nu_{\widehat\alpha}(B\cap K_{\widehat \alpha})>0$, and we have $ \dim_H E(\nu,\alpha,\beta)\ge f(\alpha,\beta)$ and $ \dim_P E(\nu,\alpha,\beta)\ge g(\alpha,\beta)$. This is enough to reverse all the previous inequalities since we also have for $\nu$ the general inequalities provided by \eqref{MuFo1}, \eqref{MuFo2} and Proposition~\ref{MFcomp}.  

\section{Proofs of Propositions~\ref{condtau}(2),~\ref{condtau1} and~\ref{MFcomp}, and some inequalities in \eqref{MuFo1} and \eqref{MuFo2}}\label{proof3}

\subsection{Proofs of Propositions~\ref{condtau}(2) and~\ref{condtau1} } 

\subsubsection{Proposition~\ref{condtau}(2)}  The fact  that $\R_+\subset {\rm dom }(\tau_\mu)$ was explained after we defined $\tau_\mu$ in the introduction. Now suppose at first that there exists $\alpha\in \R_+$ and $r_0>0$ such that for all $r\in (0,r_0)$ and $x\in \supp(\mu)$ we have $\mu(B(x,r))> r^\alpha$. Then,  the definition of the $L^q$-spectrum yields $\tau_\mu(q)\ge \tau_\mu(0)+q\alpha$ for all $q<0$, hence $\tau_\mu$ is finite over $\R$. If, on the contrary,  for all $\alpha>0$, for all $r_0>0$, there exists $r\in (0,r_0)$ and $x\in\supp(\mu)$ such that $\mu(B(x,r))\le r^\alpha$, by using again the definition of the $L^q$-spectrum we have  $\tau_\mu(q)\le \alpha q$ for all $\alpha>0$ and $q<0$, hence $\tau_\mu(q)=-\infty$ for $q<0$, so ${\rm dom }(\tau_\mu)=\R_+$. 

Now let  $\alpha\in{\rm dom}(\tau_\mu^*)$ and suppose that $-\infty<\tau_\mu^*(\alpha)<0$. Necessarily $\alpha<\infty$. Indeed one always has $\tau_\mu^*(\infty)\in\{-\infty,-\tau_\mu(0)\}$. Then, suppose first that $\alpha<\tau_\mu'(0^+)$. Let $\alpha_0=\inf\{\beta\in  (\alpha,\tau_\mu'(0^+)):   \tau_\mu^*(\beta)\ge 0\}$. The continuity of $\tau_\mu^*$ over the interior of its domain implies $\tau_\mu^*(\alpha_0)=0$.  Then for all $\beta<\alpha_0$ we have  $\tau_\mu^*(\beta)<0$, hence $\overline f^{\rm LD}(\beta)<0$. Consequently, for all $\epsilon>0$ there exists $r_0>0$ such that for all $r\in(0,r_0)$ and $x\in \supp(\mu)$ we have $\mu(B(x,r))\le r^{\alpha_0-\epsilon}$. This implies that $\tau_\mu(q)\ge \tau_\mu(0) +\alpha_0 q$ for all $q\ge 0$, and finally $\tau_\mu^*(\beta)\le \inf\{\beta q-\alpha_0 q+\tau_\mu(0):q\ge 0\}=-\infty$ for all $\beta<\alpha_0$, which contradicts the fact that $-\infty<\tau_\mu^*(\alpha)$. Next suppose  $\alpha>\tau_\mu'(0^-)$ and ${\rm dom }(\tau_\mu)=\R$. The same lines as above also yield a contradiction. If now $\alpha\in [-\tau_\mu'(0^+),\tau_\mu'(0^-)]$ and ${\rm dom }(\tau_\mu)=\R$, then $\tau_\mu^*(\alpha)=-\tau_\mu(0)\ge 0$; new contradiction. It remains the case  ${\rm dom }(\tau_\mu)=\R_+$ and $\alpha\ge \tau_\mu'(0^+)$. In this case, we necessarily have $\tau^*_\mu(\alpha)=\lim_{q\to 0^+} -\tau_\mu(q)\ge 0$. Finally,  $\tau_\mu^*$ is non-negative on its domain.

\subsubsection{Proposition~\ref{condtau1}}
(1) If ${\rm dom }(\tau)=\R$, the property ${\rm dom}(\tau^*)=[\tau'(\infty),\tau'(-\infty)]$ follows from standard considerations in convex analysis.  Then, the fact that this interval is bounded from above follows from the boundedness from below of $\tau^*$. Also, since ${\rm dom}(\tau^*)\subset \R$, the equality $(\tau^*)^*=\tau$ on ${\rm dom }(\tau)=\R$ is just the usual duality between $\tau$ and its conjugate function (\cite[Theorem 12.2, Corollary 12.2.2]{Roc}) when this one only defined on $\R$ and not also on $\R\cup\{\infty\}$ as in the convention used in this paper. 

(2)(a) Since $\tau(0)=\tau(1)=0$, by concavity of the non decreasing function $\tau$, we have $\tau=0$ over $\R_+$, and a simple verification shows that $\tau^*=\tau$ over $\R\cup\{\infty\}$. 

(2)(b)  We suppose that $\tau(0)<0$ and $\tau$ is continuous at $0^+$. Here again, standard considerations in convex analysis show that $\min {\rm dom}(\tau^*)=\tau'_\mu(\infty)$ and $[\tau'(\infty),\lim_{q\to 0^+}\tau'(q^-)]\subset{\rm dom}(\tau^*)$,  as well as the continuity and the concavity of $\tau^*$ over $ [\tau'(\infty),\lim_{q\to 0^+}\tau'(q^-)]$. If $\lim_{q\to 0^+}\tau'(q^-)<\infty$, by using the definition of $\tau^*$ one checks that $\tau^*(\alpha)=-\tau(0^+)=-\tau(0)=\tau^*(\infty)$ for all $\alpha\in [\lim_{q\to 0^+}\tau'(q^-),\infty]$. So ${\rm dom}(\tau^*)=[\tau'(\infty),\infty]$. The continuity of $\tau^*$ over ${\rm dom}(\tau^*)$  comes from the fact that $\tau^*(\infty)=-\tau(0)$. The fact that $(\tau^*)^*=\tau$ over $\R_+$ is a direct consequence of the usual duality between $\tau$ and the restriction of $\tau^*$ to $\R$, and the fact that $\tau$ is continuous at $0^+$. The equality $(\tau^*)^*=\tau$ over $\R^*_-$ is obvious.

(2)(c) If $\tau$ is discontinuous at $0^+$, clearly  $\tau'(0^+)=\infty$. It is standard from convex analysis that $\min {\rm dom}(\tau^*)=\tau'_\mu(\infty)$, and $[\tau'(\infty),\lim_{q\to 0^+}\tau'(q^-)]\subset {\rm dom}(\tau^*)$. Moreover, if $\lim_{q\to 0^+}\tau'(q^-)<\infty$, by using the definition of $\tau^*$ one checks that $\tau^*(\alpha)=-\tau(0^+)<-\tau(0)=\tau^*(\infty)$ for all $\alpha\in [\lim_{q\to 0^+}\tau'(q^-),\infty)$.  Consequently, we have  ${\rm dom}(\tau^*)=[\tau'(\infty),\infty]$, as well as the concavity and continuity of $\tau^*$ over $[\tau'_\mu(\infty),\infty)$. By using the usual duality between $\tau$ and the restriction of $\tau^*$ to $\R$, we would find that $(\tau^*)^*$ is equal to $\tau$ over $\R_+^*$ and equal to $\tau(0^+)$ at $0$. Here, taking into account that $\alpha=\infty\in{\rm dom }(\tau^*)$, we find that $(\tau^*)^*(0)=-\tau^*(\infty)=\tau(0)$. Finally,  $(\tau^*)^*=\tau$ on $\R_+$. The equality $(\tau^*)^*=\tau$ over $\R_-$ is obvious.

(2)(d) It has been proved in the previous lines.

\subsection{Proof of Proposition~\ref{MFcomp}}\label{proof33} (1)
Fix $0\le \alpha< \beta\le \infty$ (the case $\alpha=\beta$ is covered by \eqref{MuFo1} and \eqref{MuFo2}). Without loss of generality we assume that $\underline f_\mu^{\mathrm{LD}}(\alpha,\beta)>-\infty$, for otherwise one clearly has $E(\mu,\alpha,\beta)=\emptyset$. 

We first show that $\dim_H F(\alpha,\beta)\le \underline f_\mu^{\mathrm{LD}}(\alpha,\beta)$,
where
$$
F(\alpha,\beta)=\{x\in \supp(\mu): \alpha \le \underline d(\mu,x)\le \overline d(\mu,x)\le \beta\}).
$$ 
Since $E(\mu,\alpha,\beta)\subset F(\alpha,\beta)$, this yields $\dim_HE(\mu,\alpha,\beta)\le \underline f_\mu^{\mathrm{LD}}(\alpha,\beta)$. 

Fix $\eta>0$. There exists $\epsilon>0$ such that for infinitely many $r>0$, we have $f_\mu(\alpha,\beta,\epsilon,r)\le \underline f_\mu^{\mathrm{LD}}(\alpha,\beta)+\eta$. Let $(r_{j})_{j\ge 1}$ be a sequence converging to 0 such that for all $j$ we have $f_\mu(\alpha,\beta,\epsilon,r_{j})\le \underline f_\mu^{\mathrm{LD}}(\alpha,\beta)+\eta$. 

By definition, we have 
$
F(\alpha,\beta)\subset \bigcup_{N\ge 1}   F_N,
$
where 
$$F_N=\bigcap_{0<r\le 2^{-N}}\Big \{x\in\supp(\mu): \, r^{\beta+\epsilon}\le  \mu(B(x,r))\le r^{\alpha-\epsilon}\Big\}.
$$
Fix $N\ge 1$. It follows from the previous line that for any $n\ge N$, there exists $j\ge 1$ such that $r_{j}\le 2^{-n}$  and we have 
$$
F_N\subset \Big \{x\in\supp(\mu): r_j^{\beta+\epsilon}\le  \mu(B(x,r_j))\le r_j^{\alpha-\epsilon}\Big\}.
$$
It follows from Besicovitch's covering theorem (see \cite{Mat}) that there exists an integer $Q(d)$ such that, defining $F_{j}(\epsilon)=\{x\in\supp(\mu): r_j^{\beta+\epsilon}\le  \mu(B(x,r_j))\le r_j^{\alpha-\epsilon}\Big\}$, we can extract from $\{B(x,r_{j}): x\in F_{j}(\epsilon)\}$, $Q(d)$ families $\mathcal F_k$ ($1\le k\le Q(d)$) of disjoint balls such that $F_{j}(\epsilon)\subset \bigcup_{k=1}^{Q(d)} \bigcup_{B\in\mathcal F_k}B$. 

Now, setting $\gamma=\underline f_\mu^{\mathrm{LD}}(\alpha,\beta) +2\eta$, and using the covering of $F_N$ by the balls in $\bigcup_{k=1}^{Q(d)}\mathcal F_k$, we see that for $j$ large enough we have 
\begin{multline*}
\mathcal H^\gamma_{2^{-n+1}}(F_N)\le \sum_{k=1}^{Q(d)}\sum_{B\in\mathcal F_k} |B|^\gamma\\
\le Q(d) (\#\mathcal F_k) (2r_{j})^\gamma
\le 2^\gamma Q(d)  r_{j}^{-(\underline f_\mu^{\mathrm{LD}}(\alpha,\beta)+\eta)+\gamma}\le 2^\gamma Q(d)2^{-n\eta}.
\end{multline*}
Letting $n$ tend to $\infty$ yields $\mathcal H^\gamma(F_N)=0$, so $\dim F_N\le \gamma$ for all $N\ge 1$, and finally $\dim F(\alpha,\beta)\le \underline f_\mu^{\mathrm{LD}}(\alpha,\beta) +2\eta$. Since $\eta$ was arbitrary, we are done. 

Now let us prove that $\dim_H E(\mu,\alpha,\beta)\le \min(\overline f_\mu^{\mathrm{LD}}(\alpha),\overline f_\mu^{\mathrm{LD}}(\beta))$. For $\alpha'\in \R\cup\{\infty\}$, define $G_{\alpha'}=\{x: \exists \ (n_j)\nearrow \infty: \lim_{j\to\infty} \frac{\log (\mu(B(x,2^{-n_j})))}{\log (2^{-n_j})}=\alpha'\}$. We have $E(\mu,\alpha,\beta)\subset  G_{\alpha}\cap G_\beta$. Consequently, the conclusion follows from the fact that $\dim G_{\alpha'}\le \overline f_\mu^{\mathrm{LD}}(\alpha')$. Indeed, if $\alpha<\infty$, fix $\eta>0$. Then let $\epsilon>0$ and $r_0\in\N_+$ such that  for all $0<r\le r_0$ we have $f(\alpha',\epsilon,r)\le  \overline f_\mu^{\mathrm{LD}}(\alpha')+\eta$, where 
$$
f(\alpha',\epsilon,r)=\frac{\log \sup \#\Big\{i:r^{\alpha+\epsilon}\le  \mu(B(x_i,r))\le r^{\alpha-\epsilon}\Big \}}{-\log(r)}, 
$$
the supremum being taken over the packings of $\supp(\mu)$ by balls of radii equal to $r$.  For each $n\ge 1$ such that $2^{-n}\le r_0$,  we have $G_{\alpha'}\subset \bigcup_{p\ge n}G_{\alpha',p}$, where $G_{\alpha',p}=\{x\in\supp(\mu): 2^{-p(\alpha'+\epsilon)}\le \mu(B(x,2^{-p}) \le  2^{-p(\alpha'+\epsilon)}\}$. Setting $\gamma= \overline f_\mu^{\mathrm{LD}}(\alpha')+2\eta$ and using Besocovitch's covering theorem as above, we get $\mathcal{H}^\gamma_{2^{-n+1}}(G_{\alpha'})\le 2^{\gamma}Q(d) \sum_{p\ge n} 2^{p(\overline f_\mu^{\mathrm{LD}}(\alpha')+\eta)-p\gamma}$. Letting $n$ tend to $\infty$ yields $\mathcal{H}^\gamma(G_{\alpha'})=0$, hence $\dim_H G_{\alpha'}\le \overline f_\mu^{\mathrm{LD}}(\alpha')+2\eta$, $\eta>0$ being arbitrary. The case $\alpha'=\infty$ can be treated similarly.

Now we prove that $\dim_P F(\alpha,\beta)\le f_P(\alpha,\beta)=\sup\{\overline f_\mu^{\mathrm{LD}}(\alpha'):\alpha'\in[\alpha,\beta]\}$.  Suppose first that $\beta<\infty$.

Fix $\eta>0$, and for each $\alpha'\in [\alpha,\beta]$ fix $\epsilon(\alpha')>0$ and $r(\alpha')>0$ such that for all $0<r<r(\alpha')$, one has $f(\alpha',\epsilon(\alpha'),r)\le  \overline f_\mu^{\mathrm{LD}}(\alpha')+\eta$. Then let $\alpha'_1,\ldots,\alpha'_k$ such that $[\alpha,\beta]\subset \bigcup_{i=1}^kB(\alpha'_i,\epsilon(\alpha'_i))$. Set $\epsilon=\min\{\epsilon(\alpha'_i):1\le i\le k\}$ and $r_\eta=\min\{r(\alpha'_i):1\le i\le k\}$.  

Fix $N\ge 1$ and define $F_N$ as above. We have 
$$
F_N\subset \bigcap_{ 0<2^{-p}<\min (2^{-N}, r_\eta)}\bigcup_{i=1}^k \Big \{x\in\supp(\mu): 2^{-p(\alpha_i'+\epsilon(\alpha_i'))}\le  \mu(B(x,2^{-p}))\le 2^{-p(\alpha_i'-\epsilon(\alpha_i'))}\Big\}.
$$
Let $A\subset [0,1]^d$. Let $n\in\N$ such that $2^{-n}\le \min (2^{-N}, r_\eta)$ and  $\{B(x_i,r_i)\}$ a $2^{-n}$-packing of $A\cap F_N$. For each integer $p\ge n+1$, set $S_p=\{i: 2^{-p}<r_i\le 2^{-p+1}\}$. The  balls in $\{B(x_i,2^{-p}):i\in S_p, \  2^{-p(\alpha_i'+\epsilon(\alpha_i'))}\le  \mu(B(x,2^{-p}))\le 2^{-p(\alpha_i'-\epsilon(\alpha_i'))}\}$ form a $2^{-p}$-packing of $\supp(\mu)$ of cardinality less than $2^{p(\overline f_\mu^{\mathrm{LD}}(\alpha'_i)+\eta)}$, so $\# S_p\le k 2^{p(f_P(\alpha,\beta)+\eta)}$.

Let $\gamma>f_P(\alpha,\beta)+2\eta$. We have
$$
\sum_{i} (2r_i)^\gamma\le \sum_{p\ge n}\sum_{i\in S_p} (2\cdot 2^{-p+1})^\gamma\le 4^\gamma \sum_{p\ge n}( \# S_p) 2^{-p\gamma}\le  4^\gamma k \sum_{p\ge n}2^{p(f(\alpha,\beta)+\eta-\gamma)}\le   4^\gamma k \sum_{p\ge n} 2^{-p\eta},
$$
the upper bound being  independent of the $2^{-n}$-packing $\{B(x_i,r_i)\}$, and going to $0$ as $n\to\infty$. It follows that the pre-packing measure $\overline P^\gamma(A\cap F_N)$ equals 0 for all $A$, hence $\mathcal P^\gamma (F_N)$=0, so $\dim_P(F_N)\le \gamma$. Since $\eta$ is arbitrary this yields $\dim_PF_N\le f_P(\alpha,\beta)$, hence  $\dim_PF(\alpha,\beta)\le f_P(\alpha,\beta)$, and finally $\dim_P E(\mu,\alpha,\beta)\le f_P(\alpha,\beta)$. 

If $\beta=\infty$, take $\beta_1>0$ and $r(\beta)$ such that for all $0<r\le r(\infty)$, $f(\beta_1,\infty,r)\le \overline f_\mu^{\mathrm{LD}}(\infty)+\eta$, where 
$$
f(\beta_1,\infty,r)=\frac{\log \sup \#\Big\{i: \mu(B(x_i,r))\le r^{\beta_1}\Big \}}{-\log(r)}, 
$$
the supremum being taken over the packings of $\supp(\mu)$ by balls of radii equal to $r$. Then use a covering of $[\alpha,\beta_1]$ by intervals $B(\alpha'_i,\epsilon(\alpha'_i))$ as above. The argument to conclude is the same as above, except that we have to bound the cardinality of $\{B(x_i,2^{-p}):i\in S_p, \    \mu(B(x,2^{-p}))\le 2^{-p\beta_1}\}$ by $2^{p(\overline f_\mu^{\mathrm{LD}}(\infty)+\eta)}$. 

(2) The result for packing dimensions easily follows from the inclusions $\underline E(\mu,\alpha)\subset \bigcup_{\beta\ge \alpha} \uparrow F(\alpha,\beta)$ and $\overline E(\mu,\alpha)\subset \bigcup_{0\le \beta\le \alpha} \uparrow F(\beta,\alpha)$, and the previous estimates for $\dim_P F(\alpha,\beta)$. 

The upper bound for $\dim_H  \underline E(\mu,\alpha)$ is obtained by writing $\underline E(\mu,\alpha)=\bigcup_{ \beta\ge \alpha}  \uparrow\{x\in\supp(\mu): \underline d(\mu,x)=\alpha,\,  \overline d(\mu,x)\le \beta\}$. If $\alpha=\infty$, there is nothing to prove, for $\underline E(\mu,\infty)=E(\mu,\infty)$. Suppose $\alpha<\infty$. Then,  for each $\beta<\infty$, due to the estimates achieved to find an upper bound for $\dim _H E(\mu,\alpha,\beta)$, given $\eta>0$, for each $\alpha'\in [\alpha,\beta]$ one can fix $\epsilon(\alpha')>0$ such that  $\dim_H \{x\in\supp(\mu): \underline d(\mu,x)=\alpha,\, \alpha'-\epsilon(\alpha')\le \overline d(\mu,x)\le \alpha'+\epsilon(\alpha')\}\le\min\{f_\mu^{\mathrm{LD}}(\alpha), \overline f_\mu^{\mathrm{LD}}(\alpha'),  \overline f_\mu^{\mathrm{LD}}(\alpha,\alpha')\}+\eta=f_H(\alpha,\alpha')+\eta$. Since we can cover $[\alpha,\beta]$ by finitely many intervals of the form $[ \alpha'-\epsilon(\alpha'), \alpha'+\epsilon(\alpha')]$, $\alpha'\in[\alpha,\beta]$, we get $\dim_H \{x\in\supp(\mu): \underline d(\mu,x)=\alpha,\, \overline d(\mu,x)\le \beta\}\le \sup\{f_H(\alpha,\alpha'): \alpha'\in[\alpha,\beta]\}+\eta$ for any $\eta>0$, hence $\dim_H \{x\in\supp(\mu): \underline d(\mu,x)=\alpha,\, \overline d(\mu,x)\le \beta\}\le \sup\{f_H(\alpha,\alpha'): \alpha'\in[\alpha,\beta]\}$. Since we also know that $\dim_H E(\mu,\alpha,\infty)\le f_H(\alpha,\infty)$, writing $\underline E(\mu,\alpha)=E(\mu,\alpha,\infty)\bigcup \bigcup_{\infty> \beta\ge \alpha}  \uparrow\{x\in\supp(\mu): \underline d(\mu,x)=\alpha,\,  \overline d(\mu,x)\le \beta\}$ and using the previous estimates for the Hausdorff dimensions yields $\dim_H  \underline E(\mu,\alpha)\le \sup\{f_H(\alpha,\beta): \beta\ge \alpha\}$. 

The upper bound for $\dim_H  \overline E(\mu,\alpha)$ is obtained by using similar arguments.

\subsection{Proof of some inequalities in \eqref{MuFo1} and \eqref{MuFo2}} \label{proof34}We justify the  inequalities $f^H_\mu(\alpha)\le {\underline{f}}^{\rm LD}_\mu(\alpha)\le  \overline\tau_\mu^*(\alpha)$ and  $\dim_P  E(\mu,\alpha)\le \overline{f}^{\rm LD}_\mu(\alpha)$ for $\alpha\in\R_+\cup\{\infty\}$, and $ \overline{f}^{\rm LD}_\mu(\infty)\le \tau_\mu^*(\infty)$.

Let $\alpha\in \R_+\cup\{\infty\}$. Suppose that $E(\mu,\alpha)\neq\emptyset$. The inequality $f^H_\mu(\alpha)\le\underline{ f}^{\rm LD}_\mu(\alpha)$ is a special case of the upper bound established for $\dim_H F(\alpha,\beta)$ in Section~\ref{proof33}.

Similarly, the inequality $\dim_P  E(\mu,\alpha)\le \overline{ f}^{\rm LD}_\mu(\alpha)$ follows from lines similar to those used to bound $\dim_PE(\mu,\alpha,\beta)$ in Section~\ref{proof33}. Also, due to  by Proposition~\ref{condtau}(2), if $\mathrm{dom} (\tau_\mu)=\R$, then one has  $\overline{ f}^{\rm LD}_\mu(\infty)=\tau_\mu^*(\infty)=-\infty$; otherwise $\overline \dim_B (\supp(\mu))=-\tau_\mu(0)=\tau_\mu^*(\infty)$, and by definition we have $\overline{ f}^{\rm LD}_\mu(\infty)\le \overline \dim_B (\supp(\mu))$. In any case,  $\dim_P  E(\mu,\infty)\le \overline{ f}^{\rm LD}_\mu(\infty)\le \tau_\mu^*(\infty)$.

To prove that $ \underline f^{\mathrm{LD}}_\mu(\alpha)\le \overline\tau_\mu^*(\alpha)$, we assume without loss of generality that  $\underline f^{\mathrm{LD}}_\mu(\alpha)>-\infty$, hence $\underline f^{\mathrm{LD}}_\mu(\alpha)\ge 0$. The case $\alpha=\infty$ then occurs only if $\overline \tau_\mu(q)=-\infty$ if $q<0$ (same proof as when $\tau_\mu(q)=-\infty$ for $q<0$); then it is direct that $ \underline f^{\mathrm{LD}}_\mu(\infty)\le -\overline\tau_\mu(0)=\overline\tau_\mu^*(\infty)$. Suppose now that $\alpha<\infty$.  It is  enough to prove that $\overline\tau_\mu (q)\le (\underline{ f}^{\rm LD}_\mu)^*(q)$ for all $q\in \R$. Then the result follows by taking the Legendre-transform and using (with $h= \underline{ f}^{\rm LD}_\mu)$ the general inequality  $(h^*)^*\ge h$ valid for any function whose domain is not empty.  Fix $q\in\R$, $\beta\in \R$ and $\epsilon>0$. If $\{B(x_i,r)\}$ is a packing of $\supp(\mu)$ by disjoint balls, we have 
$$
\sum_{i}\mu(B(x_i,r))^q\ge(\#\{i: r^{\beta+\epsilon}\le \mu(B(x_i,r))\le r^{\beta-\epsilon}\})\cdot 
 \begin{cases} 
 r^{q(\beta-\epsilon)}&\text{if }q\ge 0\\
  r^{q(\beta+\epsilon)}&\text{otherwise}
  \end{cases}.
  $$

Taking the supremum over the packings, dividing by $\log(r)$,  and taking the $\limsup$ as $r\to 0^+$ yields $\overline\tau_\mu(q)\le q(\beta\mp\epsilon) -\liminf_{r\to 0^+} f(\beta,\epsilon,r)$, and taking the limit as $\epsilon\to 0^+$ gives $\overline\tau_\mu(q)\le q\beta- \underline f^{\mathrm{LD}}_\mu(\beta)$ for all $\beta$, hence $\overline\tau_\mu(q)\le (\underline f^{\mathrm{LD}}_\mu)^*(q)$. 

\section{Dimensions of measures and mass distribution principle}\label{MDP}Given $\nu\in\mathcal M_c^+(\R^d)$, if  $\underline d(\nu,x)$ (resp. $\overline d(\nu,x)$) takes the same value $\underline D$ (resp. $\overline D$)  at $\nu$-almost every $x$, then $\dim_H \nu$ (resp. $\dim_P\nu$) stands for the Hausdorff  (resp. packing) dimension of the measure $\nu$, defined as the number $\underline D$ (resp. $\overline D$). Then, $\nu(E)>0$ implies $\dim_H E\ge \underline D$ (resp.  $\dim_P E\ge \overline D$) for any Borel set $E$. This is the mass distribution principle we use to get lower bounds for the Hausdorff and packing dimensions of the  level sets studied in this paper (about mass distribution principle and dimensions of measures, see the accounts proposed in  \cite[Section 4.2]{Falcbook}, \cite[Ch. 2]{Pesin} and \cite{He2} (other possible references being \cite[Section 14]{Bil}, \cite{You82}, \cite{Cut0,Cut} and in connection with multifractal formalism \cite{He1}, \cite{Ngai} and \cite{Ol00}).

\end{document}